\documentclass[a4paper,11pt,twoside,reqno]{amsart}
\usepackage{amsmath}
\usepackage{amssymb}
\usepackage{amsfonts}
\usepackage{amsthm}
\usepackage{bm}
\usepackage[colorlinks=true]{hyperref}
\usepackage{mathrsfs}
\usepackage{upref}
\usepackage{dsfont}
\usepackage{mathtools}
\usepackage[a4paper,bindingoffset=0cm,left=3.2cm,right=3.2cm,top=3.6cm,bottom=5.2cm,footskip=0cm]{geometry}
\usepackage{tikz-cd}
\newcommand*{\mailto}[1]{\href{mailto:#1}{\nolinkurl{#1}}}

\usepackage{esint}

\allowdisplaybreaks

\usepackage{color}

\definecolor{darkgreen}{rgb}{0.5,0.25,0}
\definecolor{darkblue}{rgb}{0,0,1}
\definecolor{answerblue}{rgb}{0,0,0.75}

\hypersetup{colorlinks,breaklinks,
            linkcolor=darkblue,urlcolor=darkblue,
            anchorcolor=darkblue,citecolor=darkblue}

\newcommand{\ep}{\varepsilon}
\newcommand{\eps}{\varepsilon}
\newcommand{\px}{\partial_x}

\newcommand{\pd}{\partial}
\newcommand{\LL}{\langle}
\newcommand{\RR}{\rangle}

\newcommand{\Ex}{\mathbb{E}}
\renewcommand{\d}{\mathrm{d}}

\newcommand{\sgn}{\mathrm{sgn}}
\newcommand{\supp}{\mathrm{supp}}

\newcommand{\loc}{\mathrm{loc}}
\newcommand{\Z}{\mathfrak{Z}}
 
\newcommand{\dott}{\, \cdot\,}
\newcommand{\KK}{\bm{\sigma}}
\newcommand{\JJ}{\mathbf{j}^\ep}
\newcommand{\XXX}{X}

\DeclareMathOperator*{\esssup}{ess\,sup}

\newcommand{\marginlabel}[1]
{\mbox{}\marginpar{\raggedleft\hspace{0pt}\tiny{\textcolor{red}{#1}}}
}

\newcommand{\R}{\mathbb{R}}

\newcommand{\pt}{\partial_t}

\newtheorem{theorem}{Theorem}[section]
\newtheorem{proposition}[theorem]{Proposition}
\newtheorem{lemma}[theorem]{Lemma}
\newtheorem{definition}[theorem]{Definition}

\newtheorem{corollary}[theorem]{Corollary}
\theoremstyle{definition}
\newtheorem{remark}[theorem]{Remark}

\numberwithin{equation}{section}     
\allowdisplaybreaks

\title[Hunter--Saxton Equation with Noise]
{The Hunter--Saxton Equation with Noise}

\author[H. Holden]{Helge Holden}
\address[H. Holden]{Department of Mathematical Sciences\\
  NTNU Norwegian University of Science and Technology\\
  NO-7491 Trondheim\\ Norway}
\email{\mailto{helge.holden@ntnu.no}}
\urladdr{\url{https://www.ntnu.edu/employees/holden}}

\author[K. H. Karlsen]{Kenneth H. Karlsen}
\address[K. H. Karlsen]{Department of Mathematics\\
   University of Oslo\\
  NO-0316 Oslo\\ Norway}
\email{\mailto{kennethk@math.uio.no}}

\author[P.H.C. Pang]{Peter H.C. Pang}
\address[P.H.C. Pang]{Department of Mathematical Sciences\\
  NTNU Norwegian University of Science and Technology\\
  NO-7491 Trondheim\\ Norway}
\email{\mailto{peter.pang@ntnu.no}}

\date{\today}
\thanks{This research was jointly and partially supported 
by the Research Council of Norway Toppforsk 
project {\em Waves and Nonlinear Phenomena} (250070) 
and the Research Council of Norway project 
{\em Stochastic Conservation Laws} (250674/F20).}
\begin{document}
\begin{abstract}
In this paper we develop an existence theory for the 
Cauchy problem to the stochastic Hunter--Saxton equation
 \eqref{eq:sHS1}, and prove several properties of 
 the blow-up of its solutions. An important part of
  the paper is the continuation of solutions to the
   stochastic equations beyond blow-up (wave-breaking).
    In the linear noise case, using the method of 
    (stochastic) characteristics, we also study 
    random wave-breaking and stochastic effects 
    unobserved in the deterministic problem. Notably,
     we derive an explicit law for the random 
     wave-breaking time.
\end{abstract}
\maketitle

\setcounter{tocdepth}{1}
\tableofcontents
\section{Introduction}

We consider the Hunter--Saxton equation \cite{MR1135995} with noise:
\begin{equation}\label{eq:sHS1}
\begin{aligned}
\pd_tq + \pd_x (u q) +\pd_x (\sigma q) \circ \dot{W} &= \frac{1}{2}q^2,\\
\pd_x u&= q.
\end{aligned}
\end{equation}
Here evolution occurs on $[0,T]\times \R$, and over
 the stochastic basis $(\Omega, \mathscr{F},\{\mathscr{F}_t\}_{t \ge 0}, \mathbb{P})$,
 the process $W$ is a standard one-dimensional 
 Brownian motion and $\circ$ denotes Stratonovich multiplication. We also point out that in this
 paper we ultimately limit ourselves to the assumption
 that $ \sigma = \sigma(x)$ is {\em linear}. 
 This assumption simplifies the analysis considerably,
 but still allows the equation  to manifest
 some stochastic effects. 
 The Cauchy problem is posed with an initial 
 condition $q|_{t=0}=q_0 \in L^1(\R) \cap L^2(\R)$.
 
 Other stochastic versions of the stochastic Hunter--Saxton equation
 exist, see \cite{MR3434476,MR3688413}, where the noise is introduced as a source term.

In the It\^o formulation the stochastic Hunter--Saxton equation reads:
\begin{align}
\pd_t q + \pd_x(u  q) + \pd_x (\sigma q) \dot{W} - \frac{1}{2} \pd_x(\sigma \pd_x(\sigma q) )&= \frac{1}{2} q^2.
\label{eq:ito_formulation}
\end{align}

The primary aim of this paper is to develop an existence
 theory for the stochastic Hunter--Saxton equation 
 under the assumptions above. 
 Our main theorem is Theorem~\ref{thm:main_existence}, stating that the equation \eqref{eq:sHS1} has both conservative and
 dissipative global solutions when $\sigma$ is linear. (The notions of conservative and dissipative solutions are discussed below.)
 
Our line of attack relies on the method of characteristics.
 Stochastic characteristics are used widely in the
  analysis of transport type equations in fluid dynamics
   and other applications (see \cite{MR2593276} and 
   \cite[Ch. 4]{MR2796837} and references there), 
   where corresponding deterministic dynamics are
    perturbed by introducing noise to the characteristics.
 As explained in Appendix \ref{sec:variation_hamiltonian}, the physical relevance of this noise derives
 from its being a perturbation on the 
associated Hamiltonian of the system, following a 
discussion in \cite{MR3488697} for stochastic soliton
 dynamics, so that the resulting equation follows 
from a variational principle applied to the 
stochastically perturbed Hamiltonian. 

The method of characteristics as applied to \eqref{eq:sHS1},
departs from the regime treated by \cite{MR2593276}, however, 
as the transport term depends on the solution. 
 This type of equation also falls outside the scope of the
 related investigation \cite{MR3927370}, which extended \cite{MR2593276}
 in their use of the kinetic formulation. 
 The non-locality  of the dynamics of \eqref{eq:sHS1}
  means that the transport term depends not only on the values of 
the solution at a point, but on the integral thereof,
 precluding a ``kinetic'' treatment of well-posedness.
 A substantial part of this work will be devoted to
 showing that the characteristics can be extended 
beyond a blow-up that inevitably happens, 
also in the deterministic case. 
This blow-up, termed ``wave-breaking'', 
is explained in Section \ref{sec:deterministic_background} below.

 It turns out that on properly defined characteristics,
 it is possible to derive explicit solutions. As we
 are employing characteristics and solving equations
 on characteristics, it is also imperative that we
 reconcile ``solutions-along-characteristics'' with
 solutions as usually defined, and which reduces to
 the familiar weak solutions \cite{MR1361013} in the 
deterministic case $\sigma = 0$. Relying on this 
explicit representation of solutions on characteristics,
 along the way we shall develop other aspects of 
the phenomenology for various solutions to these 
equations, including a connection between the 
distribution of blow-up times and exponential 
Brownian processes.

The organisation of this paper is as follows: In 
the remainder of this section, we describe the 
deterministic theory both to develop intuition about
 the dynamics of the stochastic Hunter--Saxton equation,
 and to give ourselves a template by which to 
understand corresponding features of the stochastic
 dynamics. Some pertinent calculations in the
 deterministic theory have been relegated to
 Appendix \ref{sec:appendixB}. Physical arguments 
behind our particular choice of the noise, which 
suggest that the case we consider is of 
physical relevance, are contained in Appendix \ref{sec:variation_hamiltonian}.

In the next section we give precise definitions of
 solutions, and state a-priori bounds. These bounds
 are proven in Appendix \ref{sec:aprioribounds}.
 In Section \ref{sec:lagrange}, we set up the method of
 characteristics framework used in subsequent 
sections. In particular, we show {\em how} the 
quantity $q$ experiences  finite-time blow-up in 
$L^\infty$. We also describe how this blowup in $q$
 is reflected by the behaviour of the evolution of
 its antiderivative, $u$. In Section \ref{sec:linear} we
 specialise to the case $\sigma'' \equiv 0$. We derive an explicit
 distribution for the wave-breaking stopping time in certain cases, and
 describe how characteristics behave up to the blow-up
 of $q$. In Section~\ref{sec:existence} we first describe
 strategies to continue characteristics and solutions
 beyond blow-up. We then prove global well-posedness
 of characteristics and well-posedness of solutions
 defined along characteristics, first on special 
initial data for clarity, before extending this 
 to general data in $L^1(\R) \cap L^2(\R)$ in 
Section \ref{sec:pbt}. Finally in Section \ref{sec:characteristics_sols},
 we reconcile various notions of solutions that we
 use in the article and show that the solutions 
defined along characteristics are included in more
 traditional partial differential equation-type (PDE-type) weak solutions. We postpone 
details of discussions on uniqueness and maximal 
dissipation that we shall mention in passing in Sections \ref{sec:setup}
 and \ref{sec:characteristics_sols} to  upcoming work.

\subsection{Background and the deterministic setting}\label{sec:deterministic_background}


We shall provide here a rough sketch of the 
deterministic theory of the Hunter--Saxton equation
 by which our intuitions are driven and against 
which our results can be benchmarked. We will focus on the analysis of the
characteristics following Dafermos \cite{MR2796054}.  Most of the 
material in this subsection can be found in classical
 papers by Hunter--Zheng \cite{MR1361013,MR1361014},
 and also in \cite{MR2182833}.


Solutions in the weak sense to the equations
\begin{equation}
\begin{aligned}
\pd_t q + u \pd_x q + \frac{1}{2}q^2 &= 0, \\
\pd_x u&= q,
\end{aligned}
\end{equation}
can be constructed quite explicitly by approximation
 with step functions. Approximating an initial 
function $q_0 \in L^2(\R)$ by 
\begin{align*}
q^n_0(x) =\sum_{-\infty}^\infty V^n_k \mathds{1}_{[k/n,(k+1)/n)}(x), \qquad
V^n_k = \fint_{k/n}^{(k+1)/n} q_0(x)\;\d x,
\end{align*}
we can confine our discussion to the ``box''-type 
initial condition ${q}_0 = V_0\mathds{1}_{[0,1)}$. This is true in spite of the equation being non-linear, see \cite{MR1361013}. Here $ \mathds{1}_{A}$ denotes the characteristic, or indicator, function of a set $A$, and $\fint_{A}$ denotes the average over a set $A$, i.e., $\fint_{A}\psi(x)\;\d x=\frac1{|A|} \int_{A}\psi(x)\;\d x$.

The equation with initial data ${q}_0$ is solved uniquely for at least a finite time by
\begin{equation*}
{q}(t,x) = \frac{2V_0}{2 + V_0 t}\mathds{1}_{\{2 + V_0 t > 0\}} \mathds{1}_{\{X(t,0) \le x < X(t,1)\}},
\end{equation*}
where $X(t,x)$ with $x \in [0,1)$ are the characteristics 
\begin{align}\label{eq:deterministic_characteristics}
X(t,x) = x +  \int_0^t {u}(s,X(s,x))\;\d s &= x + \int_0^t \int_0^{X(s,x)} {q}(s,y)\;\d y \,\d s \\
&= x+ \frac14(2 + V_0 t)^2, \notag
\end{align}
with ${u}$ being the function almost everywhere
 satisfying $\pd_ x{u} = {q}$, and the 
 final equality established by solving the linear 
 ordinary differential equation using the form of ${q}$ postulated.  A calculation gives
\begin{equation*}
{u}(t,x) =\mathds{1}_{\{2 + V_0 t > 0\}} \begin{cases}
0, & x\le \frac14(2+V_0t)^2, \\
\frac{2V_0x}{2 + V_0 t}, &\frac14(2+V_0t)^2<  x\le  1+\frac14(2+V_0t)^2, \\
\frac{2V_0}{2 + V_0 t}\big(1+\frac14(2+V_0t)^2 \big), &  x> 1+\frac14(2+V_0t)^2.
\end{cases}
\end{equation*} 

The general solution to the $n$th approximation 
can be recovered by summing up these ``boxes'' 
defined on disjoint intervals at every $t$, see \cite{MR1361013}.

From the above we see that where $V_0 \ge 0$,
 this solution exists uniquely and globally. 
 If $V_0 < 0$, however, there is a break-down time
 $t^*$ at which ${u}$ remains just absolutely continuous
 in the sense of the Lebesgue decomposition as it
 develops a steeper and steeper gradient over a 
 smaller and smaller interval around $x = 0$, and 
$\|{q}\|_{L^\infty}$ tends to infinity. This phenomenon,
 where $\|{u}\|_{L^\infty}$ remains bounded but 
$\|{q}\|_{L^\infty}=\|\pd_x{u}\|_{L^\infty} \to \infty$ is known
 as {\em wave-breaking}.

Up to wave-breaking, the energy $\|q(t)\|_{L^2}$ 
is conserved. This means that the characteristics 
$X(t,x)$ starting between $x = 0$ and $x = 1$ 
contract to a point. The failure of $X(t)$ in remaining
 a homeomorphism on $\R$ at wave-breaking leads 
to uncountably many possible ways of continuing 
solutions past wave-breaking, even under the requirement
 that $\|q(t)\|_{H^{-1}_\loc}$ remains continuous
 in time.


At the point of wave-breaking $q^2(t)$ passes from
 $L^1(\R)$ into a measure. We can think of this 
measure as a ``defect'' measure storing up the energy
 (or $L^2_x$-mass of $q$). It is possible to continue
 solutions in various ways past wave-breaking by 
releasing various amounts of this mass over various
 durations. The two extremes are generally termed 
``conservative'' and ``dissipative'' solutions \cite[p.~320]{MR1361013}.
Intermediates between these extremes 
when dissipation is not mandated everywhere, 
entirely, or eternally are also possible \cite{MR3860266}, 
as are more non-physical solutions exhibiting 
spontaneous energy generation. We relegate 
calculations showing this defect measure to 
Appendix \ref{sec:appendixB}.

Conservative solutions are constructed by releasing 
all the mass stored in the defect measure instantaneously 
after wave-breaking. That is, noticing that  
the formula for ${q}$ (less the characteristic 
function $\mathds{1}_{\{2 + V^1_0 t > 0\}} $) returns 
to a bounded function of the same --- conserved 
--- $L^2(\R)$-mass immediately post wave-breaking, 
and continues to satisfy the equation weakly, it 
is accepted that the formula defines a reasonable 
notion of solution. In particular:
\begin{equation}
\label{eq:dtm_H-1_continuity}
\begin{aligned}
{q} &\in L^\infty([0,T];L^2(\R)) \cap \mathrm{Lip}([0,T];H^{-1}_\loc(\R)), \\
{u} &\in C([0,T]\times \R),\\
0&=\pd_t ({q}^2) + \pd_x({u}{q}) \mbox{ in the sense of distributions}.
\end{aligned}
\end{equation}

Dissipative solutions arise when the ``defect measure'' 
stores up all mass eternally, and ${q}$ 
is simply set to nought after the wave-breaking 
time $t^*$. In this case the equations remain 
satisfied, and the previous inclusions remain valid, but
\begin{equation*}
0 \ge  \pd_t ({q}^2) 
	+ \pd_x({u}{q}) \mbox{ in the sense of distributions}.
\end{equation*}
reflecting the dissipation characterised by the 
defect measure.

These can be compared to continuation in the general 
stochastic setting, see Section~\ref{sec:continuations}.


We propose to approach the problem of well-posedness
 via the method of characteristics. 
As solutions are non-local, 
even though we have equations for 
characteristics $\d X(t,x)$ dependent on $u(t,X(t,x))$, 
and for $\d (q(t,X))$, there is no independent 
equation for $\d u(t,X(t,x))$. One of the aspects 
of this article is making sure that characteristics 
and functions constructed along them are 
defined without circularity, up to and beyond
wave-breaking, where non-uniqueness is necessarily
introduced into the problem.
Whilst our approach reduces to that of \cite{MR2796054}
in the deterministic case,
our analysis in the stochastic setting is 
complicated by the fact that at wave-breaking, where
a choice must be made as to the way that
characteristics should be continued beyond 
wave-breaking, the set of wave-breaking times
are dependent on the spatial variable $x$ and 
on the probability space. This means that 
wave-breaking occurs on a significantly 
more complicated set, and whereas in \cite{MR3451933,MR2796054,MR2852218}, for example,
translating between a wave-breaking time
and the set of initial points with characteristics 
leading up to a wave-breaking point at those times 
is a fairly straightforward affair, this operation
is much more delicate in the stochastic setting.
Even the measurability of wave-breaking
times in the filtration of the stochastic basis 
needs to be established in order to start a characteristic at wave-breaking
and match it up properly to characteristics leading up to
that wave-breaking time (on those particular sample-paths). 
Moreover the characteristics themselves are rough, 
and it is standard that there are correction terms compensating for
this roughness in evaluating functions on these characteristics. 
These issues compel us to set forth various notions of solutions
 to handle different aspects of the problem,
and then later to reconcile them. We shall do this in the next section.

\section{Solutions and a-priori estimates}\label{sec:setup}

\subsection{Definition of Solutions}\label{sec:pde_sols}
In this subsection we give definitions of different 
types of solutions and state our main theorem. 
As in the deterministic setting, there are two extreme 
notions of solution on which we shall focus. 
Whereas we have discussed how these arise in the 
deterministic setting both in Section \ref{sec:deterministic_background}
above (supplemented by Appendix \ref{sec:appendixB} below), 
we shall postpone the discussion regarding
 continuation beyond wave-breaking 
in the stochastic setting and the resultant non-uniqueness 
to Section \ref{sec:continuations}, after we have developed 
the theory sufficiently before and up to wave-breaking, 
with their supporting calculations.

We are working on a fixed stochastic basis 
\begin{equation}\label{eq:basis}
(\Omega, \mathscr{F},\{\mathscr{F}_t\}_{t \ge 0}, \mathbb{P})
\end{equation}
to which the process $W$ in \eqref{eq:sHS1} is 
adapted as a Brownian motion. Next we define weak 
solutions in the PDE sense in the usual way: 
Note that in Definition \ref{def:weak_sol}, we only consider time-independent test functions. 
\begin{definition}[Weak Solution]\label{def:weak_sol}
A weak solution to the stochastic Hunter--Saxton 
equation \eqref{eq:sHS1} with 
$\sigma \in (C^2\cap \dot{W}^{1,\infty}\cap \dot{W}^{2,\infty})(\R)$ 
and with initial condition $(u_0,q_0)$ where 
$q_0 \in L^1(\R)\cap L^2(\R)$ and $u_0$ are  related by 
\begin{equation*}
u_0(x) = \int_{-\infty}^x q_0(y)\;\d y,
\end{equation*}
is a pair $(u,q)$ of $\{\mathscr{F}_t\}_{t \ge 0}$-adapted 
processes, with  $u\in L^2(\Omega \times [0,T]; \dot{H}^1(\R))$ 
being absolutely continuous in $x$, and in 
$C([0,T]\times \R)\cap L^\infty([0,T];\dot{H}^1(\R))$, 
$\mathbb{P}$-almost surely, and $q \in L^2(\Omega \times [0,T]\times \R)$ 
and  in $C([0,T];H^{-1}_\loc(\R))\cap L^\infty([0,T];L^2(\R)))$, 
$\mathbb{P}$-almost surely. The solution $(u,q)$ satisfies, 
for any $\varphi \in C^\infty_0(\R)$ and for any $t \in [0,T]$,  $\mathbb{P}$-almost surely,
\begin{align}
0&=\int \varphi q\;\d x\bigg|_0^t - \int_0^t  \int \big(\pd_x  \varphi\, u q 
+ \frac{1}{2}\varphi q^2\big)\;\d x \,\d s 
- \int_0^t \int \pd_x \varphi\, \sigma q \,\d x \circ \d W,\label{eq:weak_form}\\
q &= \pd_x u \quad \text{in $L^2([0,T]\times \R)$}. \notag
\end{align}
In addition, we require that $\mathbb{P}$-almost surely, $\lim_{r \to -\infty }u(r) = 0$.
\end{definition}

\begin{remark}[The It\^o formulation of the noise]\label{rem:strat_to_ito}
Using the defintion of a weak solution (Def.\ref{def:weak_sol}), we have the temporal integrability to ensure that
the stochastic integral of \eqref{eq:weak_form} is a martingale.

From the definition of the Stratonovich integral we have
\begin{equation}\label{eq:S_to_I}
\int_0^t  \int \pd_x \varphi\, \sigma q \;\d x\circ \d W
 = \int_0^t \int \pd_x \varphi\, \sigma q \;\d x \,\d W  
+ \frac{1}{2}\int_0^t  \;\d\left\LL \int \sigma q \,\pd_x \varphi \;\d x, \;W\right\RR_s.
\end{equation}
Consider now  
$\psi = \sigma \pd_x \varphi$ as a time-independent test function in \eqref{eq:weak_form} ($\sigma$ is assumed to be at least once 
continuously differentiable), we find, $\mathbb{P}$-almost surely, that
\begin{align*}
\int (\sigma \pd_x \varphi\, q)(t,\dott) \;\d x& = \int \psi q \;\d x \bigg|_{t=0} +\int_0^t  \int \big(\pd_x  \psi\, u q 
+ \frac{1}{2}\psi q^2\big)\;\d x \,\d s\\  
& \quad+ \int_0^t \int \sigma q\, \pd_x \psi \;\d x \circ \d W  \\
&= \int \psi q \;\d x \bigg|_{t=0}+\int_0^t  \int \big(\pd_x  \psi\, u q 
+ \frac{1}{2}\psi q^2\big)\;\d x \,\d s   \\
&\quad + \int_0^t \int \sigma q \,\pd_x \big(\sigma \pd_x \varphi \big)\;\d x \, \d W +\frac{1}{2}\int_0^t 
\;\d\left\LL \int \sigma q \,\pd_x \psi \;\d x, \;W\right\RR_s.
\end{align*}
As all terms on the right-hand side except for the stochastic integral, are of
finite variation, we also have
\begin{align*}
\int_0^t  \;\d\left\LL \int \sigma q\, \pd_x \varphi \;\d x, \;W\right\RR_s&= 
\int_0^t  \;\d\left\LL  \int_0^{(\dott)} \int \sigma q \,\pd_x \psi \;\d x \, \d W, \;W\right\RR_s \\
&=  \int_0^t  \int \sigma q \,\pd_x \big(\sigma \pd_x \varphi \big)\;\d x \,\d s .
 \end{align*}
 Inserting this is in \eqref{eq:S_to_I}, we find
 \begin{equation}\label{eq:S_to_I_final}
\int_0^t  \int \pd_x \varphi \,\sigma q \;\d x\circ \d W
 = \int_0^t \int \pd_x \varphi \sigma q \;\d x \,\d W  
+ \frac{1}{2} \int_0^t  \int \sigma q\, \pd_x \big(\sigma \pd_x \varphi \big)\;\d x \,\d s.
\end{equation}

We can put this directly 
back into \eqref{eq:weak_form} and conclude that 
the weak solution as given can also be understood 
as a weak formulation of the It\^o equation 
\eqref{eq:ito_formulation}:
\begin{align*}
\pd_t q + \pd_x (u q) + \pd_x (\sigma q) \dot{W} 
- \frac{1}{2} \pd_x(\sigma \pd_x(\sigma q) )
&= \frac{1}{2} q^2.
\end{align*}
\end{remark}

Weak solutions are non-unique, a fact that shall 
be further expounded upon in Section~\ref{sec:continuations}. 
We can refine Definition~\ref{def:weak_sol} by concentrating 
on two types with additional properties as in the 
deterministic setting:
\begin{definition}[Conservative Weak Solutions]\label{def:cons_sols}
A conservative weak solution is a weak solution of 
\eqref{eq:sHS1} satisfying the energy equality 
\begin{align}
	\pt q^2+ \pd_x \Big(  \big(u-\frac14\pd_x\sigma^2 \big) q^2\Big)
		+ \pd_x \left( \sigma q^2\right) \dot{W}
		+& \pd_x \sigma q^2 \dot{W}
		-\frac12 \pd_{xx}^2\left(\sigma^2 q^2\right)\notag\\
				 &= q^2 \Big(\big(\pd_x\sigma\big)^2
		-\frac14 \pd_{xx}^2\sigma^2\Big),\label{eq:sol_cons}
\end{align}
in the sense of distributions on $[0,\infty)\times \R$,
 $\mathbb{P}$-almost surely.
\end{definition}
\begin{remark}
Equation \eqref{eq:sol_cons} is derived in
 Appendix \ref{sec:aprioribounds},  
 for $S \in W^{2,\infty}(\R)$. 
 Taking $S = S_\ell(q_\ep) = q_\ep^2 \wedge (2 \ell |q| - \ell^2)$ 
 for a mollified solution $q_\ep$, and taking 
 $\ep \to 0$ before $\ell \to \infty$ (when $S_\infty(q) = q^2$), 
 the conservation in the definition above follows from \eqref{eq:entropy}. 
 The full calculation can be found in 
 Lemma \ref{thm:mollifying_errors} and the proof of 
 Prop. \ref{thm:aprioriestimates} (also housed in 
 Appendix \ref{sec:aprioribounds}).
\end{remark}

\begin{remark}[Energy conservation identity]\label{rem:energy_balance_cons}
We shall prove in Theorem \ref{thm:gen_cons} that in 
the case $\sigma'' = 0$,  conservative weak solutions
 that are also solutions-along-characteristics (Def. \ref{def:sol_on_characteristics})
also satisfy the energy identity that, $\mathbb{P}$-almost surely,
\begin{align}\label{eq:max_energ_cons1}
\int_\R q^2(t,x)\;\d x = \int_\R q_0^2(x) \exp(-\sigma' W(t)) \;\d x.
\end{align}
In particular, for a deterministic initial value 
$q_0\in L^2(\R)$,
\begin{align}
\Ex \int_\R q^2(t,x)\;\d x 
&= \Ex \int_\R q_0^2(x) \exp(-\sigma' W(t)) \;\d x \notag\\
& = \iint_{\R^2} q^2_0(x) \exp(-\sigma' y)\; \gamma_t(\d y)\,\d x 
= \|q_0\|_{L^2(\R)} e^{(\sigma')^2t/4},
\end{align}
where $\gamma_t$ is the one-dimensional Gaussian measure at $t$.

This shows both that $q \in L^\infty([0,T];L^2(\R))$, 
$\mathbb{P}$-almost surely, and, in fact, also the 
additional integrability information in $\omega$, namely 
 that $q \in L^\infty([0,T];L^2(\Omega \times \R))$. 
This inclusion holds for more general noise 
(see Proposition \ref{thm:aprioriestimates}).
\end{remark}

\begin{definition}[Dissipative Weak Solutions]\label{def:weak_diss_sols}
A dissipative weak solution is a weak solution 
of \eqref{eq:sHS1} satisfying the condition that 
$q(t,x)$ is almost surely bounded from above on 
every compact subset of $(0,\infty)\times \R$, i.e., 
on every compact $E \subseteq (0,\infty) \times \R$, 
for $\mathbb{P}$-almost every $\omega$ there exists 
$M_{\omega,E} < \infty $ such that $q(t,x) < M_{\omega,E}$ 
for any $(t,x) \in E$,  in particular, $M$ is allowed 
to depend on $\omega$.
\end{definition}

 \begin{remark}[Energy dissipation identity and maximal energy dissipation]\label{rem:energy_balance_dissp}
 We shall show in Prop. \ref{thm:aprioriestimates} 
 that weak dissipative solutions also satisfy the 
 energy inequality
\begin{align}\label{eq:sol_dissp}
\pt q^2+ \pd_x \Big(  \big(u-\frac14\pd_x\sigma^2 \big) q^2\Big)
		+ \pd_x \left( \sigma q^2\right) \dot{W}
		+& \pd_x \sigma q^2 \dot{W}
		-\frac12 \pd_{xx}^2\left(\sigma^2 q^2\right)\notag\\
		 &\le q^2 \Big(\left(\pd_x\sigma\right)^2
		-\frac14 \pd_{xx}^2\sigma^2\Big),
\end{align}
in the sense of distributions (when integrated 
against non-negative test functions) on 
$[0,\infty)\times \R$, $\mathbb{P}$-almost surely.
 
 Defining the random variable $t^*_x$ parameterised 
 by every $x \in \R$ that is a Lebesgue point of 
 $q_0$ via the equation
\begin{equation}\label{eq:t_break}
-q_0(x) \int_0^{t^*_x} \exp\big(-\sigma' W( s)\big)\;\d s  = 2,
\end{equation}
or set $t^*_x = \infty$ if this is equality never holds. 
 In the case $\sigma'' = 0$, we shall prove 
 additionally in Theorem \ref{thm:gen_dissp} that 
 $\mathbb{P}$-almost surely, dissipative weak 
 solutions  that are also solutions-along-characteristics 
 (Def. \ref{def:sol_on_characteristics}) satisfy the energy identity
\begin{align}\label{eq:max_energ_dissp}
\int_\R q^2(t,x)\;\d x 
= \int_\R q_0^2(x) \exp(-\sigma' W(t)) \mathds{1}_{\{t \le t^*_x\}}\;\d x.
\end{align}
This formula similarly shows that a dissipative 
weak solution solution in the $\sigma'' = 0$ case 
is in $L^\infty([0,T]; L^2(\Omega \times \R))$ 
as the integrand on the right is non-negative 
and cannot be greater than \eqref{eq:max_energ_cons1} 
(again, see Proposition \ref{thm:aprioriestimates} for 
a more general statement).

It was shown in Cie\'slak--Jamar\'oz \cite{MR3451933} 
that this final requirement, in the deterministic setting, 
is implied by an Oleinik-type bound from above on 
$q$, and equivalent to a maximal energy dissipation 
admissibility criterion \`a la Dafermos \cite{MR328368,MR2796054,MR2852218}. 
The energy (in)equality is derived as part of the 
$L^2$-estimate worked out in the next subsection. 
As we also mention at the end of the paper, we 
shall show in an upcoming work that maximal 
energy dissipation is given by \eqref{eq:max_energ_dissp}, 
as well as uniqueness of these (maximally) 
dissipative solutions.

Taking $\sigma \equiv 0$, we recover the well-known 
conservative and dissipative solutions, respectively, 
of \cite{MR1361013}.
\end{remark}

The main aim of this paper is to establish the following theorem:
\begin{theorem}\label{thm:main_existence}
There exists conservative and dissipative weak 
solutions to the stochastic Hunter--Saxton equation 
\eqref{eq:sHS1} with $\sigma$ for which 
$\sigma'' = 0$ and $q_0 \in L^1(\R) \cap L^2(\R)$.
\end{theorem}

As we shall be working on characteristics, in 
Section \ref{sec:solving_on_characteristics} below 
we adopt yet another notion of solutions. 

\begin{definition}[Solution-along-characteristics]\label{def:sol_on_characteristics}
On the stochastic basis  \eqref{eq:basis}, an $\{\mathscr{F}_t\}$-adapted 
process $Q\in L^2(\Omega \times [0,T] \times \R)$ and $Q\in C([0,T]; H^{-1}_\loc(\R)) \cap L^\infty([0,T];L^2(\R))$, 
$\mathbb{P}$-almost surely, is a solution-along-characteristics
 to the stochastic Hunter--Saxton equation \eqref{eq:sHS1} 
 if there exists an $\{\mathscr{F}_t\}$-adapted process 
 $U \in L^2(\Omega \times [0,T] ; \dot{H}^1(\R))$ and 
 in $C([0,T]\times \R)$, $\mathbb{P}$-almost surely,
  for which the following stochastic differential equations (SDEs) are satisfied strongly 
in the probabilistic sense and a.e.~on $[0,T]\times \R$:
\begin{align}
Q(t,x) :&= \pd_x U(t,x),\notag\\
 Q(t,\XXX(t,x)) &= q_0(x) 
 -  \frac{1}{2}\int_0^t Q^2(s,\XXX(s,x))\;\d s  
 -\int_0^t \sigma'(\XXX(s,x)) Q(s,\XXX(s,x))\circ \d W,\label{eq:QSDE_Lagrange}
\end{align}
where $q_0 \in L^1(\R) \cap L^2(\R)$, and where,
\begin{align*}
\XXX(t,x) = x + \int_0^t & U(s,\XXX(s,x))\,\d s 
+\int_0^t \sigma(\XXX(s,x)) \circ \d W(s).
\end{align*}
\end{definition}

\begin{remark}[Conservative and dissipative solutions-along-characteristics]
The solutions so defined are individualised into 
conservative and dissipative solutions-along-characteristics 
according to how $U(t,\XXX(t,x))$ 
(equivalently, $\XXX$) are extended past a 
(unique) wave-breaking time $t^*_x$ indexed by 
the initial point $x = \XXX(0,x)$,  cf.~Theorems~\ref{thm:gen_cons} 
and \ref{thm:gen_dissp}. 
We will in 
Section~\ref{sec:characteristics_sols} provide theorems showing that solutions-along-characteristics 
are weak solutions.
\end{remark}

As we shall see, the SDE \eqref{eq:QSDE_Lagrange} 
above is the Lagrangian formulation of the stochastic 
Hunter--Saxton equation \eqref{eq:sHS1}. In the linear 
case $\sigma'' = 0$ ($\sigma'$ is a constant) there is 
an explicit formula for the process 
$\mathfrak{Q} = \mathfrak{Q}(t,x)$ satisfying
\begin{equation*}
\d \mathfrak{Q} =  - \frac{1}{2}\mathfrak{Q}^2\;\d t  -\sigma'(X(t,x)) \mathfrak{Q}\circ \d W,
\end{equation*}
 as we shall demonstrate in 
 Section \ref{sec:solving_on_characteristics}. Importantly, 
 this SDE does not depends explicitly on 
 $t$ and $x$ (cf.~Remark \ref{rem:meaning_of_q(X)}).

This definition reflects our strategy of proof, 
which is to postulate a $U(t,x)$, and, using this 
function, define $Q(t,x) := \pd_x U(t,x)$ and the 
characteristics $\XXX(t,x)$ for which 
\begin{equation*}
\d \XXX(t,x) = U(t,\XXX(t,x))\;\d t + \sigma(\XXX(t,x))\circ \d W,
\end{equation*} 
and then show that $Q(t,\XXX(t,x))$ coincides with the 
explicit formula for the process $\mathfrak{Q}(t,x)$. 
A schematic diagram for our construction is as follows:

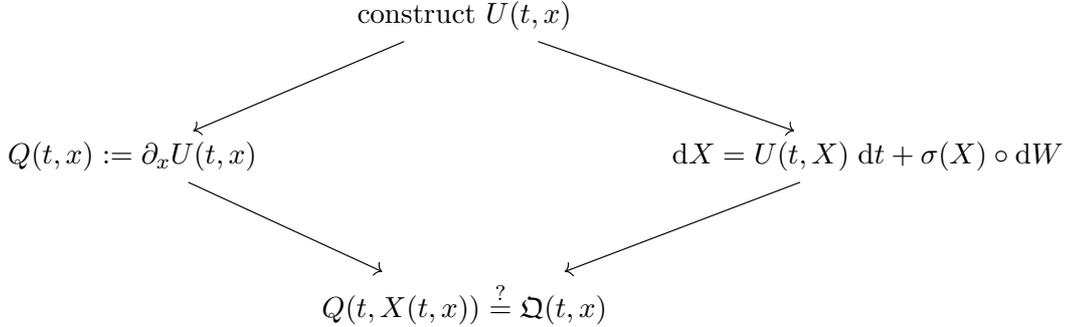
\begin{tikzcd}[column sep={5.2mm}, row sep=1.2cm]
&\mbox{construct }U(t,x) \arrow{dl}\arrow{dr}&\\
Q(t,x) := \pd_x U(t,x) \arrow{dr}&&\d \XXX = U(t,\XXX)\;\d t + \sigma(\XXX)\circ \d W\arrow{dl}\\
&Q(t,\XXX(t,x)) \overset{?}{=}\mathfrak{Q}(t,x) &
\end{tikzcd}

\subsection{A-priori bounds}
In the deterministic setting \cite[Section 4]{MR1361013} 
(see also \cite[Section 2.2.4]{MR2182833}, and references included there) 
it is known that weak conservative and dissipative 
solutions satisfy the following bounds:
\begin{align*}
\esssup_{t \in [0,T]}\| q(t)\|_{L^2(\R)} &\le \|q_0\|_{L^2(\R)},\\
\|q\|_{L^{2 + \alpha} ([0,T]\times \R)}^{2 + \alpha} 
&\le  C_{T,\alpha}\|q_0\|_{L^2(\R)}^2,
\end{align*}
for $t \in [0,T]$ and $0 \le \alpha < 1$.  
In the stochastic setting, the same types of bounds 
are generally available only in expectation. In fact, we have the following result. 

\begin{proposition}[A-priori bounds]\label{thm:aprioriestimates}
Let $q$ be a conservative or dissipative weak 
solution to the stochastic Hunter--Saxton equation \eqref{eq:sHS1}, 
with $\sigma \in (C^2 \cap \dot{W}^{1,\infty}\cap \dot{W}^{2,\infty})(\R)$, 
and initial condition $q(0) = q_0 \in L^1(\R) \cap L^2(\R)$. 
The following bounds hold:
\begin{align}
\esssup_{t \in [0,T]} \Ex\| q(t)\|_{L^2(\R)}^2 &\le C_T \|q_0\|_{L^2(\R)}^2, \label{eq:q_L2bound}\\
\Ex \|q\|_{L^{2 + \alpha} ([0,T]\times \R)}^{2 + \alpha} 
&\le  C_{T,\alpha} \|q_0\|_{L^2(\R)}^2\label{eq:q_Lalphabound},
\end{align}
for any $\alpha \in [0,1)$.
\end{proposition}

Therefore we  have 
\begin{equation*}
q \in L^\infty([0,T]; L^2(\Omega \times \R)) 
\cap L^{2 + \alpha}(\Omega \times [0,T]\times \R)
\end{equation*}
 for any $\alpha \in [0, 1)$. These bounds are not 
 expected to hold for general weak solutions, 
 because, as we shall see, spontaneous energy generation 
 (spontaneous increase in $L^2$-mass even in expectation) 
 in $q$ is permissible under Definition \ref{def:weak_sol}.

We shall prove this proposition using renormalisation 
techniques. Calculations can be found in Appendix 
 \ref{sec:aprioribounds}. More precisely, 
we have the $t$-almost everywhere bounds:
\begin{align}\label{eq:energy_consv_ap}
\Ex \int_\R |q|^2 \;\d x\bigg|_0^t  
&\le  \Ex \int_0^t  \int_\R q^2 \Big( (\pd_x \sigma)^2
 - \frac{1}{4} \pd_{xx}^2\sigma^2 \Big) \;\d x\,\d s
\end{align}
for $L^2_{\omega, x}$-control, and 
\begin{align}
\frac{1 - \alpha}{2} \Ex \int_0^t \int_\R |q|^{2 + \alpha}\;\d x\,\d s 
&\le \Ex\int_\R q|q|^\alpha \;\d x\bigg|_0^t
- \frac{\alpha (\alpha + 1)}{2}\Ex \int_0^t \int_\R q|q|^\alpha (\pd_x \sigma )^2\;\d x \,\d s\notag\\
&\qquad + \frac{\alpha}{4} \Ex \int_0^t \int_\R \pd_{xx}^2 \sigma^2 q|q|^\alpha\;\d x\,\d s
\label{eq:2plusalpha_prep}
\end{align}
for control in $L^{2 + \alpha}_{\omega,t,x}$, by interpolation.

Because of the first term on the right-hand side 
of \eqref{eq:2plusalpha_prep} and the use of 
interpolation/H\"older's inequality, and because 
we only have pointwise almost everywhere-in-time 
bounds for $\Ex\|q(t)\|_{L^p_x}$ with $p = 2$, we 
cannot extend these estimates past $\alpha <1$ 
(but see Remark \ref{rem:high_integ} regarding 
possible higher integrability as a manifestation 
of regularisation-by-noise).

\begin{remark}[Energy conservation]
With respect to \eqref{eq:energy_consv_ap}, 
the equation $(\pd_x \sigma)^2 = \pd_{xx}^2 \sigma^2/4$, 
which implies energy conservation, can be solved 
explicitly by $\sigma(x) = A e^{\pm x}$ or $\sigma(x) \equiv C$, 
the first of which does not satisfy 
our linearity assumption except with $A = 0$.

This is nevertheless a noise of particular interest 
as shown by Crisan and Holm \cite[Thm. 10]{MR3815211}. 
The related stochastic Camassa--Holm equation 
derived via a stochastic perturbation of the 
associated Hamiltonian can be understood as a 
compatibility condition for the deterministic 
Camassa--Holm isospectral problem and a stochastic 
evolution equation for its eigenvalue if the noise 
takes the form $\sigma(x) = Ae^x + Be^{-x} + C$ 
for $A,B,C \in \R$. (Note that there is a calculation 
error in (2.13) of \cite{MR3815211} that invalidates 
Theorem 16 there --- see also Remark \ref{rem:sigma_constant} 
below, and Section \ref{sec:asian_options} for genuinely 
stochastic wave-breaking.)
\end{remark}

\section{The Lagrangian Formulation and Method of Characteristics}\label{sec:lagrange}

\subsection{Solving \texorpdfstring{$q$}{q} on characteristics}\label{sec:solving_on_characteristics} 
Even though the Hunter--Saxton equation is not 
spatially local, in the deterministic setting, 
characteristics  
\begin{equation*}
\pd_t X(t,x) = {u}(t,X(t,x)) 
\end{equation*}
essentially fix the evolution of the 
equations because functions constant-in-space 
between two characteristics remain constant-in-space, 
and $\|q(t)\|_{L^2}$ is conserved up to wave-breaking 
(and also beyond --- this being {\it one} way to characterise 
continuation of solutions past wave-breaking). In 
the stochastic setting the behaviour between 
characteristics is more complicated and there 
is no conserved quantity. Nevertheless, taking 
cue from the classical construction of characteristics, 
much can still be deduced for solutions to the stochastic equations.

The ``characteristic equations'' from which the 
stochastic Hunter--Saxton equation arise are written 
with Stratonovich noise, as pointed out by \cite{ABT2019}:
\begin{align}\label{eq:characteristics}
X(t,x) = x + \int_0^t & u(s,X(s,x))\,\d s 
+\int_0^t \sigma(X(s,x)) \circ \d W(s).
\end{align}

Assuming that these characteristics are well-posed, via a general 
It\^o--Wentzell formula \cite{MR2800911}, since 
$q(t;\omega)$ takes values in $L^2(\R)$, one can 
derive from \eqref{eq:sHS1} the simpler (Lagrangian 
variables) equation:
\begin{align}
\d q(t,X(t)) &=  - \frac{1}{2}q^2(t,X(t))\;\d t  
-\sigma'(X(t)) q(t,X(t))\circ \d W.\label{eq:qX_sde}
\end{align}

As mentioned after Definition \ref{def:sol_on_characteristics} 
above, the SDE \eqref{eq:qX_sde} satisfied by $q(t,X(t))$
 (if suitably well-defined), can be written without reference
 to $x$ or to compositions of solution with characteristics as:
\begin{align}
\d \mathfrak{Q} = -\frac{1}{2} \mathfrak{Q}^2 \;\d t
 - \sigma' \mathfrak{Q}\circ \d W,\label{eq:qX_sde2}
\end{align}
and can in fact be solved explicitly without dependence 
on $X$, in the case $\sigma'' = 0$. We shall see this
 in \eqref{eq:q_on_characteristics} of Lemma \ref{thm:qX_behaviours}.

As in the previous section, since we are working 
presently on the assumption of well-posedness, in 
this section we do not restrict ourselves to $\sigma'' = 0$. 
We shall do so starting in Section \ref{sec:linear}.  
We postpone resolving the issue of the well-posedness 
of the characteristics equation \eqref{eq:qX_sde} to 
section Section \ref{sec:continuations}, but record here 
some properties of the composition $q(t,X(t,x))$ 
if it exists and is a strong solution of 
the SDE \eqref{eq:qX_sde}: 

\begin{lemma}\label{thm:qX_behaviours}
\begin{itemize}
\item[(i)] Assume that $X(t,x)$ is a collection 
of adapted processes with $\mathbb{P}$-almost surely 
continuous paths for each $x$ in the collection of 
Lebesgue points of $q_0$. Suppose that the composition 
$q(t,X(t,x))$ is a strong solution to the SDE \eqref{eq:qX_sde}
with $\sigma \in C^2(\R) \cap \dot{W}^{2,\infty}(\R)$
 (i.e., $u$ is $C^2$ with bounded second derivative), 
for each $x$ in the same set. Then $q(t,X(t,x))$ 
can be expressed by the formula
\begin{align}\label{eq:q_on_characteristics}
q(t,X(t,x))&= \frac{Z(t,x)}{\frac1{q_0(x)} 
+\frac12 \int_0^t Z(s,x) \;\d s},
\end{align}
where $Z(t,x) = \exp\big( - \int_0^t \sigma'(X(s,x)) \circ \d W\big)$, 
up to  the random time $t = t^*_x$ defined by 
\begin{align}\label{eq:blow-up_asianoption}
- \frac{1}{2}q_0(x)\int_0^{t^*_x}\exp\Big( - \int _0^s \sigma'(X(r,x)) \circ \d W(r)\Big) \;\d s = 1.
\end{align}

\item[(ii)] For $X$ as above assume further that 
$X(t)\colon \R \to \R $ is a homeomorphism of $\R$. 
If $q_0(x)$ can be written as a sum $q_1(0,x) + q_2(0,x)$
 of functions of disjoint support, then  
\begin{equation*}
q(t,x) = q_1(t,X(t,X(t)^{-1}(x))) +  q_2(t,X(t,X(t)^{-1}(x))),
\end{equation*}
and $q_1(t)$ and $q_2(t)$ have $\mathbb{P}$-almost surely 
disjoint supports.

\end{itemize}

\end{lemma}

\begin{remark}[Non-associativity of the Stratonovich product]\label{rem:non_assoc_stratonovich_prod}
Before we proceed to the proof we point out two obvious 
distinctions
\begin{itemize}
\item[(i)] $(\d q)(t,X(t))$ is not $\d(q(t,X(t)))$; 
these are related by the It\^o--Wentzell formula:
\begin{equation*}
\d(q(t,X(t))) = (\d q)(t,X(t)) + (\pd_xq)(t,X(t))\circ \d X;
\end{equation*}
to avoid the over-proliferation of parentheses, we take
$\d q(t,X(t))$ always to mean $\d (q(t,X(t)))$. 
\item[(ii)]Also, $(AB)\circ \d C$, for three processes 
$A$, $B$, and $C$ with finite quadratic variation, 
is not $A (B \circ \d C)$. The difference is
\begin{equation*}
(AB) \circ \d C - A (B \circ \d C) = \frac{1}{2} B \LL A, C \RR.
\end{equation*}
For notational convenience $AB\circ \d C$ will 
always denote $(AB) \circ \d C$, which, as especially 
pointed out in \cite[Lemma 3.1]{ABT2019}, 
is also equivalent to $A\circ (B \circ \d C)$.
\end{itemize}
\end{remark}

\begin{proof}
No requirements on linearity need be made here, but we 
remark after the end of this proof how formulas
derived simplify in an important way in this special case.

Using the change-of-variable $q(t,X(t)) \mapsto h(t) = 1/q(t,X(t))$
 reduces the above to a linear SDE in $h(t)$:
\begin{align*}
\d h =  \d \frac{1}{q(t,X(t))}&= \frac{-1}{q^2(t,X(t))} \circ \,\d q(t,X(t))\\
&= \frac{-1}{q^2(t,X(t))}\circ
 \bigg[ -\frac{1}{2} q^2(t,X(t))\,\d t -  \sigma'(X(t))q(t,X(t)) \circ \,\d W \bigg] \\
&= \frac{1}{2}\,\d t + \sigma'(X(t)) h(t) \circ\,\d W.
\end{align*}

From \cite[Eq. IV.4.51]{MR1214374}, the equation for 
$h(t)$, and hence for $q(t,X(t))$, can be solved 
explicitly, being the solution of the 
\emph{stochastic Verhulst equation}.
Setting
\begin{align}\label{eq:verhulst_integrating_factor}
Z(t) = Z(t,x) = \exp\bigg(-\int_0^t \sigma'(X(s,x)) \circ \d W\bigg), 
\end{align}
the linear equation for $h$ and $q(t,X(t))$ can be solved explicitly:
\begin{align*}
h(t)&= \frac{1}{Z(t)} \Big(h(0) +\frac12 \int_0^t Z(s)\;\d s\Big),
\end{align*}
because
\begin{align*}
\d \bigg[\frac{1}{Z(t)}\Big(h(0) & + \frac12\int_0^t Z(s) \;\d s\Big) \bigg] \\
&= \frac{1}{Z(t)} \circ \frac{1}{2}Z(t) \;\d t
-  \Big(h(0) +\frac12 \int_0^t Z(s) \;\d s\Big) \circ (\frac{1}{Z^2(t)} \circ \d Z(t))\\
&= \frac{1}{2} \;\d t-  \Big(h(0) + \frac12\int_0^t Z(s) \;\d s\Big) 
\frac{1}{Z(t)} \circ (-\sigma'(X(t)) \circ \d W)\\
&= \frac{1}{2} \;\d t + \sigma'(X(t)) h \circ \d W
\end{align*}
as sought. Here we used the rule $A \circ (B \circ \d C) = (AB) \circ \d C$ repeatedly.
 And consequently,
\begin{align*}
q(t,X(t,x))&= \frac{Z(t,x)}{\frac1{q_0(x)} +\frac12 \int_0^t Z(s,x) \;\d s},
\end{align*}
proving \eqref{eq:q_on_characteristics}.

Since $Z > 0$ everywhere, and $X(0,x) = x$, 
blow-up of $q(t,X(t,x))$ occurs at $t = t^*_x$ at which 
\begin{align}\label{eq:wave_breaking_time}
-\frac{1}{2}q_0(x)\int_0^{t^*_x} \exp\Big( - \int _0^s \sigma'(X(r,x)) \circ \d W(r)\Big) \;\d s = 1.
\end{align}

It is immediate that if $q_0(x) = 0$,  then $q(t,X(t,x)) = 0$. 
This implies that initial conditions with disjoint 
support give rise to solutions that have disjoint 
support, up to wave-breaking.

\end{proof}

\begin{remark}[Pathwise formulation for constant $\sigma$]\label{rem:sigma_constant}
 It is similarly immediate that if $\sigma' = 0$ 
 ($\sigma$ constant), then the blow-up time 
 coincides with that arising from deterministic 
 dynamics.  In fact, before we proceed to the next 
 section, we point out that the case $\sigma' = 0$ 
 is effectively the deterministic equations because 
 in a ``frame-of-reference'' given via a path-wise 
 transformation $x \mapsto x + \sigma W$, see 
 \cite[Prop. 2.6]{MR3927370} and \cite[Section 6.2]{MR2593276}, 
 then modulo measurability concerns,
\begin{align*}
U(t,x) = u(t,x + \sigma W(t)), \quad V(t,x) 
= q(t,x + \sigma W(t))
\end{align*}
solve the deterministic Hunter--Saxton equation
\begin{align*}
0&= \pd_t V + U\pd_x V + \frac{1}{2} V^2, \\
V&= \pd_x U,
\end{align*}
exactly when $q$ and $u$ solve \eqref{eq:sHS1} 
with constant $\sigma$. In fact, this is true for 
all equations of the form
\begin{equation*}
0 = \pd_t u + \mathcal{B}[u] + \sigma \pd_x u \circ \dot{W},
\end{equation*}
in which $\mathcal{B}$ is an integro-differential 
functional in the spatial variable 
(but not directly dependent on the same) as these 
operations are invariant in $x$-translations. 
See also Remark \ref{rem:alium}.
\end{remark}

\begin{remark}[The special case $\sigma'' = 0$]\label{rem:meaning_of_q(X)}
Referring to \eqref{eq:q_on_characteristics},
\eqref{eq:blow-up_asianoption}, 
and \eqref{eq:wave_breaking_time}, consider the case of linear 
$\sigma$. Since then $\sigma'$ is a constant,
we conclude that $q(t,X)$ and the wave-breaking 
time depend on $x$ only through $q_0$ --- 
{\em and not also cyclically through} $X(t,x)$, and 
in \eqref{eq:q_on_characteristics}, 
$Z(t,x) = \exp(-\sigma' W(t))$ is independent 
of $x$ altogether.

The expression \eqref{eq:q_on_characteristics} 
 can this case be written as 
\begin{align}\label{eq:q_on_char_lin_sigma}
\mathfrak{Q}(t,x) = \frac{e^{-\sigma' W(t)}}{\frac1{q_0(x)} + \frac12\int_0^t e^{-\sigma' W(s)}\;\d s }.
\end{align}

As mentioned after Definition \ref{def:sol_on_characteristics}, 
we shall define $\mathfrak{Q}(t,x)$ up to $t^*_x$ 
in subsequent discussions where $\sigma'' = 0$, as a 
family of processes indexed by $x$ by equation 
\eqref{eq:q_on_char_lin_sigma}, and  
{\it not as the composition} of some yet 
unknown $q(t,x)$ with a yet unknown $X(t,x)$ 
(that is, for example, the expression $q(t,X(s,x))$ 
has no meaning for us yet where $s \not=t$) .
\end{remark}

\begin{remark}[An application of the theory of Bessel processes/Ray--Knight theorems]

As an aside, we mention that it is possible to 
represent $\mathfrak{Q}$ as (a simple function of) a 
time-changed squared Bessel process of dimension 
$1$ when $\sigma''\equiv 0$ (that is, as the absolute value
of some Brownian motion $\tilde{W}$).

A result of Lamperti \cite{MR307358}, see also \cite[XI.1.28]{MR1725357},
showed that there exists a Bessel process 
$R^{(\nu)}$ of index $\nu$, i.e., of dimension $ d = 2 (\nu + 1)$,  
for which
\begin{equation*}
\exp(W(t) + \nu t) = R^{(\nu)}\Big(\int_0^t \exp\big(2(W(s) + \nu s)\big)\;\d s\Big).
\end{equation*}

By a slight modification of Lamperti's result, it 
can be shown that there exists a squared Bessel 
process $\Z^{(\delta)}(t)$ of dimension 
$d = 1 + 2c/(\sigma')^2$ for which 
\begin{align*}
\frac{2}{(\sigma')^2}\exp(-\sigma' W + c t) 
&= \Z^{(\delta)}(\LL M, M \RR(t)), \\
 M(t) &
= -\int_0^t \frac{1}{\sqrt{2}}
	\exp\Big(\frac12\big(-\sigma' W(s) + cs)\Big)\;\d W(s).
\end{align*}

We can see this as follows.
A squared Bessel process of dimension 
$d$ (starting at $\lambda$) satisfies:
\begin{equation*}
\Z^{(\delta)}(t)= \lambda + 2 \int_0^t \sqrt{\Z^{(\delta)}}\;\d B + \delta t.
\end{equation*}
Letting $B$ be the Brownian motion for which
\begin{equation*}
B({\LL M, M\RR(t)}) = M(t)
\end{equation*}
under the Dambis--Dubins--Schwarz theorem,
\begin{align}\label{eq:misc1}
\Z^{(\delta)}({\LL M, M \RR(t) }) 
= \lambda + 2\int_0^t \sqrt{\Z^{(\delta)}({\LL M, M \RR(s) })}\;\d M(s)
 + \delta {\LL M, M \RR(t) }.
\end{align}
Expanding ${\LL M, M \RR(t) } =\frac12 \int_0^t \exp(-\sigma' W(s) + cs)\;\d s$, 
we find that with 
\begin{equation*}
\lambda = \frac{2}{(\sigma')^2},\qquad \delta 
= \frac{2c}{(\sigma')^2} + 1,
\end{equation*}
the ansatz $Y(t)= \lambda \exp(-\sigma' W(t) + c t) $ 
satisfies the equation
\begin{equation*}
\d Y(t) = \frac{-2}{\sqrt{2}} \sqrt{Y(t)} \exp\big(\frac{-\sigma' W(t) + ct }{2}\big) \;\d W(t)
 + \frac{\delta}{2} \exp(-\sigma' W(t) + c t)\;\d t,
\end{equation*}
which is \eqref{eq:misc1} above with 
$Y(t) =\Z^{(\delta)}({\LL M, M \RR(t) })$.

Therefore choosing $c = 0$ above, there exists a 
squared Bessel process $\Z$ of dimension one
(the absolute value of a Brownian motion) for which
\begin{equation*}
\exp(-\sigma' W(t)) = \Z(\frac12\int_0^t \exp(-\sigma' W(s))\;\d s), 
\end{equation*}
and hence,
\begin{equation*}
q(t,X(t,x)) 
	= \frac{\Z(\frac12\int_0^t \exp(-\sigma' W(s))\;\d s)}{\frac1{q_0(x)}
		 +\frac12 \int_0^t \exp(-\sigma' W(s))\;\d s}.
\end{equation*}
\end{remark}

Finally we prove our main technical lemma, which 
will be useful in establishing well-posedness later. 
This lemma is important because it describes the 
main feature of wave-breaking --- that $u$ gets 
steeper and steeper as $q$ nears wave-breaking, 
but the jump is actually smaller and smaller, so 
that in the limit, around the point of wave-breaking, 
$u$ remains absolutely continuous, but 
$(\pd_x u)^2 = q^2$ passes into a measure.
\begin{lemma}[Absolute continuity of $u$ at wave-breaking]\label{thm:technical}
 Let $t^*_x$ be the wave-breaking time defined by 
 \eqref{eq:blow-up_asianoption} indexed by the 
 Lebesgue points $x$ of $q_0$. Assume that $X(t,x)$ 
 is a collection of adapted processes with 
 $\mathbb{P}$-almost surely continuous paths for 
 each $x$ in the collection of Lebesgue points of $q_0$. 
 Suppose that the composition $q(t,X(t,x))$ is a 
 strong solution to the SDE \eqref{eq:qX_sde} for 
 each $x$ in the same collection. Set
\begin{align}
\mathfrak{u}(t,x;\omega) 
&= \mathfrak{u}(t,x) \notag \\
&:= q(t,X(t,x))\exp\Big(\int_0^t q(s,X(s,x))\;\d s 
	+ \int_0^t \sigma'(X(s,x)) \circ\d W(s)\Big).\label{eq:pre_u}
\end{align}
It holds that for such $x \in \R$ as aforementioned,
\begin{equation*}
\mathbb{P}-a.s.,\qquad  \lim_{t \nearrow t^*_x} \mathfrak{u}(t,x) = 0.
\end{equation*}
\end{lemma}
\begin{remark} The quantity \eqref{eq:pre_u} ought 
to be thought of heuristically as 
$$
q(t,X(t,x))\frac{\pd X}{\pd x},
$$ 
and will be integrated in $x$
to construct a function $U(t,x)$, defined on 
characteristics (cf.~\eqref{eq:Psi_auxU}). The 
exponential is a $\mathbb{P}$-almost surely finite 
quantity up to blow-up because we assume that 
$\sigma'$ is bounded (and then constant in Section \ref{sec:linear}). 
Furthermore up to blow-up (if there is blow-up) there 
is always an upper bound on $q(t,X(t,x))$ depending 
on $q_0(x)$ and $\sigma'$. In the case $\sigma'' = 0$, 
we can define $\mathfrak{u}$ as a well-defined 
quantity with $\mathfrak{Q}(t,x)$ given by 
\eqref{eq:q_on_char_lin_sigma} in the place of 
$q(t,X(t,x))$, sans assumptions on $q$ and $X$, so 
that $\mathfrak{u}$ is expressible as
\begin{align}\label{eq:pre_u_lin_sigma}
\mathfrak{u}(t,x)
	 := \mathfrak{Q}(t,x)
	 	\exp\Big(\int_0^t \mathfrak{Q}(s,x)\;\d s + \sigma' W(t)\Big),
\end{align}
which, as we shall see in the proof, 
cf.~\eqref{eq:u_simple}, reduces to
\begin{align}\label{eq:u_simple_lin_sigma}
q_0(x)\Big(1 + \frac12 q_0(x) \int_0^t e^{-\sigma' W(s)} \;\d s\Big).
\end{align}
It is easily seen from the preceding formula that 
in the deterministic case, where the integral reduces 
further to $t/2$, we recover the linear term 
familiar in the deterministic theory.
\end{remark}

\begin{proof} 
Let $Z(t,x) = \exp(-\int_0^t\sigma'(X(s,x))\circ \d W(s))$. 
Using the expression \eqref{eq:q_on_characteristics}, 
we have
\begin{align}
 \mathfrak{u}(t,x) 
 &= q(t,X(t,x)) \exp\bigg(\int_0^t q(s,X(s,x))\;\d s 
 	+\int_0^t \sigma'(X(s,x))\circ\d W(s)\bigg)\notag\\
 &= \frac{Z(t,x)}{\frac1{q_0(x)} + \frac12\int_0^t Z(s,x)\;\d s} \notag\\
  &\qquad\times\exp\Big(\int_0^t \frac{Z(s,x)}{\frac1{q_0(x)} + \frac12\int_0^s Z(r,x)\;\d r} \;\d s 
  	 +\int_0^t \sigma'(X(s,x))\circ\d W(s)\Big)\notag\\
  &=  \frac{Z(t,x)}{\frac1{q_0(x)} + \frac12\int_0^t Z(s,x)\;\d s} \notag\\
  & \quad\times\exp\Big(2 \int_0^t \frac{\d }{\d s} \log\big(-\frac1{q_0(x)} 
  		-\frac12 \int_0^s Z(r,x)\;\d r\big)\;\d s  +\int_0^t \sigma'(X(s,x))\circ\d W(s)\Big)\notag\\
  &=  \frac{Z(t,x)}{\frac1{q_0(x)} + \frac12\int_0^t Z(s,x)\;\d s} 
  	\Big(-1 - \frac12q_0(x)\int_0^t Z(s,x)\;\d s\Big)^2 
  		e^{ \int_0^t \sigma'(X(s,x))\circ\d W(s)}\notag\\
  &= Z(t,x) \exp\big( \int_0^t \sigma'(X(s,x))\circ\d W(s)\big)
  	 q_0(x) \Big(1+ \frac12q_0(x)\int_0^t Z(s,x)\;\d s\Big)\notag\\
	 &=  q_0(x) \Big(1+ \frac12q_0(x)\int_0^t Z(s,x)\;\d s\Big)
	 \label{eq:u_simple}.
\end{align}
By the definition of $t^*_x$ given in \eqref{eq:wave_breaking_time}, 
this quantity vanishes exactly at $t = t^*_x$.

\end{proof}

Although the result derived above holds for general 
$\sigma \in W^{1,2}$, we emphasize again that 
whenever $\sigma'$ is a constant, $Z(t,x)$ only 
depends on $x$ through $q_0$.  In the case $\sigma'$ 
is constant, a closer look at \eqref{eq:verhulst_integrating_factor} 
and \eqref{eq:q_on_characteristics} confirms that 
$Z(t,x)$ is independent of $x$, so if $q_0$ is 
constant over an interval $I\subseteq \R$, then 
for $x,y \in I$, until the blow-up time,
\begin{align}\label{eq:injective_on_stepfunctions}
\mathfrak{Q}(t,x) = \mathfrak{Q}(t,y),
\end{align}
just as in the deterministic setting. Therefore 
the point of the Lemma \ref{thm:technical} is that 
where we start with $q_0 = V_0\,\mathds{1}_{x \in [0,1]}$,  we have
$\mathfrak{Q}(t,x) = \mathfrak{Q}(t,\frac12)$ for $x \in [0,1]$, 
and $\mathfrak{u}(t,x)$ should be a constant 
multiple of the value of $u(t,x)$. We next explore 
finer properties concerning blow-up time.

\section{Wave-Breaking Behaviour}\label{sec:linear}

\subsection{Explicit calculation of the law of wave-breaking time using exponential Brownian motion}\label{sec:asian_options}

In this section we provide an expression for the 
distribution of the blow-up time $t^*_x$ defined 
in \eqref{eq:blow-up_asianoption}, under the 
condition that $\sigma'' = 0$, from which we are 
also assured of its measurability. This is of independent 
interest as it describes the (random) time of 
wave-breaking precisely.

Where $\sigma'$ is a constant, the blow-up condition 
\eqref{eq:blow-up_asianoption} simplifies to
\begin{equation*}
-\frac{1}{2}q_0(x)\int_0^{t^*_x} \exp\left( - \sigma' W(s)\right) \;\d s = 1.
\end{equation*}
Exponential Brownian functionals such as the one 
above have been studied in detail by Yor \cite{MR1854494} 
and others (see also the surveys \cite{MR2203675, MR2203676}).
The distribution for the blow-up can be explicitly 
computed:

Let
\begin{align}
A(t) &:=\frac{1}{2}\int_0^t \exp\left( - \sigma' W(s)\right) \;\d s,\notag\\
A^{(\mu)}(t) &:=\int_0^t \exp( 2 \mu s + 2 W(s) )\;\d s \label{eq:A_mu}.
\end{align}

In \cite[Theorem 4.1]{MR2203675} (originally derived 
in another form in \cite{MR1174378}) it was shown that 
\begin{align}\label{eq:misc2}
\mathbb{P}(A^{(\mu)}(t) \in \d \chi) 
= \frac{\d \chi}{\chi}\int_\R e^{\mu r - \mu^2 t/2} 
	\exp\Big(-\frac{1 + e^{2r}}{2\chi}\Big) \vartheta(e^r/\chi,t) \;\d r,
\end{align}
where the integral is taken against $\d x$, and 
\begin{equation*}
\vartheta(y,t) = \frac{y}{\sqrt{2 \pi^3 t}} 
	e^{\pi^2/(2t)}\int_0^\infty e^{-\xi^2/(2t)}
		e^{-y \cosh(\xi)}\sinh(\xi) \sin\bigg(\frac{\pi \xi}{t}\bigg)\;\d \xi.
\end{equation*}

We shall apply the explicit formula for the 
distribution of $A^{(\mu)}$ to give a similarly 
explicit formula for the distribution of the blow-up 
time $t^*_x$.

\begin{proposition}\label{thm:asianop_distb}
Let $t^*_x$ be defined as in \eqref{eq:blow-up_asianoption},
 and let $A^{(\mu)}$ be defined as in \eqref{eq:A_mu}. Then
\begin{align}\label{eq:blowuptime_distribution}
 \mathbb{P}(\{t^*_x \ge t\}) 
 =\mathbb{P}\bigg(\bigg\{A^{(0)}\bigg(\frac{(\sigma')^2t}{4}\bigg) 
 \le \frac{ -(\sigma')^2}{2q_0(x)}\bigg\}\bigg).
\end{align}
\end{proposition}

\begin{proof}

In the following we use ``$\sim$'' to denote 
equality in law under $\mathbb{P}$.

We can use the scaling invariance of Brownian motion 
to show that
\begin{align}\label{eq:laws_scaling}
A(t) \sim \frac{2}{(\sigma')^2} A^{(0)}\bigg(\frac{(\sigma')^2t}{4}\bigg),
\end{align}
which gives us the distribution of $A(t)$ explicitly:
\begin{align*}
A^{(0)}(t) = \int_0^t \exp(2 W(s))\;\d s 
&= \int_0^{4t/(\sigma')^2} \exp\Big(-\sigma'\frac{-2}{\sigma'} W(\tau/(2/\sigma')^2) \Big) \;\d \frac{\tau}{4/(\sigma')^2}\\
&\sim  \frac{(\sigma')^2}{4}\int_0^{4t/(\sigma')^2}\exp(-\sigma' \tilde{W}(\tau)) \;\d {\tau}\\
&= \frac{(\sigma')^2}{2} A\big(4t/(\sigma')^2\big).
\end{align*}
Here $\tilde{W}$ is another standard Brownian motion, 
by the scaling invariance of the process.

We know that $A(0) = 0$ because it is an integral 
of a continuous process. It is also an increasing 
process because the integrand is positive. This 
implies that the supremum process 
$A^*(t) = \sup_{s \le t} A(s)$ is simply $A(t)$. 
Finally, $-\frac1{q_0(x)} > 0$. Therefore,
\begin{align*}
 \mathbb{P}(\{t^*_x \ge t\}) =\mathbb{P}(\{A(t) \le -\frac1{q_0(x)}\}).
\end{align*}
\end{proof}

\begin{remark}[Consistency in the limit 
$\sigma' \to 0$.]\label{rem:alium}
With regards to Remark \ref{rem:sigma_constant}, 
it is instructive to see that if $(\sigma')^2/4$ is 
treated as a parameter and taken to nought, then of course
\begin{equation*}
A(t) = \frac{1}{2}t,
\end{equation*}
or alternatively,
\begin{align*}
\lim_{(\sigma')^2/4 \to 0} \frac{2}{(\sigma')^2} 
	A^{(0)}\bigg(\frac{(\sigma')^2t}{4}\bigg) 
	= \frac{1}{2} \lim_{c \to 0} \frac{1}{c}\int_0^{ct} \exp(2W(s))\;\d s 
		= \frac{1}{2}t,
\end{align*}
by the Lebesgue differentiation theorem, and this 
matches the deterministic dynamics of wave-breaking 
exactly. This again verifies that the $\sigma' = 0$ 
setting cannot result in random blow-up.
\end{remark}

\subsection{Meeting time of characteristics}
We turn our attention now to the characteristics 
themselves, described by \eqref{eq:characteristics} 
and reproduced below:
\begin{align*}
X(t,x)
	 = x + \int_0^t & u(s,X(s,x))\,\d s 
	 	+\int_0^t \sigma(X(s,x)) \circ \d W(s).
\end{align*}

Consider again the explicit ``box'' example with initial condition
\begin{align}\label{eq:classic_box}
q_0 = V_0\, \mathds{1}_{[0,1]} \le 0, \qquad V_0 \in (-\infty,0).
\end{align}

We seek to prove that in the case $\sigma'' = 0$, 
wave-breaking only occurs when characteristics meet, 
and when characteristics meet, wave-breaking occurs. 
This allows us later to use characteristics to capture 
precisely the behaviour of wave-breaking.

As mentioned after \eqref{eq:injective_on_stepfunctions}, 
in the case of  ``box'' initial conditions \eqref{eq:classic_box}, 
by \eqref{eq:q_on_char_lin_sigma} and reproduced here:
\begin{align*}
\mathfrak{Q}(t,x) 
	 = \frac{e^{-\sigma' W(t)}}{\frac1{q_0(x)} + \frac12\int_0^t e^{-\sigma' W(s)}\;\d s },
\end{align*}
we see from the dependence on $x$ only via $q_0(x)$ 
that $\mathfrak{Q}$ is piecewise constant over $x$. 
In particular, this means $\mathfrak{Q}(t,x) = \mathfrak{Q}(t,\frac12)$ 
over $x \in (0,1)$. 

We shall show that it is possible to construct a 
function $U(t,x)$ from this information, and 
characteristics from $U(t,x)$ in the next section. 
For now we assume that characteristics as defined 
by $\d X = U(t,X) \;\d t + \sigma(X) \circ \d W$ 
exist 
and that $(\pd_x U)(t,X(t))$ --- the composition 
of $(\pd_x U)$ with a characteristic at the same time 
--- is equal to the process $\mathfrak{Q}$ above.
We shall establish this existence in 
Section \ref{sec:well_posedness1} below.

\begin{proposition}[Characteristics meet at wave-breaking]\label{thm:char_wave_breaking}
Let $\sigma'' = 0$, and suppose $X(t,x)$ is a strong solution
to the equation \eqref{eq:characteristics}, 
for which $(\pd_x u)(s,X(t,x)) =\mathfrak{Q}(t,x)$
for each $x \in \R$,  with $q_0 = V_0 \mathds{1}_{[0,1]}$.
Then the first meeting time of any two 
characteristics $X(t,x)$ and $X(t,y)$,
\begin{align*}
\tau_{x,y} := \inf\{t > 0  :    X(t,x) = X(t,y)\}, \qquad \quad x,y\in [0,1],
\end{align*}
 is $\mathbb{P}$-almost surely equal to the 
 wave-breaking time $t^*_{1/2}$ defined by 
 \eqref{eq:blow-up_asianoption}.
\end{proposition}

\begin{remark}
In particular, the explicit formula for the distribution 
of the meeting time of characteristics is also 
given by \eqref{eq:blowuptime_distribution}. 
In the case $\sigma'' \not \equiv 0$, we cannot 
immediately extract an explicit form for $u$ and 
thereby one for $X$ as in \cite{ABT2019}, 
because of nonlocality. 
\end{remark}

\begin{proof}

Recall that in the linear case, $\mathfrak{Q}$ is 
given via \eqref{eq:q_on_char_lin_sigma} as the process 
\begin{align*}
\mathfrak{Q}(t,x) 
	= \frac{e^{-\sigma' W(t)}}{\frac1{q_0(x)} +\frac12 \int_0^t e^{-\sigma' W(s)}\;\d s }.
\end{align*}
If $x,y \in [0,1]$, then
\begin{align}\label{eq:u_on_characteristicsI}
\frac{u(s,X(s,x)) - u(s,X(s,y))}{X(s,x) - X(s,y)} 
= \mathfrak{Q}(s,x)  = \mathfrak{Q}(s,y) 
	= \mathfrak{Q}(s,\frac12),
\end{align}
and similarly,
\begin{equation*}
\frac{\sigma(X(s,x)) - \sigma(X(s,y))}{X(s,x) - X(s,y)} 
		= \sigma',
\end{equation*}
as both $q$ and $\sigma'$ are constant in space 
over the interval $[X(s,0),X(s,1)]$. This leads us to
\begin{align*}
X(t,x) - X(t,y) &= (x - y) 
	+ \int_0^t \mathfrak{Q}(s,\frac12)\,(X(s,x) - X(s,y))\;\d s\\
			&\quad + \sigma' \int_0^t (X(s,x) - X(s,y))\circ\d W(s), 
\end{align*}
for $x,y \in [0,1]$. This is eminently solvable:
\begin{align}\label{eq:characteristics_difference}
X(t,x) - X(t,y)  
	= (x - y)\exp\Big( \int_0^t \mathfrak{Q}(s,\frac12) \;\d s + \sigma' W(t)\Big).
\end{align}

Since $\|q(t)\|_{L^2}^2$ is $\mathbb{P}$-almost 
surely bounded, the first meeting time $\tau_{0,1}$ 
cannot occur after the blow-up time $t^*_x$ of 
$\mathfrak{Q}(t,x)$ on the characteristic $X(t,x)$ 
(which, again, by \eqref{eq:injective_on_stepfunctions} 
is the same for any $x \in [0,1]$ --- we have chosen $x = \frac12$ for concreteness).
 The meeting time also cannot occur before the 
blow-up time, so that dissipation (instantaneous 
in the conservative case) cannot occur without wave-breaking. 

To see this it suffices to ask how the exponential in 
\eqref{eq:characteristics_difference} can possibly 
become nought --- it cannot become so before 
$\mathfrak{Q}(t,\frac12)$ blows up to $-\infty$.
\end{proof}

The fact that the exponential {\em does} become 
nought when this happens gives us a rate in time 
at which $ \mathfrak{Q}(s,x)$ blows up, which may 
otherwise have been difficult to extract from 
\eqref{eq:q_on_char_lin_sigma}.

\section{Existence of Solutions}\label{sec:existence}
\subsection{Solutions post wave-breaking: a discussion}\label{sec:continuations}

This subsection consists solely of a discussion on 
different ways characteristics, and solutions 
defined along them, can be continued past wave-breaking. 
We shall not limit ourselves to $\sigma'' = 0$. 
This is a question of cardinal importance because 
here as in the deterministic setting, non-uniqueness 
turns on there being various ways in which to 
continue solutions past wave-breaking. 
Accurately prescribing this continuation will allow 
us both to prove global existence of individual 
characteristics and thereby, on them, of $q$.

As noted following \eqref{eq:wave_breaking_time} 
in Lemma \ref{thm:qX_behaviours}, if $q_0(x) =0$, 
then along a characteristic starting at $x$, we expect $q(t,X(t,x)) = 0$. 
Therefore as in the deterministic setting, 
it should be possible to patch solutions together: 
That is, if $q_1(0)$, $q_2(0)$ are two $L^2(\R)$-valued 
random variables (or simply $L^2(\R)$ functions, 
if invariant over all but a measure zero set of $\Omega$) 
of compact and disjoint support on $\R$, then the 
solution $q$ with initial condition $q_0 = q_1(0) + q_2(0)$
 is simply $q(t) = q_1(t) + q_2(t)$. Furthermore, 
 from \eqref{eq:blow-up_asianoption}, the non-negativity 
 of the exponential function also shows that there 
 ought not to be blow-up along $X(t,x)$ if $q_0(x) \ge 0$. 
These heuristics imply that, as in the deterministic setting,
  ``box''-type initial conditions given \eqref{eq:classic_box}
 should retain special interest in the stochastic setting.

As discussed in Section \ref{sec:deterministic_background}
 there are two extreme ways by which solutions are 
 continued past wave-breaking. They give rise to 
 ``conservative'' and ``dissipative'' solutions.

In the deterministic setting, conservative solutions 
are constructed by simply extending the definition 
by explicit formulas to times $t > t^*_x$, as, in 
the example of the box, the explicit formula is 
undefined only at the point of wave-breaking, 
and reverts immediately to being well-defined 
thereafter. Seeing as $ t \mapsto \int_0^t \exp(-\sigma'(X(s,x)) W(s))\;\d s $ 
is $\mathbb{P}$-almost surely an increasing function 
in $t$ for each fixed $x$, simply allowing $q(t,X(t))$ 
to be be defined by \eqref{eq:q_on_characteristics} 
is similarly admissible in the stochastic setting 
(if the characteristics $X(t,x)$ are properly defined). 
Of course, continuity of $q(t)$ in suitable norms, 
and that of $X(t)$, requires proof. We also stress 
that there is no conservation of $L^2(\R)$ even 
in expectation in the general stochastic setting 
--- however, on taking $\sigma = 0$, we shall be able to recover 
the well-studied deterministic conservative solutions.

Alternatively, one can mandate dissipation by 
setting all concentrating $L^2(\R)$-mass to 
nought at wave-breaking. This is the 
``dissipative solution''.  In the stochastic setting 
(complete) dissipation can also be replicated, 
though this is again predicated on proofs of continuity, 
for example, of the ${H}^{-1}_\loc$ norms of $q$. 
Suppose all characteristics $X(t,z)$ for $z \in [x,y]$ 
meet at the stopping time $t^*_z$. This is a stopping 
time by Prop. \ref{thm:asianop_distb}. Assuming 
$\sigma'$ is locally bounded, as we always do, by 
the standard existence and uniqueness theorem for 
SDEs, these can be continued as  
\begin{align}\label{eq:dissipative_characteristics}
\d X(t^*_x + t, X(t^*_x,x)) 
	= \sigma (X(t^*_x + t, X(t^*_x,x))) \circ \d \tilde{W},
\end{align}
where $\tilde{W}$ is the Brownian motion starting 
at $t^*_x$, at the initial point $W(t^*_x)$.

\subsection{Well-posedness for box initial data}\label{sec:well_posedness1}

We focus again on the $\sigma''  =0$ case. 
Here we use the  ``box''-type initial condition 
\eqref{eq:classic_box} to illustrate the derivation 
of well-posedness, and the chief aspects of the 
general well-posedness theorem will appear here. 
We shall extend these results to the general data 
case in Section \ref{sec:pbt}. In this subsection all 
solutions refer to conservative or dissipative 
solutions-along-characteristics.

Recall that by \eqref{eq:injective_on_stepfunctions}, 
for the case described by \eqref{eq:classic_box} 
the wave-breaking time $t^*_x$ defined in 
\eqref{eq:wave_breaking_time} is uniform in 
$x \in[0,1]$. Thus, we denote this time simply by $t^*$:
\begin{align}\label{eq:wave_breaking_time_box}
-\frac{1}{2}q_0(\frac12)\int_0^{t^*} \exp\Big( - \int _0^s \sigma' \circ \d W(r)\Big) \;\d s = 1.
\end{align}
 The result of Lemma \ref{thm:technical} then states 
 that $\mathbb{P}$-almost surely, as $t \to t^*$ 
 from below,
\begin{align}\label{eq:ufrak_lin_sigma}
\mathfrak{u}(t;\omega) = \mathfrak{u}(t) 
:= \mathfrak{Q}(t,\frac12) \exp\Big(\int_0^t \mathfrak{Q}(s,\frac12)\;\d s + \sigma' W(t)\Big) \to 0,
\end{align}
where $\mathfrak{Q}(s,x)$, given explicitly by 
\eqref{eq:q_on_char_lin_sigma}, is also uniform 
in $x \in [0,1]$ because it only depends on $x$ 
through the initial condition.

Next we proceed to the focus of this subsection 
--- to resolve the primary questions of existence 
and uniqueness concerning the characteristics 
defined in \eqref{eq:characteristics}, including 
the continuation of them past wave-breaking. 
This will in turn lead us to different ways of 
continuing $\mathfrak{Q}(t,x)$ (given by 
\eqref{eq:q_on_char_lin_sigma} in the ``box''-type
 initial data case) past wave-breaking.

Our plan of attack is as follows 
(cf.~diagram at the end of Section \ref{sec:pde_sols}):
\begin{itemize}
\item[(i)] Postulate a $U(t,x)$, and use it to 
find characteristics $\XXX(t,x)$ satisfying 
\begin{align}\label{eq:postulated_characteristics}
\XXX(t,x) = x +  \int_0^t U(s,\XXX(s,x))\;\d s 
+ \int_0^t \sigma(\XXX(s,x)) \circ \;\d W(s).
\end{align}
\item[(ii)] Show that for $Q(t,x) = \pd_x U(t,x)$, 
the process $Q(t,\XXX(t,x))$ agrees with 
$\mathfrak{Q}(t,x)$, $\mathbb{P}$-almost surely, 
up to $t = t^*$, and remains a strong solution to 
\eqref{eq:qX_sde}: 
\begin{equation*}
\d \tilde{Q}(t)= - \frac{1}{2} \tilde{Q}^2(t)\;\d t 
		+ \sigma' \tilde{Q}(t)\circ \d W.
\end{equation*}

\item[(iii)] Finally we extend $U$ and $Q$ past 
wave-breaking in ways that preserve their continuity 
pointwise and in $H^{-1}_\loc(\R)$, respectively.
\end{itemize}

Our goal in this subsection is to prove the following two theorems:
\begin{theorem}[Conservative Solutions: Box Initial Data]\label{thm:box_cons}
Suppose $\sigma'' = 0$ and $q_0 = V_0\,\mathds{1}_{[0,1]}$, 
$V_0 \in \R$. There exists a $U \in C([0,\infty)\times\mathbb{R})$,
 $\mathbb{P}$-almost surely, absolutely continuous in $x$, 
 such that for each $x \in \R$, the following SDE is globally well-posed:
\begin{align*}
\XXX(t,x) 
	= x +  \int_0^t U(s,\XXX(s,x))\;\d s 
		+ \int_0^t \sigma(\XXX(s,x)) \circ \;\d W(s).
\end{align*}

For $Q(t,x) = \pd_x U(t,x)$, the process 
$Q(t,\XXX(t,x))$ agrees $\mathbb{P}$-almost surely 
with $\mathfrak{Q}(t,x)$, defined in \eqref{eq:q_on_char_lin_sigma}, 
up to $t = t^*$ and can be represented globally as
\begin{equation*}
Q(t,\XXX(t,x)) 
	= \frac{\exp(-\sigma' W(t))}{\frac1{V_0} + \frac12\int_0^t \exp(-\sigma' W(s))\;\d s}  \mathds{1}_{[0,1]}(x).
\end{equation*}
We have $Q(0,x)=q_0(x)$.
In particular, $\tilde{Q}(t)= Q(t,\XXX(t,x))$ 
satisfies \eqref{eq:qX_sde2} strongly and globally: 
\begin{equation*}
\d \tilde{Q}(t)= - \frac{1}{2} \tilde{Q}^2(t)\;\d t 
+ \sigma' \tilde{Q}(t)\circ \d W.
\end{equation*}
\end{theorem}

Similarly, for the dissipative solutions-along-characteristics, we have:
\begin{theorem}[Dissipative Solutions: Box Initial Data]\label{thm:box_dissp}
Suppose $\sigma'' = 0$ and $q_0 = V_0\,\mathds{1}_{[0,1]}$, 
$V_0 \in \R$. There exists a $U \in C((0,\infty)\times \mathbb{R})$, 
$\mathbb{P}$-almost surely, absolutely continuous in $x$, 
such that for each $x \in \R$, the SDE
\begin{align*}
\XXX(t,x) = x +  \int_0^t U(s,\XXX(s,x))\;\d s +
 		\int_0^t \sigma(\XXX(s,x)) \circ \;\d W(s)
\end{align*}
 is globally well-posed.

For $Q(t,x) = \pd_x U(t,x)$, the process $Q(t,\XXX(t,x))$ 
agrees $\mathbb{P}$-almost surely with $\mathfrak{Q}(t,x)$ 
as given by \eqref{eq:q_on_char_lin_sigma}, 
up to $t = t^*$ and can be represented globally in time as
\begin{align}\label{eq:box_dissp_sol}
Q(t,\XXX(t,x)) 
	= \begin{cases}\frac{\exp(-\sigma' W(t))}{\frac1{V_0} + \frac12\int_0^t \exp(-\sigma' W(s))\;\d s} 
		\mathds{1}_{[0,1]}(x), & t < t^*,\\
		0,&t > t^*.\end{cases}
\end{align}
We have $Q(0,x)=q_0(x)$.
In particular, $\tilde{Q}(t)= Q(t,\XXX(t,x))$ 
satisfies \eqref{eq:qX_sde2} strongly and globally (in time):
\begin{equation*}
\d \tilde{Q}(t)= - \frac{1}{2} \tilde{Q}^2(t)\;\d t + \sigma' \tilde{Q}(t)\circ \d W.
\end{equation*}
\end{theorem}

We relegate the computation of $H^{-1}_\loc$ to 
Section \ref{sec:pbt} where it is done in the general context 
(see also Remark \ref{rem:temp_cont_QH-1}). 
Theorems \ref{thm:box_cons} and \ref{thm:box_dissp} 
are proved in similar fashion and we shall present 
one in full and sketch out the other. In both of them 
the bulk of the work rests on a proper construction 
of $U$. Obviously in both proofs we shall be making 
heavy use of \eqref{eq:q_on_char_lin_sigma} and 
on our main technical result, Lemma \ref{thm:technical}.

For dissipative solutions we can also show the 
one-sided Oleinik-type estimate (cf.~discussion 
following Definition \ref{def:weak_diss_sols}):
\begin{corollary} [One-Sided Estimate: Box Initial Data]\label{thm:one_sided_estm}
Suppose $\sigma'' = 0$ and $q_0= V_0\,\mathds{1}_{[0,1]}$, $V_0 \in \R$. 
Then the dissipative solution $Q(t,x)$ with initial 
condition $Q(0) = q_0$ satisfies $\mathbb{P}$-almost 
surely the following one-sided bound:
\begin{align*}
Q(t,x) \le \frac{\exp(-\sigma' W(t))}{\frac1{\max({V_0},0)} +\frac12 \int_0^t \exp(-\sigma' W(s))\;\d s}.
\end{align*}
\end{corollary}
Because of \cite[Theorem 4.1]{MR2203675}, 
the law of the right-hand side is known.

\medskip
We now present the proofs of the above theorems, starting with the conservative case.
\begin{proof}[Proof of Theorem \ref{thm:box_cons}]
We divide the proof into two parts:
\begin{itemize}
\item[(1)] We postulate $U$ and construct globally 
(in time) extant characteristics $\XXX(t,x)$.
\item[(2)] We show that $(\pd_x U)(t,\XXX(t,x))$ 
satisfies \eqref{eq:qX_sde2}.
\end{itemize}

\smallskip
{\em 1. Construction of $U$ and global characteristics.}

\smallskip
Using \eqref{eq:q_on_char_lin_sigma}, 
$\mathfrak{Q}(t,x)$ is constant over $x \in [0,1]$ 
for time up to $t = t^*_{1/2}$ 
($= t^*_0 = t^*_1$ by this constancy). 
Therefore we simply construct $U(t, \dott)$ to be 
the piecewise linear function taking the value 
$U(t,x) = 0$ for  $x < \XXX(t,0)$ and 
$U(t,x) = \mathfrak{Q}(t,\frac12) (\XXX(t,1) - \XXX(t,0))$
 for $ x > \XXX(t,1)$.
(Because $U(t)$ is piecewise linear by construction, 
$Q(t)$ will be constant between $\XXX(t,0)$ and 
$\XXX(t,1)$.) This definition can be extended 
to all times $t \ge 0$ by taking $\mathfrak{Q}(s,\frac12)$ 
in the definition of $\mathfrak{u}$ 
(cf.~\eqref{eq:ufrak_lin_sigma}) to mean:
\begin{align*}
\mathfrak{Q}(t,\frac12) 
	= \frac{\exp(-\sigma' W(t))}{\frac1{V_0} + \frac12\int_0^t \exp(-\sigma' W(s))\;\d s}.
\end{align*}

The only difficulty is that $U$ so defined 
depends on $\XXX(t,0)$ and $\XXX(t,1)$ 
in a circular fashion. To rectify this circularity, 
we take one more step back and define characteristics 
$\XXX(t,0)$ and $\XXX(t,1)$, 
which will later self-evidently be solutions to 
\eqref{eq:postulated_characteristics} at 
$x = 0$ and $x = 1$.

For $x \le  0$ or $x \ge 1$, set
\begin{align}\label{eq:pre_tilde_X}
\XXX(t,x) 
= x + \mathds{1}_{\{x \ge 1\}} \int_0^t  \mathfrak{u}(s) \;\d s 
	+ \int_0^t \sigma(\XXX(s,x)) \circ \;\d W(s),
\end{align}
which has a global unique strong solution 
in the space of adapted process with $\mathbb{P}$-almost surely 
continuous paths by the basic theorem 
on well-posedness of SDEs (see, e.g., \cite[Thm. IX.II.2.4]{MR1725357}), and by the boundedness of 
$\mathfrak{u}$ ensured by the formula 
\eqref{eq:u_simple_lin_sigma}. The function $\mathfrak{u}$ here
has been defined explicitly in \eqref{eq:u_simple_lin_sigma}.

We now postulate the ansatz $U(t,x)$ for $u(t,x)$:
\begin{align}\label{eq:U_construction}
U(t,x) 
	= \begin{cases}0, & x \le \XXX(t,0),\\
		\frac{x - \XXX(t,0)}{\XXX(t,1) - \XXX(t,0)} \mathfrak{u}(t),
			 &x \in (\XXX(t,0),\XXX(t,1)),\\
			\mathfrak{u}(t), & x \ge \XXX(t,1),\end{cases}
\end{align}
where $U$ is defined pointwise in $(t,x)$,
 $\mathbb{P}$-almost surely.

In the $\sigma'' = 0$ case, $\mathfrak{u}(t)$ 
(given in \eqref{eq:pre_u_lin_sigma}) 
does {\em not} depend on any characteristic.

Now we define $\XXX(t,x)$ by the equation
\begin{align} \label{eq:chararcteristics_tilde}
\XXX(t,x) 
	= x +  \int_0^t U(s,\XXX(s,x))\;\d s 
		+ \int_0^t \sigma(\XXX(s,x))\circ\d W.
\end{align}
(We re-use the symbol $\XXX$ from above as 
this equation simply augments equation \eqref{eq:pre_tilde_X}.) 
By taking a spatial derivative, we see that this 
SDE also has an explicit solution: for $x \in [0,1]$, 
$t < t^*$,
\begin{align*}
\frac{\partial \XXX(t,x)}{\partial x} 
= e^{\sigma' W(t)} \Big(1 
	+ \int_0^t \frac{e^{-\sigma' W(s)}}{\XXX(s,1) 
		- \XXX(s,0)}\mathfrak{u}(s) \;\d s\Big),
\end{align*}
and consequently,
\begin{align*}
\XXX(t,x) = \XXX(t,0) 
	+ x e^{\sigma' W(t)}\Big(1 
	+ \int_0^t \frac{e^{-\sigma' W(s)}}{\XXX(s,1) 
		- \XXX(s,0)}\mathfrak{u}(s) \;\d s\Big) 
			\qquad x \in [0,1],\, t < t^*.
\end{align*}

Again, by direct differentiation of the equation 
above, we can see that the derivative 
$\pd \XXX/\pd x$ is independent of $x$,
\begin{align}\label{eq:dxdx21}
\frac{\pd \XXX(t,x)}{\pd x} &= \XXX(t,1) - \XXX(t,0) .
\end{align}
It is also signed, since alternatively to 
\eqref{eq:u_simple} we also have
\begin{align}
\mathfrak{u}(s) 
&= \mathfrak{Q}(s,\frac12) 
	\exp\Big(\int_0^s \mathfrak{Q}(r,\frac12)\;\d r + \sigma'W(s)\Big)\notag\\
 &= e^{\sigma' W(s)}\frac{\d}{\d s} 
 	\exp\Big(\int_0^s \mathfrak{Q}(r,\frac12)\;\d r\Big),\label{eq:ufrak_alternative}
\end{align}
so solving the SDE for $\XXX(t,1) - \XXX(t,0)$,
\begin{align}
\XXX(t,1) - \XXX(t,0)  
&= e^{\sigma' W(t)} \Big( 1 + \int_0^t e^{-\sigma' W(s)} \mathfrak{u}(s)\;\d s\Big) \notag\\
&=  \exp\Big(\int_0^t \mathfrak{Q}(s,\frac12)\;\d s  + \sigma' W(t)\Big) \ge  0\label{eq:X1-X0},
\end{align}
with strict inequality except at $t = t^*$.

We record the fact that characteristics do not 
cross except at wave-breaking as a lemma, see Lemma \ref{thm:stoch_dafermos_box} after this proof. 

The global well-posedness for the end-point 
characteristics $\XXX(t,0)$ and $\XXX(t,1)$, 
and  \eqref{eq:dxdx21}, allow us to extend 
$\XXX(t,x)$ globally (beyond $t^*$) via
\begin{align}\label{eq:fraction_constant}
\frac{\XXX(t,x) - \XXX(t,0)}{\XXX(t,1) - \XXX(t,0)}  = x.
\end{align}

\smallskip
{\em 2. Verifying properties of $\pd_x U$.}

Setting
\begin{equation*}
Q(t,x) = \pd_x U(t,x),
\end{equation*}
we shall proceed to show that up to $t = t^*$,
 $\mathbb{P}$-almost surely,
\begin{equation*}
Q(t,\XXX(t,x)) = \mathfrak{Q}(t,x),
\end{equation*} 
and that we have the (global) explicit formula:
\begin{equation*}
Q(t,\XXX(t,x)) 
	= \frac{\exp(-\sigma' W(t))}{\frac1{V_0} +\frac12 \int_0^t \exp(-\sigma' W(s))\;\d s} 
			\mathds{1}_{x \in [0,1]}.
\end{equation*}

By construction $U$ was built by integrating 
$\mathfrak{Q}(t,\frac12)$ in time. 
Using \eqref{eq:U_construction}, 
\eqref{eq:ufrak_lin_sigma}, and 
\eqref{eq:X1-X0} directly, it comes as no surprise that:
\begin{align*}
\pd_x U(t,x) 
 &= \begin{cases}0, & x \le \XXX(t,0),\\
 	\frac{\mathfrak{u}(t)}{\XXX(t,1) - \XXX(t,0)}, &x \in (\XXX(t,0),\XXX(t,1)),\\
 		0, & x \ge \XXX(t,1),\end{cases} \\
&=\begin{cases}0, & x \le \XXX(t,0),\\
	\mathfrak{Q}(t,\frac12), &x \in (\XXX(t,0),\XXX(t,1)),\\
		0, & x \ge \XXX(t,1),\end{cases}
\end{align*}
which, by differentiating directly, yields
\begin{equation*}
\d Q(t,\XXX(t,x)) 
	= -\frac{1}{2}Q^2(t,\XXX(t,x))\;\d t 
		+ \sigma' Q(t,\XXX(t,x))\circ \d W.
\end{equation*}

We emphasise once again that no conservation 
of any norms of $Q$ is proven or even claimed. 
\end{proof}

We state for clarity the following result, which simply re-establishes 
Prop. \ref{thm:char_wave_breaking} without the unproven assumption concerning 
the existence of characteristics.
\begin{lemma}[Stochastic Flow of Diffeomorphisms before Wave-breaking]\label{thm:stoch_dafermos_box}
Let $q_0 = V_0\,\mathds{1}_{[0,1]}$ be a 
``box''-type initial data. Let $\sigma'' = 0$ 
and $\{\XXX(t,x)\}_{x \in \R}$ be defined by 
\eqref{eq:U_construction}, \eqref{eq:chararcteristics_tilde}. 
Then up to $t^*$ defined by \eqref{eq:wave_breaking_time_box}, 
$\phi_t \colon x \mapsto \XXX(t,x)$ is a flow 
(i.e., a one-parameter semi-group in $t$) of 
diffeomorphisms of $\R$.

And for given $(t,x)$, $t \not= t^*$, there is a 
unique random variable $y:\Omega \to \R$ for 
which $\XXX(t,y) = x$.
\end{lemma}

\smallskip 
We now turn to the proof in the  dissipative case.
\begin{proof}[Proof of Theorem \ref{thm:box_dissp}]

First we notice that by construction and 
Lemma \ref{thm:technical}, at the wave-breaking 
time $t^*$, $U(t^*,\dott) \equiv 0$, $\mathbb{P}$-almost surely. 
Since we have unique paths up to $t^*$, the pair of equations
\begin{equation}\left\{
\begin{aligned}
\d \XXX(t,x)&= U(t,\XXX(t,x))\;\d t + \sigma(\XXX(t,x))\circ \d W(t),&&t^* > t \ge 0,\\
\d \XXX(t^*+t,x)&= \sigma(\XXX(t^*+t,x)) \circ \d W(t^* + t),&&t \ge 0,
\end{aligned}\right.
\end{equation}
gives unique global solutions $\XXX(t,x)$ 
for each $x$ that are continuous in $t$.
These equations represent stopping the characteristic 
at the time $t^*$, and then starting it again where
$U(X)$ becomes nought. Measurability is not an issue as
$W$ is strong Markov, and $t^*$ was shown to be a stopping time
in Section \ref{sec:asian_options}. Lemma \ref{thm:technical} 
in fact guarantees that $U(X)$ tends continuously
to zero at wave-breaking. 

In effect we have postulated a truncated $\tilde{U}(t,x)$ 
in place of $U$ in \eqref{eq:U_construction}, to wit:
\begin{align*}
\tilde{U}(t,x) = \begin{cases} U(t,x), & t < t^*,\\ 
0, &t \ge t^*,  \end{cases}
\end{align*}
and used the result of Lemma \ref{thm:technical}.

By defining
\begin{equation*}
Q(t,x) = \begin{cases}\pd_x U(t,x), & t < t^*,\\ 
0, & t \ge t^*. \end{cases}
\end{equation*}

It is clear that as in the previous proof, 
$Q(t,x)$ and $Q(t,\XXX(t,x))$ still satisfy
\begin{align*}
\d Q(t,\XXX(t,x))&= -\frac{1}{2} (Q(t,\XXX(t,x)))^2\;\d t - \sigma' Q(t,\XXX(t,x)) \circ \d W
\end{align*}
over $t < t^*$, and that this holds trivially thereafter, as sought.

\end{proof}

\begin{proof}[Proof of Corollary \ref{thm:one_sided_estm}]
This follows directly from \eqref{eq:q_on_char_lin_sigma}, and from \eqref{eq:box_dissp_sol} in Theorem \ref{thm:box_dissp}.

\end{proof}

\begin{remark}[Optimality of higher integrability for the case $\sigma'' = 0$]\label{rem:high_integ}
As we can extend solutions to and past wave-breaking, 
using \eqref{eq:q_on_char_lin_sigma}, 
\eqref{eq:u_simple}, and \eqref{eq:characteristics_difference} 
it is possible to compute $\|q(t)\|_{L^2}$ explicitly 
for the ``box''-type initial condition \eqref{eq:classic_box} 
{\em in the conservative case}, because $q(s)$, as in 
the deterministic case, does not vary over 
the interval $(X(s,0),X(s,1))$:
\begin{align*}
\|q(t)\|_{L^2}^2&= (X(t,0) - X(t,1)) \mathfrak{Q}^2(t,\frac12) \\
&= \mathfrak{Q}(t,\frac12)\,\mathfrak{u}(t,x)\\
&=\frac{Z(t,x)}{\frac1{q_0(x)} + \frac12\int_0^s Z(r,x)\;\d r}
	 \Big(\frac1{q_0(x)} + \frac12\int_0^s Z(r,x)\;\d r\Big)\\
&= Z(t,x) = \exp(-\sigma' W(t)).
\end{align*}

It may be hoped that if the distribution of $t^*$ is 
sufficiently dispersed, then at any deterministic time $t$, 
only a $\mathbb{P}$ measure zero set of paths 
experience wave-breaking and higher integrability 
beyond $L^{3-\ep}(\Omega \times [0,T]\times \R)$ 
proven in Prop. \ref{thm:aprioriestimates} may be achieved. 
This hope proves false, however, as we shall now show:

By the preservation of boxes under the flow of 
the equations in the case $\sigma'' \equiv 0$, 
\begin{align*}
\Ex \|q\|_{L^p([0,T]\times \R)}^p 
&=\Ex \int_0^T|\mathfrak{Q}(t,\frac12)|^p  (X(t,1) - X(t,0))  \;\d t \\
&=  \Ex \int_0^T|\mathfrak{Q}(t,\frac12)|^{p - 1}|\mathfrak{u}(t)| \;\d t .
\end{align*}

With $\mathfrak{Q}(t,x)$ again given by 
\eqref{eq:q_on_char_lin_sigma} and \eqref{eq:u_simple},
 we can simplify the integrand as follows:
\begin{align*}
|\mathfrak{Q}(t,\frac12)|^{p - 1}\mathfrak{u}(t) 
&= \frac{|Z(t)|^{p - 1}}{|\frac1{q_0(x)} 
	+\frac12\int_0^t Z(s,x)\;\d s|^{p - 1}}\Big(\frac1{q_0(x)} + \frac12\int_0^s Z(r,x)\;\d r\Big)\\
&=  \frac{|Z(t)|^{p - 1}}{|\frac1{q_0(x)} 
	+ \frac12\int_0^t Z(s,x)\;\d s|^{p - 2}}.
\end{align*}
Therefore,
\begin{align*}
 \Ex \big(|\mathfrak{Q}&(t,\frac12)|^{p - 1}|\mathfrak{u}(t)| \big)\\
&=  \int_{\xi \in [0,\infty)} \int_{r \in \R}
	\frac{|\exp(-\sigma' r) |^{p - 1}}{|\frac1{V_0} + \chi|^{p - 2}} 
		\mathbb{P}\bigg(\bigg\{ W(t) \in \d r ,
			\; \frac12\int_0^t \exp(-\sigma' W(s))\;\d s  \in \d \chi \bigg\}\bigg).
\end{align*}

This law is almost given in \cite[Theorem 4.1]{MR2203675}
 (see also \cite{MR1174378}), where using the notation 
 established in Section \ref{sec:asian_options}, 
 it was shown that for
 \begin{equation*}
A^{(\mu)}(t) = \int_0^t \exp( 2 \mu s + 2 W(s) )\;\d s,
\end{equation*}
one has
\begin{align*}
\mathbb{P}(\{A^{(\mu)}(t) \in \d \chi,\, W(t) \in \d r \}) 
	= \frac{\d \chi}{\chi} e^{\mu r - \mu^2 t/2} 
			\exp\Big(-\frac{1 + e^{2r}}{2\chi}\Big) \vartheta\big(\frac{e^r}{\chi},t\big) \;\d r,
\end{align*}
where again,
\begin{equation*}
\vartheta(y,t) = \frac{y}{\sqrt{2 \pi^3 t}}
	 e^{\pi^2/(2t)}\int_0^\infty e^{-\xi^2/(2t)}
	 	e^{-y \cosh(\xi)}\sinh(\xi) \sin\Big(\frac{\pi \xi}{t}\Big)\;\d \xi.
\end{equation*}

It is possible simply to scale time in both 
$A^{(\mu)}(t)$ and $W(t)$ simultaneously 
as in \eqref{eq:laws_scaling}:
\begin{align*}
\mathbb{P}(\{A^{(0)}(t)& \in \d \chi,\, W(t) \in \d r \}) \\
&=  \mathbb{P}\bigg(\bigg\{\frac{(\sigma')^2}{4}
	\int_0^{4t/(\sigma')^2} \exp(-\sigma'\tilde{W}(\tau))\;\d  \tau
			 \in \d \chi,\, W(t) \in \d r \bigg\}\bigg),
\end{align*}
so that
\begin{align*}
\mathbb{P}(\{A^{(0)}(t) & \in  \d \chi,\, W(t) \in \d r \}) \\
&=  \mathbb{P}\bigg(\bigg\{\frac{(\sigma')^2}{4}
	\int_0^{4t/(\sigma')^2} \exp(-\sigma'\tilde{W}(\tau))\;\d  \tau 
		\in \d \chi,\, -\frac{\sigma'}{2} \tilde{W}\big(4t/(\sigma')^2\big) 
				\in \d r \bigg\}\bigg),
\end{align*}
and 
\begin{align*}
\mathbb{P}\bigg(\bigg\{\int_0^{t} \exp(-\sigma'\tilde{W}(s))\;\d  s &
	\in \d \chi,\, \tilde{W}(t) \in \d r \bigg\}\bigg)\\ 
	&=  \frac{-\sigma'}{2} \frac{\d \chi}{\chi} 
		\exp\bigg(-\frac{2(1 + e^{-\sigma' r})}{(\sigma')^2\chi}\bigg) 
			\vartheta\bigg(\frac{4e^{-\sigma' r/2}}{(\sigma')^2\chi}, \frac{(\sigma')^2t}{4}\bigg) \;\d r\notag.
\end{align*}

Finally integrating in time we find
\begin{align*}
\int_0^T \Ex \big(|\mathfrak{Q}(t,\frac12)|^{p - 1}|\mathfrak{u}(t)| \big)\;\d t
&= \int_0^T\int_\R\int_0^\infty \frac{\exp(-(p - 1)\sigma' r)}{|1/V_0 + \chi|^{p - 2}} \\
&\quad\times\frac{-\sigma'}{2}  \exp\Big(-\frac{2(1 + e^{-\sigma' r})}{(\sigma')^2\chi}\Big)
		 \vartheta\big(\frac{4e^{-\sigma' r/2}}{(\sigma')^2\chi}, \frac{(\sigma')^2t}{4}\big)
		 		 \;\frac{\d \chi}{\chi}\,\d r\,\d t.
\end{align*}

As can be seen, there is no bound for the blow-up 
of this quantity in the small ball $\chi \in B_\ep(-1/V_0)$ 
except if $p - 2 < 1$. However, it is still conceivable 
that there is higher integrability if $\sigma'' \not= 0$). 
Under the principle that ``boxes'' are preserved 
under the flow, the spatial dimension is essentially 
lost in the triple integral (in space, time, and 
probability), but freeing up the spatial variable 
from this constraint gives us, effectively, an 
extra dimension to integrate, opening the 
possibility that the integral remains bounded 
at a higher exponent than $3 - \ep$. This can be 
understood as an effect of regularisation-by-{\em multiplicative} 
noise if indeed it holds \cite{MR2593276}.
\end{remark}

\subsection{Well-posedness for general data} \label{sec:pbt}
Using the same procedure outlined after 
\eqref{eq:ufrak_lin_sigma}, we now extend our 
analysis to general data. We work directly with 
$L^2(\R) \cap L^1(\R)$-valued random variables. 
The following does not generalise easily beyond 
the linear $\sigma$ case, again because in the $\sigma''=0$ case, 
there is no dependence of $\mathfrak{Q}(t,x)$ on 
$x$ through characteristics $X(t,x)$. In particular, 
as mentioned in Remark \ref{rem:meaning_of_q(X)}, 
$\mathfrak{Q}(t,x)$ is simply defined up to 
wave-breaking via \eqref{eq:q_on_char_lin_sigma}:
\begin{align}
\mathfrak{Q}(t,x) 
	= \frac{e^{-\sigma W(t)}}{\frac1{q_0(x)} + \frac12\int_0^t e^{-\sigma W(s)}\;\d s }.
\end{align}
In this subsection, all solutions refer exclusively to conservative 
or dissipative solutions-along-characteristics.

\begin{theorem}[Conservative Solutions: General Initial Data]\label{thm:gen_cons}
Suppose $\sigma'' = 0$ and $q_0 \in L^2(\R) \cap L^1(\R)$. 
There exists a $U \in C([0,\infty)\times \mathbb{R})$, 
absolutely continuous in $x$, $\mathbb{P}$-almost surely, 
such that for each $x \in \R$, the following SDE is 
globally well-posed:
\begin{align*}
\XXX(t,x) = x +  \int_0^t U(s,\XXX(s,x))\;\d s 
	+ \int_0^t \sigma(\XXX(s,x)) \circ \;\d W(s).
\end{align*}

For $Q(t,x) = \pd_x U(t,x)$, the process 
$Q(t,\XXX(t,x))$ agrees $\mathbb{P}$-almost surely 
with $\mathfrak{Q}(t,x)$ as given by \eqref{eq:q_on_char_lin_sigma} 
up to $t = t^*$ and can be represented globally as
\begin{equation*}
Q(t,\XXX(t,x)) 
= \frac{\exp(-\sigma' W(t))}{\frac1{q_0(x)} + \frac12\int_0^t \exp(-\sigma' W(s))\;\d s}.
\end{equation*}

In particular, $\tilde{Q}(t)= Q(t,\XXX(t,x))$ 
satisfies \eqref{eq:qX_sde2}: 
\begin{equation*}
\d \tilde{Q}(t)= - \frac{1}{2} \tilde{Q}^2(t)\;\d t 
	+ \sigma' \tilde{Q}(t)\circ \d W.
\end{equation*}

Furthermore, $Q \in  L^2(\Omega \times [0,T] \times \R)$ 
and in $ C([0,T];H^{-1}_\loc(\R))$, $\mathbb{P}$-almost surely, 
and the energy can be expressed $\mathbb{P}$-almost surely as
\begin{align}\label{eq:energy_consv}
\int_\R Q^2(t,x)\;\d x =  \int_\R Q^2(0,x) \exp(-\sigma' W(t))\;\d x.
\end{align}

\end{theorem}

Similarly, for the dissipative solutions-along-characteristics, we have:

\begin{theorem}[Dissipative Solutions: General Initial Data]\label{thm:gen_dissp}
Suppose $\sigma'' = 0$ and $q_0 \in L^2(\R) \cap L^1(\R)$. 
There exists a $U \in C([0,\infty)\times \mathbb{R})$, 
absolutely continuous in $x$, $\mathbb{P}$-almost surely, 
such that for each $x \in \R$, the SDE
\begin{align*}
\XXX(t,x) = x +  \int_0^t U(s,\XXX(s,x))\;\d s
	 + \int_0^t \sigma(\XXX(s,x)) \circ \;\d W(s)
\end{align*}
 is globally well-posed.

For $Q(t,x) = \pd_x U(t,x)$, the process $Q(t,\XXX(t,x))$ 
agrees $\mathbb{P}$-almost surely with $\mathfrak{Q}(t,x)$ 
up to $t = t^*_x$ and can be represented globally as
\begin{align}
Q(t,\XXX(t,x)) = \begin{cases}\frac{\exp(-\sigma' W(t))}{\frac1{q_0(x)} 
	+\frac12 \int_0^t \exp(-\sigma' W(s))\;\d s}, & t < t^*_x,\\ 
	0,&t > t^*_x.\end{cases}
\end{align}
Here $t^*_x$ is given by \eqref{eq:t_break}.

In particular, $\tilde{Q}(t)= Q(t,\XXX(t,x))$ 
satisfies \eqref{eq:qX_sde2}:
\begin{equation*}
\d \tilde{Q}(t)= - \frac{1}{2} \tilde{Q}^2(t)\;\d t + \sigma' \tilde{Q}(t)\circ \d W.
\end{equation*}

Furthermore, $Q \in L^2(\Omega \times [0,T] \times \R)$ 
and in $C([0,T];H^{-1}_\loc(\R))$, $\mathbb{P}$-almost surely, 
and the energy can be expressed $\mathbb{P}$-almost surely as
\begin{align}\label{eq:energy_dissp}
\int_\R Q^2(t,x)\;\d x =  \int_\R Q^2(0,x)
	 \exp(-\sigma' W(t))\mathds{1}_{\{t \le t^*_x\}}\;\d x.
\end{align}
\end{theorem}

This generalizes the main theorem in 
\cite[Thm. 4.1]{MR2796054} to the stochastic setting.

\begin{remark}
The inclusions preceding \eqref{eq:energy_consv} 
and \eqref{eq:energy_dissp} are  
implied by the respective equations. This was already shown 
in Remarks \ref{rem:energy_balance_cons} and 
\ref{rem:energy_balance_dissp}, respectively.
\end{remark}

\subsection{Conservative solutions} 

In the case $\sigma'' = 0$, $\mathfrak{Q}(t,x)$ 
in \eqref{eq:q_on_char_lin_sigma} is independent of $X(t,x)$, 
and only depends on $x$ via $q_0(x)$. It becomes possible, 
if $q_0 \in L^2(\R) \cap L^\infty(\R)$, to define 
$U(t,\XXX(t,x))$ as the spatial integral of $u$. However, 
in order to avoid cyclic dependencies when $U$ is 
used to define $\XXX$ via an SDE analogous to 
\eqref{eq:characteristics}, we define first an 
auxiliary function which should be thought of as $U(t,\XXX(t,y))$:
\begin{align}\label{eq:Psi_auxU}
\Psi(t,y) = \int_{-\infty}^y \mathfrak{u}(t,x) \;\d x.
\end{align}
Recall that $\mathfrak{u}$ is explicitly given in 
\eqref{eq:u_simple_lin_sigma} and depends on 
$x$ only via $q_0$. In the conservative construction 
we extend this definition by the same formula to 
$t > t^*_x$ as we did in the specific cases of 
``box''-type data.

Define the characteristics via the equation:
\begin{align}\label{eq:tildeX_L2}
\XXX(t,x) = x + \int_0^t\Psi(s,x)\;\d s 
		+ \int_0^t \sigma(\XXX(s,x)) \circ \d W(s),
\end{align}
which is straightforward as $\sigma$ is linear 
and $\Psi(t,y)$ is a well-defined process, being dependent
only on $\mathfrak{u}$, which in turn is defined
explicitly in \eqref{eq:u_simple_lin_sigma},
as, analogous to \eqref{eq:characteristics_difference}, 
the derivative
\begin{align}\label{eq:dX_dx}
 \frac{\pd \XXX(t,x)}{\pd x} 
 	= \exp\Big(\int_0^t \mathfrak{Q}(s,x)\;\d s + \sigma' W(t)\Big)
 \end{align}
is well-defined and non-negative, the right-hand side again being 
dependent on $x$ only through $q_0$. 

This allows us to define
\begin{equation*}
U( t,x) := \Psi(t,y), \qquad \XXX(t,y;\omega) = x,
\end{equation*}
as long as $t \not= t^*_y $ (cf.~\eqref{eq:wave_breaking_time}).
 Such a $y$ exists because $\pd \XXX/\pd x$ is 
 $\mathbb{P}$-almost surely bounded, and strictly positive. 
 The function $U$ is well-defined even though $y$ as a random variable 
 may not be unique because $U$ only depends on $y$ via 
 $\XXX(t,y)$. The variable $y$ is therefore a 
 device for shifting stochasticity back-and-forth 
 between $x$ and $X(t,y)$, and depends on the 
 Jacobian $\pd X(y)/\pd y$ being non-singular. 
 To expand on this point we record a general version 
 of Lemma \ref{thm:stoch_dafermos_box}:
 
\begin{lemma}[``Stochastic Flow of Diffeomorphism'' before Wave-breaking for General Data] \label{thm:general_stoch_flow_diffeo}
 Given $t$ and $x$ deterministic, there is a random 
 variable $y\colon \Omega \to \R$ such that $\XXX(t,y) = x$, 
 $\mathbb{P}$-almost surely. If there are two such 
 random variables $y_1$ and $y_2$ that satisfy this equation, 
 then $y_1 - y_2$ is supported on the set 
 $\{\omega: t^*_{y_1} = t\}\cap\{\omega:  t^*_{y_2} = t\}$ in the sense 
 that on the full $\mathbb{P}$-measure of the complement, 
 the difference is nought.
\end{lemma}

We emphasize here the hierarchy of dependencies, 
being that $U$ depends on $\XXX$, 
which depends on $\Psi$ in the above. 
The function $\Psi$ in turn depends on $\mathfrak{u}$, 
which in the $\sigma'' = 0$ case, is given explicitly 
by formula \eqref{eq:u_simple_lin_sigma}, 
derived using the similarly explicit formula \eqref{eq:q_on_char_lin_sigma}
for the process $\mathfrak{Q}(t,x)$.

The definition of $U$ ensures that
\begin{equation*}
U(t,\XXX(t,x)) = \Psi(t,x)
\end{equation*}
and consequently
\begin{equation*}
\XXX(t,x) = x + \int_0^t U(s,\XXX(s,x))\;\d s 
		+ \int_0^t \sigma(\XXX(s,x)) \circ \d W(s).
\end{equation*}

It remains for us to check that,  $\mathbb{P}$-almost surely,
\begin{itemize}

\item[(i)]  $Q(t,\XXX(t,x)) =  (\pd_x U)(t,\XXX(t,x))$ satisfies \eqref{eq:qX_sde2}, and

\item[(ii)]  $Q \in C([0,T]; H^{-1}_\loc(\R))$.
\end{itemize}

By continuity in $H^{-1}_\loc$ we mean that for 
every pre-compact $B \in \R$,  $\|Q(t)\|_{H^{-1}(B)}$ 
is continuous. In turn, the space $H^{-1}$ is defined 
as the dual space of of compactly supported $H^1$ functions. 
It is norm-equivalent to $L^2$ of the 
anti-derivative on compact sets.

\begin{proof}[Proof of Theorem  \ref{thm:gen_cons}]
By construction, (i) is already satisfied. 
We can take the spatial derivative easily enough:
\begin{align*}
(\pd_x U)(t,x) &= \pd_x \int_0^y \mathfrak{Q}(t,z) \frac{\pd \XXX}{\pd z} \;\d z\\
&= \mathfrak{Q}(t,y)\frac{\pd \XXX}{\pd y} \frac{\pd y}{\pd x}\\
&=  \mathfrak{Q}(t,y).
\end{align*}
Putting $X(t,x)$ in the place of $x$, 
we can put $x$ in the place of $y$, giving us:
\begin{align*}
(\pd_x U)(t,\XXX(t,x)) = \mathfrak{Q}(t,x).
\end{align*}

To prove \eqref{eq:energy_consv} we again invoke 
Lemma \ref{thm:technical} (in particular, \eqref{eq:u_simple}) and \eqref{eq:dX_dx}:
\begin{align*}
\int |Q(t,x)|^2\;\d x&= \int |Q(t,\XXX(t,y))|^2 \frac{\pd \XXX(y)}{\pd y}\;\d y \\
&= \int \frac{Z(t,y)}{\frac1{Q(0,y)} +\frac12 \int_0^t Z(s,y)\;\d s} Q(t,\XXX(t,y))\\
	&\qquad\qquad \qquad \times \exp\bigg(\int_0^t Q(s,\XXX(s,y))\;\d s 
			+ \sigma' W(t)\bigg)\;\d y\\
&= \int Q(0,y) ^2 Z(t,y)\;\d y\\
&= \int Q(0,y) ^2 \exp(-\sigma' W(t))\;\d y,
\end{align*}
where again we have used the notation 
$Z(t,y) = \exp\big(\int_0^t \sigma' \;d W\big) = \exp(\sigma' W(t))$.

Finally to see (ii), we consider the almost sure 
continuity of $\|U(t)\|_{L^2(B)}^2$ over a 
pre-compact set $B \subseteq \R$:
\begin{align}\label{eq:H-1_UL2}
\int_B |U(t,x)|^2\;\d x &= \int_B \bigg|\int_{-\infty}^x \mathfrak{u}(t,y)\;\d y \bigg|^2\;\d x.
\end{align}
As was shown in \eqref{eq:u_simple_lin_sigma}, 
\eqref{eq:u_simple} in Lemma \ref{thm:technical}, 
\begin{align}\label{eq:fraku_simple2}
\mathfrak{u}(t,y) 
= q_0(y) \Big(1 + \frac{q_0(y)}{2}
		 \int_0^t \exp(-\sigma' W(s))\;\d s\Big),
\end{align}
which is path-by-path continuous in time for each 
fixed $y$ that is a Lebesgue point of $u$. 
The boundedness of the integral on the right in 
\eqref{eq:H-1_UL2} is then a result of the 
assumption $q_0 \in L^2(\R)\cap L^1(\R)$. 
Therefore, 
\begin{align}
\|U(t)& -U(s)\|_{L^2(B)}^2  \notag  \\
&= \int_B \bigg|\int_{-\infty}^x \mathfrak{u}(t,y)\;\d y
 - \int_{-\infty}^x \mathfrak{u}(s,y)\;\d y \bigg|^2\;\d x\notag\\
&= \int_B \bigg| \int_{-\infty}^x \frac{q^2_0(y)}{2} \bigg|^2 \;\d x \times \bigg( \int_s^t \exp(-\sigma'W(\tau))\;\d \tau\bigg)^2.
		\label{eq:QH-1_continuity}
\end{align}
The same boundedness of integral of $\mathfrak{u}$, 
and continuity of $\mathfrak{u}(\dott, y)$ in time 
means that the limit as $s \to t$ is almost surely $0$. 
This shows the continuity of $\|Q(t)\|_{H^{-1}_\loc}$ in time. 

\end{proof}

\begin{remark}[Temporal continuity of $\|Q(t)\|^2_{{H}^{-1}(B)}$]
\label{rem:temp_cont_QH-1}
From second factor in the integral with respect to $x$ of the 
foregoing calculation, \eqref{eq:QH-1_continuity}, 
upon comparison with \eqref{eq:fraku_simple2}, 
it can be seen that in fact $\|Q(t)\|^2_{{H}^{-1}(B)}$ 
is $\mathbb{P}$-almost surely in $C^{1/2 - 0}$, 
and not simply continuous. Even though on taking the square root 
 $\|Q(t)\|_{H^{-1}(B)}$ possesses strictly higher regularity-in-time
than simply $\mathbb{P}$-almost sure inclusion in $C(\R)$,
this still contrasts with the local Lipschitz continuity of $\|{q}(t)\|_{H^{-1}_\loc}$ 
that deterministic solutions ${q}$ possess (cf.~\eqref{eq:dtm_H-1_continuity}).
\end{remark}

\subsection{Dissipative solutions} We proceed directly to 
the proof of Theorem \ref{thm:gen_dissp}.
\begin{proof}[Proof of Theorem  \ref{thm:gen_dissp}]
By dissipative we mean solutions for which
\begin{equation*}
Q(t,\XXX(t,x)) = \begin{cases} \exp(-\sigma' W(t))\big(\frac1{q_0(x)} + \frac12\int_0^t \exp(-\sigma'W(s))\;\d s\big)^{-1}, & t < t^*_x,\\ 
0, & t > t^*_x.\end{cases} 
\end{equation*}
Again, as $\sigma'' = 0$, the right-hand side 
only depends on $x$ via $q_0$.

Defining $U(t,\XXX(t,x))$ as before, we can write
\begin{align}
U(t,\XXX(t,x)) &= \int_{-\infty}^x Q(t,\XXX(t,y))\frac{\pd \XXX}{\pd x}\;\d y \notag\\
&= \int_{-\infty}^x q_0(y) \Big(1 +\frac{q_0(y)}{2}
	\int_0^t \exp(-\sigma' W(s))\;\d s\Big)\mathds{1}_{\{t < t^*_y\}}\;\d y.
\end{align}

Therefore, again, there is no dependence of $U(t,\XXX(t,x))$ 
on $\XXX(t,x)$, and $U(t,\XXX(t,x))$ is explicitly known. 
From this and the boundedness of $U$ it is clear
 that we can find a global solution to 
\begin{equation*}
\XXX(t,x) = x + \int_0^t U(s,\XXX(s,x))\;\d s + \int_0^t \sigma (\XXX(s,x))\circ \d W.
\end{equation*}

It follows as in the conservative solution that 
\begin{itemize}

\item[(i)]  ${\XXX}(t,x)$ also satisfies
\begin{equation*}
\d {\XXX}(t,x) = U(t,\XXX(t,x)) \;\d t + \sigma(\XXX(t,x))\circ \d W
\end{equation*}
up to $t = t^*_x$, and remains well-defined beyond this time, and

\item[(ii)]  $Q(t,\XXX(t,x)) =  (\pd_x U)(t,\XXX(t,x))$ satisfies \eqref{eq:qX_sde2}.
\end{itemize}

To prove \eqref{eq:energy_dissp} we invoke 
Lemma \ref{thm:technical} (in particular, \eqref{eq:u_simple}) 
and \eqref{eq:dX_dx} exactly as in the proof 
immediately foregoing:
\begin{align*}
\int |Q(t,x)|^2\;\d x&= \int |Q(t,\XXX(t,y))|^2 \frac{\pd \XXX(y)}{\pd y}\;\d y \\
&= \int \frac{\exp(-\sigma' W(t))}{\frac{1}{Q(0,y)} + \frac12\int_0^t \exp(-\sigma' W(t))\;\d s}\\
&\qquad \times \mathds{1}_{\{t \le t^*_y\}} Q(t,\XXX(t,y))
	 \exp\Big(\int_0^t Q(s,\XXX(s,y))\;\d s + \sigma' W(t)\Big)\;\d y\\
&= \int Q(0,y) ^2 \exp(-\sigma' W(t))\mathds{1}_{\{t \le t^*_y\}}\;\d y.
\end{align*}

We also show $Q \in C([0,T];H^{-1}_\loc(\R))$ by 
showing that $\|U(t)\|_{L^2(B)}$ is continuous in time. 
As before we have
\begin{align*}
\int_B |U(t,x)|^2\;\d x &= \int_B \bigg| \int_{-\infty}^x q_0(y) 
	\Big(1 +\frac{q_0(y)}{2}\int_0^t \exp(-\sigma' W(s))\;\d s\Big)
		\mathds{1}_{\{t < t^*_y\}}\;\d y \bigg|^2\;\d x.
\end{align*}
Continuity follows as in part (ii) of the proof of Theorem \ref{thm:gen_cons}. 
The only difference is continuity at wave-breaking. 
This in turn follows from Lemma \ref{thm:technical}, 
where this time we invoke its main conclusion that at $t^*_y$, 
the integrand of the inner integral, $\mathfrak{u}(t,y)$, 
tends $\mathbb{P}$-almost surely to nought. 
In dissipative solutions, we continue $U$ past wave-breaking 
by simply setting $\pd_x U(t,x)$ to be nought 
after $t = t^*_y$ for a $y$ where $\XXX(t,y) = x$.

\end{proof}

Finally, as in the case of ``box''-type initial data, 
we retain the Oleinik-type one-sided estimate:
\begin{corollary}\label{thm:one_sided_estm_gen}
Suppose $\sigma'' = 0$ and $q_0 \in L^1(\R) \cap L^2(\R)$. 
Then the dissipative solution $Q$ with initial condition 
$Q(0) = q_0$ in $L^1(\R) \cap L^2(\R)$ satisfy 
$\mathbb{P}$-almost surely the following one-sided bound:
\begin{align*}
Q(t,\XXX(t,y)) \le \frac{\exp(-\sigma' W(t))}{\frac1{\max(q_0(y),0^+)} 
		+ \frac12\int_0^t \exp(-\sigma' W(s))\;\d s}.
\end{align*}

\end{corollary}

\begin{remark}[Discrete approximations]
From Lemma \ref{thm:qX_behaviours} (ii),
 it may be possible first to consider well-posedness 
 in the space of step functions, and thereafter to 
 extend this by a limiting procedure to more 
 general compactly supported $L^2(\R)$ functions. 
 As in the deterministic setting, see e.g., \cite{MR2182833}, 
 it is enough to add the boxes together:

Let $P = (x_0, \ldots, x_n)$ be a partition of 
$[x_0,x_n] \subset \R$, and $q_0$ be the function
\begin{align*}
q_0(x) = \sum_{i = 1}^n V_0^i 
		\mathds{1}_{(x_{i - 1}, x_{i + 1})}(x), \qquad V_0^i \in \R.
\end{align*}

For $i = 1,\ldots, n$, let $t^*_i$ be the 
wave-breaking time for the $i$th box. 
These are obviously not dependent on one another. 
Where $V^i_0 \ge 0$, we put $t^*_i = \infty$, 
$\mathbb{P}$-everywhere.

As neighbouring intervals are almost disjoint on $\R$ 
the analysis on any one box can be extended to show 
that where $U_i(t,x)$ is counterpart of \eqref{eq:U_construction},
\begin{align*}
U_i(t,x) = \begin{cases}0, & x \le \XXX(t,x_{i - 1}),\\
\frac{x - \XXX(t,x_{i - 1})}{\XXX(t,x_i) - \XXX(t,x_{i - 1})} \mathfrak{u}_i(t), &x \in (\XXX(t,x_{i -1}),\XXX(t,x_i)),\\
\mathfrak{u}_i(t), & x \ge \XXX(t,x_i),\end{cases}
\end{align*}
with (Recall that the left-hand side does not actually 
depend on some $X(s,\frac12(x_{i - 1} + x_i))$, but 
only on the value $q_0(\frac12(x_{i - 1} + x_i))$.)
\begin{align*}
\mathfrak{u}_i(t) 
:= \mathfrak{Q}(t,\frac12(x_{i - 1} + x_i)) 
	\exp\Big(\int_0^t \mathfrak{Q}(s,\frac12(x_{i - 1} + x_i))\;\d s + \sigma' W(t)\Big),
\end{align*} 
we can write the solution $u(t,x)$ as the sum
\begin{align*}
u(t,x) = \sum_{i = 1}^{n} U_i(t,x).
\end{align*}

This can be extended to an $L^2(\R)$ initial condition $q_0$ by setting
\begin{equation*}
V_0^i  = \fint_{x_{i - 1}}^{x_i} q_0(x)\;\d x,
\end{equation*}
so that the approximation with the partition $P$ is 
\begin{equation*}
q^P_0(x) =  \sum_{i = 1}^n V_0^i \mathds{1}_{(x_{i - 1}, x_{i + 1})}(x).
\end{equation*}
Next, suppose one can find spaces on which the set 
$\{u^P,q^P\}_{\|P\| > 0}$ is weakly compact, 
and on which the associated collection of laws 
$\{\mu^P\}_{\|P\| > 0}$ is correspondingly tight 
(see Ondrej\'at \cite{MR2659757} for conditions giving 
compact embeddings into spaces of functions weakly 
continuous in time, and $W^{k,p}_\loc(\R)$ in space). 
Invoking the Jakubowski--Skorohod theorem \cite{MR1453342} 
in taking the limit of a subsequence as $\|P \| \to 0$, 
one obtains a limit process whose law on a new stochastic 
basis is the same as that of the weak-star limit $\mu$ 
of the tight sequence $\{\mu^P\}$ on the original 
stochastic basis,  that is, the same conclusions 
as for the conventional Skorohod theorem, 
but applied to function spaces without the requisite separability.

It then only behooves one to conclude the 
argument by showing that the stochastic integrals 
against $\d \tilde{W}$, where $\tilde{W}$ is the 
representation of the original Brownian motion in 
the new stochastic basis, remain martingales, 
in the manner of \cite{MR3549199, MR3502597}.
\end{remark}

\section{Reconciling Different Notions of Solutions}\label{sec:characteristics_sols}

Finally we complement the results concerning 
conservative and dissipative solutions-along-characteristics 
by reconciling them with conservative and dissipative 
weak solutions, respectively, which are more traditional 
to the subject of partial differential equations. 
These notions of solutions are all defined in Section \ref{sec:pde_sols}. 

\begin{proposition}[Existence of Conservative Weak Solutions]\label{thm:consv_sol}
Suppose $q_0\in L^1(\R) \cap L^2(\R)$ and $\sigma'' = 0$. For processes given by
\begin{align*}
U(t,X(t,x))&= \int_{-\infty}^x q_0(y) \bigg(1 +\frac{q_0(y)}{2}\int_0^t \exp(-\sigma' W(s))\;\d s\bigg)\;\d y,\\
X(t,x) &= x + \int_0^t U(s,X(s,x))\;\d s + \int_0^t \sigma (X(s,x))\circ \d W,\\
Q(t,X(t,x)) &= \exp(-\sigma' W(t))\bigg[\frac1{q_0(x)}+ \frac12\int_0^t \exp(-\sigma'W(s))\;\d s\bigg]^{-1},
\end{align*}
the function defined by
\begin{equation*}
q(t,x) = Q(t,X(t,y)),
\end{equation*}
where $y \in \R$ satisfies $x = X(t,y)$, is a conservative weak solution.
\end{proposition}

\begin{proposition}[Existence of Dissipative Weak Solutions]\label{thm:dissp_sol}
Suppose $q_0\in L^1(\R) \cap L^2(\R)$ and $\sigma'' = 0$. 
For a collection $\{t^*_x\}$ of random variables defined by
\begin{equation*}
-q_0(x) \int_0^{t^*_x} \exp\big(-\sigma' W( s)\big)\;\d s  = 2,
\end{equation*}
indexed by the Lebesgue points $x$ of $q_0(x)$, 
and processes given by
\begin{align*}
U(t,X(t,x))&= \int_{-\infty}^x q_0(y) \Big(1 +\frac{q_0(y)}{2}\int_0^t \exp(-\sigma' W(s))\;\d s\Big)\mathds{1}_{\{t < t^*_y\}}\;\d y,\\
X(t,x) &= x + \int_0^t U(s,X(s,x))\;\d s + \int_0^t \sigma (X(s,x))\circ \d W,\\
Q(t,X(t,x)) &= \begin{cases} \exp(-\sigma' W(t))\bigg[\frac1{q_0(x)} + \frac12\int_0^t \exp(-\sigma'W(s))\;\d s\bigg]^{-1}, 
& t < t^*_x,\\ 
0, & t > t^*_x,\end{cases} 
\end{align*}
the function defined by
\begin{equation*}
q(t,x) = Q(t,X(t,y)),
\end{equation*}
where $y \in \R$ satisfies $x = X(t,y)$, 
is a dissipative weak solution.

\end{proposition}

\begin{proof}[Proof of Proposition \ref{thm:consv_sol}]

Since the process $\tilde{Q}(t) = Q(t,X(t,y))$ satisfies \eqref{eq:qX_sde2},
\begin{align*}
\d \tilde{Q}(t) = -\frac{1}{2}\tilde{Q}^2(t) \;\d t - \sigma' \tilde{Q}(t) \circ \d W,
\end{align*}
up to $t < t^*_y$, pointwise for $y$ in the 
set of Lebesgue points of $q_0$, 
by the It\^{o} formula it manifestly holds 
that up to the same stopping time, 
\begin{align}\label{eq:tildeQ_sq}
\d \tilde{Q}^2(t) =  - \tilde{Q}^3(t) \;\d t - 2 \sigma' \tilde{Q}^2(t) \circ \d W.
\end{align}

On $\mathbb{P}$-almost every path, except at the 
time $t = t^*_y$, we have shown that these equations remain valid. 
This is possible because we are only concerned 
with the Lebesgue points of $q_0$, which is a deterministic, 
time independent object. Let $\varphi \in C^\infty_0(\R)$. 
First we observe that since $\pd X(t,y)/\pd y > 0$ 
for almost every $(t,y) \in [0,T]\times \R$, 
$\mathbb{P}$-almost surely, it holds that 
for almost every $t$, $\mathbb{P}$-almost surely,
\begin{align}\label{eq:almost_homeomorphism}
\int_\R Q^2(t,X(t,y)) \varphi(X(t,y))\;X(\d y) 
	= \int_\R Q^2(t,x) \varphi(x) \;\d x,
\end{align}
where we have used $X(\d y)$ instead of 
$\d X(y)$ to denote the {\it deterministic} 
differential to emphasise integration in the 
{\em spatial}, and not the temporal variable. 
We can disregard the measure zero set 
in $t$ (wave-breaking only 
occurs once along each characteristic) as 
we shall be integrating over $t$.

 By \eqref{eq:dX_dx}, in the sense of It\^o, 
 we have the $\mathbb{P}$-almost sure equality
\begin{align}
\d\bigg(\int_\R Q^2(t,X(t,y))& \varphi(X(t,y))\;X(\d y)\bigg) \notag\\
&= \d \bigg(\int_\R Q^2(t,X(t,y)) \varphi(X(t,y)) \frac{\pd X(t,y)}{\pd y}\;\d y\bigg)\notag\\
& =  \int_\R \d Q^2(t,X(t,y)) \circ \Big( \varphi(X(t,y)) \frac{\pd X(t,y)}{\pd y}\Big)\;\d y \notag\\
& \quad  + \int_\R Q^2(t,X(t,y)) \circ \d \Big( \varphi(X(t,y)) \frac{\pd X(t,y)}{\pd y}\Big) \;\d y.\label{eq:QX_energy}
\end{align}

We already know how to expand $\d Q^2(t,X(t,y))$ 
from \eqref{eq:tildeQ_sq}. Therefore we inspect the 
second summand in the final line of the foregoing 
calculation. Since $\varphi$ is a smooth, deterministic 
function, by the regular chain rule,
\begin{align*}
\d \Big( \varphi(X(t,y)) &\frac{\pd X(t,y)}{\pd y}\Big)  \\
& = \d \varphi(X(t,y)) \circ \frac{\pd X(t,y)}{\pd y} 
	+ \varphi(X(t,y))\circ \d\frac{\pd X(t,y)}{\pd y}\\
 & = \frac{\pd X(t,y)}{\pd y}\circ \Big(\pd_x \varphi(X(t,y)) \circ \d X(t,y)\Big) 
 	+  \varphi(X(t,y))\circ \d\frac{\pd X(t,y)}{\pd y}.
\end{align*}
Recalling Remark \ref{rem:non_assoc_stratonovich_prod}, 
and using \eqref{eq:dX_dx} and the equation for 
the characteristics in the theorem statement,
\begin{align*}
\d \Big( \varphi(X(t,y))& \frac{\pd X(t,y)}{\pd y}\Big) \\
 & = \frac{\pd X(t,y)}{\pd y}\pd_x \varphi(X(t,y)) \circ \d X(t,y) 
 	+  \varphi(X(t,y))\circ \d\frac{\pd X(t,y)}{\pd y}\\
 &=  \pd_x\varphi(X(t,y))U(t,X(t,y))\frac{\pd X(t,y)}{\pd y} \;\d t\\
 &\quad + \pd_x\varphi(X(t,y))\sigma(X(t,y))\frac{\pd X(t,y)}{\pd y}\circ \;\d W\\
&\quad + \varphi(X(t,y)) \circ \Big( Q(t,X(t,y)) \frac{\pd X(t,y)}{\pd y}\; \d t 
	+ \sigma' \frac{\pd X(t,y)}{\pd y}\circ \d W \Big). 
\end{align*}
Inserting this into \eqref{eq:QX_energy} and 
using \eqref{eq:tildeQ_sq} and \eqref{eq:almost_homeomorphism}, 
we recover the weak energy balance \eqref{eq:sol_cons}, 
where $\pd_{xx}^2\sigma =0$ in the linear $\sigma$ case.
\end{proof}

For dissipative solutions, we shall be multiplying 
by an extra factor of $\mathds{1}_{\{t < t^*_y\}}$ 
in the proof below. 
The selection of $y$ for times $t > t^*_y$ has in 
fact been dealt with in Section \ref{sec:well_posedness1}, 
where we have shown how to extend characteristics 
globally through a wave-breaking point.

\begin{remark}
If it can be shown that any conservative weak solution 
$(u,q)$ can be used to construct characteristics
\begin{equation*}
\d X(t,y) = u(t,X(t,y))\;\d t + \sigma(X(t,y))\circ \d W
\end{equation*}
that are for almost every $t \in [0,T]$ and 
$\mathbb{P}$-almost surely a $C^1$ surjection of $\R$
for which $\pd X/\pd x \ge 0$, then 
the calculations of the foregoing proof 
can be done in reverse to attain the reverse implication 
that conservative weak solutions are necessarily 
conservative solutions-along-characteristics. 
This would imply uniqueness of solutions. 
We relegate this proof to an upcoming work.
\end{remark}

\begin{proof}[Proof of Prop. \ref{thm:dissp_sol}]
The proof here essentially follows the one for 
Proposition \ref{thm:consv_sol} with the exception that 
there is a defect measure arising from the temporal 
derivative, and we employ \eqref{eq:tildeQ_sq} 
in evaluating the quantity:
\begin{align*}
\d (Q^2(t,X(t,y)) \mathds{1}_{\{t \le t^*_y\}}) 
	= \mathds{1}_{\{t \le t^*_y\}}\,\d Q^2(t,X(t,y)) 
			- Q^2(t,X(t,y))\delta(t - t^*_y)\,\d t,
\end{align*}
understood in the weak sense. (See 
Appendix \ref{sec:appendixB} for
 the deterministic analogue, 
along with a discussion of 
this ``defect measure''.)

Since $Q^2 \delta \ge 0$, the inequality replaces 
the equal sign when this measure is suppressed. 
This is the weak energy inequality \eqref{eq:sol_dissp}. 
See also \eqref{eq:defect_meas} for the deterministic analogue.

Almost sure boundedness from above is given by 
Lemma \ref{thm:one_sided_estm_gen}. 
Except on the set $\{\omega: t^*_x \ge t\}$, 
for $\mathbb{P}$-almost every $\omega$ there 
exists a unique $y$ such that $X(t,y) = x$. 
On that set we know that $Q(t,x)$ can be bounded by $0$. 
Since every $x = X(t,y)$ can be reached from 
some $y$ at $t =0$ on a characteristic, 
the one sided estimate holds for $Q(t,x)$ in the general case.
\end{proof}

\begin{remark}[Maximal dissipation of of energy]
With regards to comments following Definition \ref{def:weak_diss_sols}, 
we intend to show in an upcoming work that 
maximal energy dissipation is given by
 \ref{eq:max_energ_dissp}, as well as the 
 uniqueness of dissipative weak solutions.
\end{remark}

\appendix
\section{Lagrangian and Hamiltonian Approaches to the Hunter--Saxton equation}\label{sec:variation_hamiltonian}
Here we motivate the stochastic Hunter--Saxton equation \eqref{eq:sHS1} that we study
in this paper.

From Hunter--Zheng \cite{MR1361013} we know that the 
evolution part of the Hunter--Saxton equation is given by
\begin{equation}\label{eq:hamil}
\pd_t u=D^{-1}\frac{\delta H(u)}{\delta u},
\end{equation}
where the Hamiltonian reads
\begin{equation*}
H(u)=\frac12\int u (\pd_x u)^2 \, \d x,
\end{equation*}
and $D^{-1}=\int^x$.   We find that
\begin{equation*}
\frac{\delta H(u)}{\delta u}= \frac12 (\pd_x u)^2-\pd_x(u\pd_x u),
\end{equation*}
which yields
\begin{equation*}
\pd_ x(u_t+u\pd_x u)=\frac12 (\pd_x u)^2.
\end{equation*}
Note that we can write \eqref{eq:hamil} as
\begin{equation*}
\pd_t q=\frac{\delta H(u)}{\delta u}.
\end{equation*}

If we perturb the Hamiltonian as in, e.g., \cite{MR3488697}, by
\begin{equation*}
\tilde{H}(u)=H(u)+\frac12\int \sigma (\pd_x u)^2\circ \dot{W} \, \d x,
\end{equation*}
we find
\begin{equation*}
\frac{\delta \tilde{H}(u)}{\delta u}= \frac12 (\pd_x u)^2-\pd_x(u\pd_x u)-\pd_x(\sigma \pd_x u)\circ \dot{W},
\end{equation*}
which yields
\begin{equation}
\pd_x(\pd_t u+u\pd_x u)=\frac12 (\pd_x u)^2-\pd_ x (\sigma \pd_x u)\circ \dot{W},
\end{equation}
and this is the stochastic Hunter--Saxton equation.

An alternative approach is based on a Lagrangian 
formulation. Let $L=L(u,\pd_t u,\pd_x u) $ denote the Lagrangian.
If we take the first variation
\begin{equation*}
\delta \iint L(u,\pd_t u,\pd_x u) \;\d x\,\d t,
\end{equation*}
we find that the Euler--Lagrange equation reads
\begin{equation*}
\frac{\partial}{\partial x}  \frac{\partial L}{\partial (\pd_x u)} + \frac{\partial}{\partial t}  \frac{\partial L}{\partial (\pd_t u)} - \frac{\partial L}{\partial u}=0.
\end{equation*}

Introduce \cite{MR1361013}
\begin{equation*}
L(u,\pd_t u,\pd_x u)=\pd_x u\pd_t u+u (\pd_x u)^2+\sigma (\pd_x u)^2 \circ \dot{W}.
\end{equation*}
Then we find again that
\begin{equation}
\pd_x (\pd_t u+u\pd_x u)=\frac12 (\pd_x u)^2-\pd_x (\sigma \pd_x u)\circ \dot{W}.
\end{equation}

\section{A-priori Bounds}\label{sec:aprioribounds}
In this appendix we shall establish the a-priori estimates of Proposition \ref{thm:aprioriestimates}, which we re-state below:
\begin{proposition}[A-priori bounds]
Let $q$ be a conservative or dissipative weak solution to the stochastic Hunter--Saxton equation \eqref{eq:sHS1}, with $\sigma \in (C^2 \cap \dot{W}^{1,\infty}\cap \dot{W}^{2,\infty})(\R)$, and initial condition $q(0) = q_0 \in L^1(\R) \cap L^2(\R)$. The following bounds hold:
\begin{align}
\esssup_{t \in [0,T]}\Ex\| q(t)\|_{L^2(\R)}^2 & \le C_T \|q_0\|_{L^2(\R)}^2,\\
\Ex \|q\|_{L^{2 + \alpha} ([0,T]\times \R)}^{2 + \alpha} & \le  C_{T,\alpha} \|q_0\|_{L^2(\R)}^2,
\end{align}
for any $\alpha \in [0,1)$.
\end{proposition}

Consider a standard mollifier  defined by 
$J_\ep(x)=\frac{1}{\ep}J\left(\frac{x}{\ep}\right)$ where
\begin{equation*} 
\quad 0\leq J \in C^\infty_c(\R), 
\quad \supp\, (J)\subseteq [-1,1], 
\quad J(-x)=J(x), 
\quad \int_{\R} J(x)\, \d x =1.
\end{equation*}
We write
\begin{equation*}
	h_\ep := J_\ep \star h	
\end{equation*}
for the (spatial) convolution of a function $h$.  
We prove the following technical lemma on mollifiers.
\begin{lemma}[Regularisation Lemma]\label{thm:molly}
Let $q$ be a weak solution to the stochastic 
Hunter--Saxton equation \eqref{eq:sHS1} with 
$\sigma \in (C^2 \cap \dot{W}^{1,\infty}\cap \dot{W}^{2,\infty})(\R)$. 
The mollified equation holds pointwise in $\R$ over $t < T$,
in the sense of It\^{o} that:
\begin{align}\label{eq:mollif_sHS}
 \d q_\ep = -  J_\ep \star  \pd_x (u q) \;\d t 
 	+ \frac{1}{2} J_\ep \star q^2\;\d t 
 		-  J_\ep \star\pd_x(\sigma q) \;\d W 
 			+ \frac{1}{2} J_\ep \star\pd_x (\sigma\pd_x(\sigma q) )\;\d t .
 \end{align}

In particular, for fixed $\ep$, there is a 
representative of $q_\ep$ (also called $q_\ep$) 
such that for each $\omega \in \Omega$,
\begin{equation*}
q_\ep(\omega) \in C([0,T]; C^\infty(\R)).
\end{equation*}
\end{lemma}
\begin{proof}
The main point is to check that there $q_\ep$ is 
$\mathbb{P}$-almost surely pointwise continuous in time, 
so that there are no dissipative effects when 
an entropy is applied to it, and so that It\^o's 
formula can be applied pointwise in $x$.

By the Burkh\"older--Davis--Gundy inequality, 
for $\beta, \theta > 0 $ and deterministic times $s,t \in [0,T]$,  
\begin{align*}
\Ex \left\|\int_s^t J_\ep\underset{}{\star}\pd_x (\sigma q)\;\d W
	\right\|^\theta_{H^{\beta}_\loc} 
&\le  C \Ex\left( \int_s^t \left\|J_\ep\underset{}{\star}\pd_x (\sigma q) 
	\right\|_{H^{\beta}_{\loc}}^2 \;\d r\right)^{\theta/2}\\
&= C \Ex\left( \int_s^t \left\|(\pd_x^{\beta + 1} J_\ep)\underset{}{\star}(\sigma q) 
	\right\|_{L^2_{\loc}}^2 \;\d r\right)^{\theta/2}\\
 &\le  C \Ex\left( \int_s^t\left\|(\pd_x^{\beta + 1} J_\ep)\underset{}{\star}(\sigma q) 
 	\right\|_{L^\infty_{\loc}}^2 \;\d r\right)^{\theta/2}\\
& \le C_{\beta,\ep,\theta,T, \sigma} |t - s|^{\theta/2},
\end{align*}

By Kolmogorov's continuity theorem, for fixed $\ep > 0$,
we have a $C^{1/2-0}([0,T];H^\beta_\loc(\R))$ 
  representative of the martingale
\begin{equation*}
J_\ep \star\int_0^t \pd_x(\sigma q) \;\d W =
\int_0^t J_\ep \star \pd_x(\sigma q) \;\d W.
\end{equation*}
By the Sobolev embedding theorem on $\R$, 
for $\beta \ge 1$, we have spatial continuity 
to arbitrary order.

 All the other temporal integrals are integrals of 
finite variation, and hence continuous in $t$, 
with integrands that are convolutions against a fixed, smooth function, 
and hence smooth in $x$. This means that 
\begin{equation*}
q_\ep(t) - q_\ep(0) = - \int_0^t J_\ep \star \pd_x (u q) \;\d s 
	- \frac{1}{2} J_\ep \star q^2\;\d s 
		- J_\ep \star\int_0^t \pd_x(\sigma q) \circ \d W
\end{equation*}
is also pointwise continuous. 
This means there is no dissipation arising 
from the mollified equation for fixed $\ep > 0$.

Moreover, since 
\begin{equation*}
J_\ep \star\int_0^t \pd_x(\sigma q) \;\d W 
\end{equation*}
has a  $C^{1/2-0}([0,T];H^\beta_\loc(\R))$ 
continuous representative, we can write its cross-variation with $W$ as
\begin{equation*}
\bigg\LL J_\ep \star\int_0^{(\dott)} \pd_x(\sigma q) \;\d W , W\bigg\RR_t
	 =  J_\ep \star\int_0^t \pd_x(\sigma q) \;\d s.
\end{equation*}
Therefore the normal It\^o formula is sufficient to establish
 equivalence of the Stratonovich and It\^o formulations.

\end{proof}

\begin{lemma}[Mollification error bounds]\label{thm:mollifying_errors}
On the same assumptions as Lemma \ref{thm:molly}, 
the mollified equation \eqref{eq:mollif_sHS} can be re-written as
\begin{align}\label{eq:q_ep}
d q_\eps + \pd_x (u_\ep q_\eps)  \, \d t
	+  \px (\sigma q_\eps) \d W 
		- \frac12 \pd_x (\sigma \pd_x (\sigma q_\ep))
=  \frac12 q_\eps^2\, \d t +(r_\eps + \rho_\ep)\, \d t + \tilde{r}_\ep \,\d W, 
\end{align}
where the mollification error
\begin{align}
r_\eps &:=\Bigl( \px(u_\ep q_\eps)
-\pd_x\bigl(u q\bigr)\star J_\eps\Bigr)
+\frac12 \left( q^2 \star J_\eps 
-q_\eps^2\right)
\end{align}
tends to zero in $L^1( [0,T] ;L^1(\R))$ as $\eps\to 0$, $\mathbb{P}$-almost surely,
\begin{align}
 \tilde{r}_\ep &:=\Bigl(\pd_x (\sigma q_\ep)
-\pd_x\bigl(\sigma q)\star J_\eps\Bigr)\label{eq:r_ep}
\end{align}
tends to zero in $L^2( [0,T] ;L^2(\R))$ as 
$\eps\to 0$, $\mathbb{P}$-almost surely, and, 
for any $S \in C^{1,1}(\R)$ with $\sup_{r \in \R}(|S'(r)| + |S''(r)|) <\infty$, 
\begin{equation*}
\int_0^T \int_{\R} \Big(-S'(q_\ep) \rho_\ep 
	+\frac{1}{2} S''(q_\ep) \Big( (\pd_x(\sigma q_\ep))^2 
			-\big( \pd_x(\sigma q)\star J_\ep\big)^2 \Big)\Big) \;\d x\,\d t
\end{equation*}
tends to zero as $\ep \to 0$, $\mathbb{P}$-almost surely,
where
\begin{align}
\rho_\ep &:=  \frac{1}{2}\Bigl(\pd_x\bigl(\sigma \pd_x (\sigma q))\star J_\eps 
	- \pd_x (\sigma \pd_x (\sigma q_\ep)) \Bigr).
\end{align}
\end{lemma}

\begin{proof}
By Definition \ref{def:weak_sol}, for a weak solution 
$(u,q)$ to the stochastic Hunter--Saxton equation 
\eqref{eq:sHS1}, $q \in L^\infty( [0,T]; L^2(\R))$, 
and $u \in L^\infty([0,T];\dot{H}^1(\R))$ with 
unit $\mathbb{P}$-probability. This fact will be 
implicitly invoked along with the dominated 
convergence theorem, among other instances, 
in the following tripartite calculations.

\smallskip
{\it 1. Estimate of ${r}_\ep$.}

We can decompose $ \px(u_\ep q_\eps) -\pd_x\bigl(u q\bigr)\star J_\eps$ as follows:
\begin{align*}
 \px(u_\ep q_\eps) -&\pd_x\bigl(u q\bigr)\star J_\eps\\
& = \big(\pd_x(u_\ep q_\ep) - \pd_x( u q_\ep) \big)
	+ \big(\pd_x (u q_\ep) -\pd_x\bigl(u q\bigr)\star J_\eps\big).
\end{align*}

We estimate the above  in $L^1(\R)$ term-by-term.

Treating the $L^1_\loc(\R)$ integral as an 
$\dot{H}^1(\R)$ -- $H^{-1}(\R)$ pairing 
between $|u - u_\ep|$ and $|\pd_x q_\ep|$, we have
\begin{align*}
\|\pd_x\big( ( u - u_\ep)q_\ep\big)\|_{L^1(\R)} 
\le C \big(\|u - u_\ep\|_{L^\infty(B_R)} 
	+ \|q - q_\ep\|_{L^2(\R)}\big)\|q_\ep\|_{L^2(B_R)}
\end{align*}
as $\ep \to 0$,  by standard results on convolutions, 
$\mathbb{P}$-almost surely.

The second term tends to nought in $L^1(\R)$ for 
almost every $t \in [0,T]$, $\mathbb{P}$-almost surely 
by  \cite[Lemma 2.3]{MR1422251} (also see \cite[Lemma II.1]{MR1022305}).

The final part of $r_\ep$ is
\begin{align*}
q^2 \star J_\eps - q_\eps^2 
	= q^2 \star J_\eps - q^2 
		+ q( q -q_\ep) + (q- q_\ep) q_\ep.
\end{align*}
It tends to nought in $L^1(\R)$ for almost 
every $t \in [0,T]$, $\mathbb{P}$-almost surely by 
standard theorems on convolutions. 
By the $\mathbb{P}$-almost sure inclusion 
$q \in L^\infty([0,T] ; L^2(\R))$ for weak solutions, 
$\mathbb{P}$-almost surely the $L^1(\R)$ norm of 
the expression above can be uniformly bounded by 
$C\sup_{t \in [0,T]}\|q(t)\|^2_{L^2(\R)}$. 
This expression is of course integrable over $[0,T]$. 
Therefore, by the dominated convergence theorem, 
$\mathbb{P}$-almost surely, the $L^1([0,T];L^1(\R))$ 
convergence follows from the pointwise-in-$t$ 
convergence to zero of
\begin{equation*}
\|q^2(t) \underset{}{\star} J_\eps - q^2(t)\|_{L^1(\R)} 
	+  \|q( q -q_\ep)\|_{L^1(\R)} + \|(q- q_\ep) q_\ep\|_{L^1(\R)},
\end{equation*}
as $\ep \to 0$.

 By the estimate 
\begin{equation*}
\|\pd_x( uq_\ep)(t) - \pd_x (uq)(t) \star J_\ep \|_{L^1_\loc(\R)} 
	\le C \|q(t)\|_{L^2(\R)}^2,
\end{equation*}
also established in \cite[Lemma 2.3]{MR1422251}, 
 it again follows from the $\mathbb{P}$-almost sure
  inclusion $\|q(t)\|_{L^2(\R)} \in L^\infty([0,T])$ 
  via the dominated convergence theorem that
\begin{align}\label{eq:r_ep_estimate}
\int_0^T\|r_\ep(t)\|_{L^1(\R)} \;\d t\le  C_{T,\ep} \int_0^T \|q(t)\|_{L^2(\R)}^2 \;\d t,
\end{align}
where $C_{T,\ep}$ is a quantity independent of $t$, 
vanishing as $\ep \to 0$, and therefore $r_\ep \to 0$ in 
$L^1( [0,T]; L^1_\loc(\R))$, $\mathbb{P}$-almost surely.

\smallskip
{\it 2. Estimate of $\tilde{r}_\ep$.}

This is treated similarly to the second term above, with $\sigma$ in place of $u$.

Since $\sigma \in \dot{W}^{1,\infty}$ and 
$q \in L^\infty([0,T]; L^2(\R))$, it holds that 
$\pd_x (\sigma q) \in L^\infty( [0,T]; L^2_\loc(\R))$ 
with unit $\mathbb{P}$ probability. Therefore this 
time we have slightly higher spatial integrability, 
allowing us to conclude via \cite[Lemma 2.3]{MR1422251} that
\begin{equation*}
\tilde{r}_\ep(t) = \pd_x(\sigma q_\ep)(t) -
	 \pd_x (\sigma q)(t)\star J_\ep \to 0
\end{equation*}
in $L^2(\R)$ for almost every $t \in [0,T]$, 
$\mathbb{P}$-almost surely, as $\ep\to 0$.

Next, by an application of  the dominated convergence theorem 
in a manner previously demonstrated, 
we can conclude that
\begin{align}\label{eq:tilder_ep_estimate}
\int_0^T \|\tilde{r}_\ep(t)\|_{L^2(\R)}^2\;\d t 
	\le C_{T,\ep} \int_0^T \|q(t)\|_{L^2(\R)}^2\;\d t,
\end{align}
where $C_{T,\ep}$ depends on the continuity properties 
of $\sigma$ and its derivatives, in additional 
to $\ep$, for which we have the limit 
$C_{T,\ep} \to 0$ as $\ep \to 0$, $\mathbb{P}$-almost surely. 
Hence $\tilde{r}_\ep \to 0$ in $L^2([0,T]; L^2(\R))$, 
$\mathbb{P}$-almost surely.

\smallskip
{\it 3. Estimate of $\rho_\ep$.}

The estimate of $\rho_\ep$ takes inspiration 
from the proof of \cite[Prop.~3.4]{MR3803774}. 
However, whereas they considered the 
commutator between the operators 
$\tilde{\KK} f := \sigma \pd_x f$ and 
${\JJ} f := f \star J_\ep$, we shall 
have to consider the analogous question 
for $\KK f := \pd_x( \sigma f)$ and $\JJ$. 

Recall that here, we seek not to show 
that $\rho_\ep$ vanishes but that the 
following quantity does:
\begin{align*}
\int_0^T \int_{\R} \Big(-S'(q_\ep) \rho_\ep 
	+\frac{1}{2} S''(q_\ep) \Big( (\pd_x(\sigma q_\ep))^2 
			-\big( \pd_x(\sigma q)\star J_\ep\big)^2 \Big)\Big) \;\d x\,\d t.
\end{align*}

We can write $\rho_\ep$ as
\begin{align}
\rho_\ep &:=\frac{1}{2}\Bigl(\pd_x\bigl(\sigma \pd_x (\sigma q))\star J_\eps 
	- \pd_x (\sigma \pd_x (\sigma q_\ep)) \Bigr)\notag\\
	& \;=\frac{1}{2} \big(\JJ \KK \KK q - \KK \KK \JJ q\big) \notag\\
	&\; = \frac{1}{2} \big(\JJ \KK \KK q - \KK \JJ \KK q + \KK \JJ \KK q - \KK \KK \JJ q\big)\notag\\
	&\; = \frac{1}{2} \big(\big[\JJ,\KK\big] (\KK q)+ \KK \big[\JJ,\KK\big](q)\big)	\label{eq:rho_ep_1},
\end{align}
where
$$
\big[\JJ,\KK\big](q) = \JJ \KK q - \KK \JJ q.
$$

Similarly, we can write the 
remaining part of the integrand as 
\begin{equation}\label{eq:rho_ep_extra}
 (\pd_x(\sigma q_\ep))^2 
			-\big( \pd_x(\sigma q)\star J_\ep\big)^2  
			 = \big( (\KK \JJ q )^2 - (\JJ \KK q)^2\big).
\end{equation}

Therefore, following the calculations in \cite[p.~655]{MR3803774}, we find that
\begin{align*}
 \frac{1}{2} S''(q_\ep) \cdot\eqref{eq:rho_ep_extra}
&= \frac{1}{2} S''(q_\ep)   \big( \KK \JJ q  - \JJ \KK q\big)
	\big( \KK \JJ q  + \JJ \KK q\big)\\
&= -\frac{1}{2} S''(q_\ep) \big(\big[\KK,\JJ\big](q)\big)^2 
	+  S''(q_\ep) (\KK \JJ q) \big[\KK,\JJ\big](q) \\
&= -\frac{1}{2} S''(q_\ep) \big(\big[\KK,\JJ\big](q)\big)^2 
	+  S''(q_\ep) \pd_x \sigma q_\ep \big[\KK,\JJ\big](q)\\  
		&\qquad+ \sigma \pd_x (S'(q_\ep)) \big[\KK,\JJ\big](q)\\
&= -\frac{1}{2} S''(q_\ep) \big(\big[\KK,\JJ\big](q)\big)^2 
	+  S''(q_\ep) \pd_x \sigma q_\ep \big[\KK,\JJ\big](q) \\
		&\quad + \pd_x \big(\sigma S'(q_\ep) \big[\KK,\JJ\big](q)\big) 
			- S'(q_\ep) \pd_x \big(\sigma  \big[\KK,\JJ\big](q)\big)\\
&= -\frac{1}{2} S''(q_\ep) \big(\big[\KK,\JJ\big](q)\big)^2 
	+  S''(q_\ep) \pd_x \sigma q_\ep \big[\KK,\JJ\big](q) \\
		&\quad + \pd_x \big(\sigma S'(q_\ep) \big[\KK,\JJ\big](q)\big) 
			- S'(q_\ep) \KK \big[\KK,\JJ\big](q),
\end{align*}
by invoking the definition of $\KK$.

Adding this to \eqref{eq:rho_ep_1}, we find that
\begin{align}
-S'(&q_\ep) \cdot \eqref{eq:rho_ep_1} 
	+ \frac{1}{2} S''(q_\ep) \cdot \eqref{eq:rho_ep_extra} \notag\\
&= -\frac{1}{2} S''(q_\ep) \big([\KK, \JJ](q)\big)^2 
	+  S''(q_\ep) \pd_x \sigma q_\ep \big[\KK,\JJ\big](q)  
		+ \pd_x \big(\sigma S'(q_\ep) \big[\KK,\JJ\big](q)\big) \notag\\
& \quad - S'(q_\ep) \KK \big[\KK,\JJ\big](q) 
	- \frac{1}{2} S'(q_\ep)\big([\JJ, \KK] (\KK q)
		+ \KK \big[\JJ,\KK\big] (q)\big)	\notag\\
& = -\frac{1}{2} S''(q_\ep) \big([\KK, \JJ](q)\big)^2 
	+  S''(q_\ep) \pd_x \sigma q_\ep \big[\KK,\JJ\big](q) \notag\\
		&\quad + \pd_x \big(\sigma S'(q_\ep) \big[\KK,\JJ\big](q)\big) 
			+ \frac{1}{2}  S'(q_\ep)\big([\KK, \JJ] (\KK q) - \KK [\KK ,\JJ] (q)\big)\notag\\
& = -\frac{1}{2} S''(q_\ep) \big([\KK, \JJ](q)\big)^2 
	+  S''(q_\ep) \pd_x \sigma q_\ep \big[\KK,\JJ\big](q) \notag\\
		&\quad + \pd_x \big(\sigma S'(q_\ep) \big[\KK,\JJ\big](q)\big)  
			+ \frac{1}{2}  S'(q_\ep)\Big[\big[\KK,\JJ\big],\KK\Big](q). \label{eq:integrand}
\end{align}
We have already established that 
$\big[\KK , \JJ](q) = \tilde{r}_\ep\to 0$ in
$L^2([0,T]; L^2(\R))$ as $\ep\to0$. 

Therefore, we focus on the double commutator, which, for clarity, is
\begin{align}
\Big[\big[\KK,\JJ\big],\KK\Big](q) &=  
	\big[\KK,\JJ\big](\KK q) - \KK\big[\KK,\JJ\big](q)\notag\\
& = 2 \KK \JJ \KK q - \JJ \KK \KK q - \KK \KK \JJ q.\label{eq:commutators}
\end{align}

Term-by-term in this commutator we have
\begin{align}
2 \KK \JJ \KK q (x)&= 
	2  \int_\R \pd_{xx}^2 J_\ep(x - y) \sigma(x)\sigma(y) q(y) \;\d y \label{eq:commute1}\\
& \quad + 2 \int_\R \pd_x J_\ep(x - y) \pd_x \sigma(x) \sigma(y) q(y) \;\d y,\label{eq:commute2}\\[2mm]
\JJ \KK \KK q (x)&= 
	\int_\R \pd_{xx}^2 J_\ep(x - y) \sigma^2(y) q(y) \;\d y\label{eq:commute3}\\
& \quad - \int_\R \pd_x J_\ep(x - y) \sigma(y)\pd_y \sigma(y) q(y)\;\d y ,\label{eq:commute4}\\
\intertext{and}
 \KK \KK \JJ q (x)&=
 	 \int_\R  J_\ep(x - y) \pd_x \big(\sigma(x) \pd_x \sigma(x)\big) q(y) \;\d y\label{eq:commute5}\\
& \quad +  3\int_\R \pd_x J_\ep(x - y) \sigma(x)\pd_x \sigma(x) q(y)\;\d y \label{eq:commute6}\\
& \quad + \int_\R \pd_{xx}^2 J_\ep(x - y) \sigma^2(x)q(y)\;\d y \label{eq:commute7}.
\end{align}

There are more terms here than in 
\cite{MR3803774} because we do not 
necessarily have the divergence-free 
condition $\pd_x \sigma = 0$.

Now we can estimate \eqref{eq:commute1} to
 \eqref{eq:commute7} above by 
considering the sums
\begin{align*}
\mathfrak{I}_1 & := \eqref{eq:commute2} - \eqref{eq:commute4} -\eqref{eq:commute6},\\
\mathfrak{I}_2 & := \eqref{eq:commute1} - \eqref{eq:commute3} - \eqref{eq:commute7},
\end{align*}
and finally the stand-alone
 integral \eqref{eq:commute5},
where from \eqref{eq:commutators}, 
we see that
\begin{align}\label{eq:commutator_decomposition}
\Big[\big[\KK,\JJ\big],\KK\Big](q) 
	= \mathfrak{I}_1 + \mathfrak{I}_2 - \eqref{eq:commute5}.
\end{align}

We shall use \cite[Lemma II.1]{MR1022305}
 to establish that this sum above tends to
 nought in an appropriate topology. 
Estimating these integrals separately, we have
\begin{align*}
\|\mathfrak{I}_1\|_{L^2(\R)}
	&= \bigg\|\int_\R \pd_x J_\ep(\dott - y) \Big(2 \sigma(y) \pd_x \sigma(\dott) 
		+ \sigma(y) \pd_y \sigma(y) 
			- 3 \sigma(\dott) \pd_x \sigma(\dott)\Big) q(y) \;\d y\bigg\|_{L^2(\R)}\\
 & =\bigg\| \int_\R \pd_x J_\ep(\dott - y) \\
 &\quad\times
 	\Big(2 \big(\sigma(y) - \sigma(\dott)\big) \pd_x \sigma(\dott) 
 		+ \big(\sigma(y) \pd_y \sigma(y) 
 			- \sigma(\dott) \pd_x \sigma(\dott)\big)\Big) q(y) \;\d y\bigg\|_{L^2(\R)}\\
& \le 	\bigg\| \int_\R |\pd_x J_\ep(\dott - y)| \\
 &\quad\times
 	\Big(2 |\sigma(y) - \sigma(\dott)|\,  |\pd_x \sigma(\dott)| 
 		+ |\sigma(y) \pd_y \sigma(y) 
 			- \sigma(\dott) \pd_x \sigma(\dott)|\Big) |q(y)| \;\d y\bigg\|_{L^2(\R)}\\					
 & \le C \bigg\| \int_\R  \frac1\ep |\dott - y| J_\ep(\dott - y) \\
 &\quad\times
 	\bigg(2\bigg|\frac{\sigma(y) - \sigma(\dott)}{y - \dott } \pd_x \sigma(\dott)\bigg| 
 		+ \bigg|\frac{\sigma(y) \pd_y \sigma(y) - \sigma(\dott) \pd_x \sigma(\dott)}{y - \dott }
 			\bigg|\bigg) |q(y)| \;\d y\bigg\|_{L^2(\R)}\\
 & \le C\Big(\|\pd_x\sigma\|_{L^\infty(\R)}^2 
 + \|\sigma \pd_x \sigma\|_{L^\infty(\R)} \Big)
 \bigg\| \int_\R  \frac1\ep |\dott - y| J_\ep(\dott - y)|q(y)| \;\d y\bigg\|_{L^2(\R)} \\	
 & \le C\Big(\|\pd_x\sigma\|_{L^\infty(\R)}^2 
 + \|\sigma \pd_x \sigma\|_{L^\infty(\R)} \Big) \big\| \frac1\ep |\dott| J_\ep(\dott) \big\|_{L^1(\R)}
 \big\| q\big\|_{L^2(\R)} 		\\	 
 & \le C \big(\|\pd_x\sigma\|_{L^\infty(\R)}^2  + 
 	\|\sigma \pd_x \sigma\|_{L^\infty(\R)}\big)\|q\|_{L^2(\R)}.
\end{align*}
Here we used that   $|\pd_x J_\ep| \lesssim \ep^{-1}J_\ep$ and Young's inequality for convolutions. 
Similarly  we find that
\begin{align*}
\|\mathfrak{I}_2\|_{L^2(\R)} 
	&= \bigg\|\int_\R \pd_{xx}^2 J_\ep(\dott - y)
		 \big(2 \sigma(\dott) \sigma(y) - \sigma^2(\dott)
		 	 - \sigma^2(y)\big) q(y) \;\d y\|_{L^2(\R)}\\
 & = \bigg\|\int_\R \pd_{xx}^2 J_\ep(\dott - y) 
 	\big(\sigma(\dott) - \sigma(y)\big)^2 q(y) \;\d y\bigg\|_{L^2(\R)}\\
 & \le \bigg\|\int_\R |\pd_{xx}^2 J_\ep(\dott - y)|
		 \big|2 \sigma(\dott) \sigma(y) - \sigma^2(\dott)
		 	 - \sigma^2(y)\big|\, |q(y)| \;\d y\|_{L^2(\R)}\\	
 & \le C \bigg\|\int_\R \frac1{\ep^2} (\dott - y)^2J_\ep(\dott - y)
 	 \bigg|\frac{\sigma(\dott) - \sigma(y)}{y - \dott}
 	 	\bigg|^2 |q(y)| \;\d y\bigg\|_{L^2(\R)}\\	
 & \le C	\|\pd_x \sigma\|_{L^\infty(\R)}^2		
 \bigg\|\int_\R \frac1{\ep^2} (\dott - y)^2J_\ep(\dott - y) |q(y)| \;\d y\bigg\|_{L^2(\R)}\\	
  & \le C\|\pd_x \sigma\|_{L^\infty(\R)}^2	
  \big\| \frac1{\ep^2} (\dott)^2J_\ep(\dott)\big\|_{L^1(\R)} \big\| q\big\|_{L^2(\R)} \\
  & \le C \|\pd_x \sigma\|_{L^\infty(\R)}^2\|q\|_{L^2(\R)}.
\end{align*}
We also have
$$
\|\eqref{eq:commute5}\|_{L^2(\R)}
	 \le C \|J_\ep\|_{L^1(\R)}
	 	\|\pd_x(\sigma\pd_x \sigma)\|_{L^\infty(\R)}\|q(t)\|_{L^2(\R)}.
$$

Now for smooth functions $q$,
\begin{align*}
\mathfrak{I}_2 & =\int_\R \pd_{xx}^2 J_\ep(x - y)
					 \big(2 \sigma(x) \sigma(y) - \sigma^2(x)
		 			 - \sigma^2(y)\big) q(y) \;\d y\\
		 	 & =- 2\int_\R \pd_{xx}^2 J_\ep(x - y) \frac{(x - y)^2}{2}
					 \left(\frac{\sigma(y) - \sigma(x)}{(y - x)} \right)^2q(y) \;\d y\\
 			 & = - 2 (\pd_x \sigma)^2 q(x) \int_\R \frac{z^2}{2}\pd_{zz}^2 J_\ep(z) \; \d z + o(\ep).
\end{align*}
A similar calculation can be done for $\mathfrak{I}_1$, 
where there is only one derivative on the mollifier, 
and which can be found directly in the proof of 
\cite[Lemma II.1]{MR1022305}.

The limit of \eqref{eq:commute5} as $\ep \to 0$ for smooth $q$ is standard.

Reasoning then as in the proof of 
\cite[Lemma II.1]{MR1022305}, we find that
$$
\mathfrak{I}_1 \to  \pd_x (\sigma \pd_x \sigma)q+ 2 (\pd_x \sigma)^2 q, 
	\qquad \mathfrak{I}_2  \to - 2 (\pd_x \sigma)^2 q, 
		\qquad -\eqref{eq:commute5}\to - \pd_x(\sigma \pd_x \sigma) q,
$$
in $L^2(\R)$ almost everywhere in time,
 $\mathbb{P}$-almost surely as $\ep \to 0$.  Adding these together, with 
reference to \eqref{eq:commutator_decomposition},
 we can conclude that   $\Big[\big[\KK,\JJ\big],\KK\Big](q) 
\to 0$ in 
$L^2([0,T]; L^2(\R))$
 $\mathbb{P}$-almost surely as $\ep \to 0$. 

Recall \eqref{eq:integrand}. We have the $\mathbb{P}$-almost sure bound,
\begin{align}
&\bigg|\int_0^T \int_{\R} \Big[-S'(q_\ep) \rho_\ep 
	+\frac{1}{2} S''(q_\ep) \Big( (\pd_x(\sigma q_\ep))^2 
		-\big( \pd_x(\sigma q)\star J_\ep\big)^2 \Big)\Big] \;\d x\,\d t\bigg| \notag\\
&\qquad\qquad= 		\bigg|\int_0^T \int_{\R} \Big(-S'(q_\ep) \cdot \eqref{eq:rho_ep_1} 
	+ \frac{1}{2} S''(q_\ep) \cdot \eqref{eq:rho_ep_extra} \Big) \;\d x\,\d t\bigg|  \notag\\
&\qquad\qquad			 \le C_{T, \sigma, \ep} \|q\|_{L^2([0,T]\times\R)}
\label{eq:rho_ep_estimate},
\end{align}
where $C_{T,\sigma, \ep} \to 0$ as $\ep \to 0$.

\end{proof}

Next we prove Proposition \ref{thm:aprioriestimates}:

\begin{proof}[Proof of  Prop. \ref{thm:aprioriestimates}.]

We carry out this proof in three steps:
\begin{itemize}
\item[(1)] We first renormalise the mollified equation, 
	finding an equation for $S(q_\ep)$ with $S\in C^{1,1}$.
\item[(2)] Using the renormalisation in (1) 
	prove the explicit $L^2$-bound \eqref{eq:q_L2bound}.
\item[(3)] Exploiting the explicit $L^2$-bound, 
	we demonstrate the $L^{2 + \alpha}$-bound \eqref{eq:q_Lalphabound}.
\end{itemize}

\smallskip
{\it 1. Renormalisation.}

Since convolution commutes with differentiation in $x$, 
\begin{equation*}
\pd_x u_\ep = q_\ep.
\end{equation*}

For any non-negative $S \in C^2(\R)$, 
we can use It\^{o}'s formula to write 
\begin{align*}
0 &= \d q_\ep + \big(\pd_x(u_\ep q_\ep) 
	- \frac{1}{2} q_\ep^2\big) \;\d t 
		+ \pd_x(\sigma q_\ep)\;\d W\\
 & \quad - \frac{1}{2} \pd_x\big(\sigma \pd_x (\sigma q_\ep)\big)\;\d t
 	- (r_\ep + \rho_\ep)\;\d t - \tilde{r}_\ep \;\d W.
\end{align*}
Furthermore, 
\begin{align*}
0 &= \d S(q_\ep) + S'(q_\ep) \bigg(  \pd_x(u_\ep q_\ep) 
	- \frac{1}{2} q_\ep^2 -
		 \frac{1}{2} \pd_x(\sigma \pd_x (\sigma q_\ep))\bigg)\;\d t 
		 	+  S'(q_\ep) \pd_x(\sigma q_\ep)\;\d W\\
&\quad - \frac{1}{2} S''(q_\ep) \big( \pd_x(\sigma q_\ep)
	 - \tilde{r}_\ep \big)^2\;\d t 
	 	- S'(q_\ep) ( r_\ep \;\d t+ \rho_\ep  \;\d t + \tilde{r}_\ep \;\d W )\\
&= \mathfrak{L}- \frac{1}{2} S''(q_\ep) \big( \pd_x(\sigma q_\ep)
	 - \tilde{r}_\ep \big)^2\;\d t 
	 	- S'(q_\ep) ( r_\ep \;\d t+ \rho_\ep  \;\d t + \tilde{r}_\ep \;\d W ).
\end{align*}
For the first term $\mathfrak{L}$ we find
\begin{align*}
\mathfrak{L}&=\d S(q_\ep)  + S'(q_\ep)\big(  \px (u_\ep q_\ep) 
	 - \frac{1}{2}q_\ep^2\big) \,\d t 
	 	+ S'(q_\ep) \px (\sigma q_\eps) \,\d W 
		  -  \frac12 S'(q_\ep) \pd_x (\sigma \pd_x (\sigma q_\ep))\,\d t\\
&=  \d S(q_\ep)  + 
	  \big(\pd_x (u_\ep S(q_\ep)) - q_\ep S(q_\ep) 
	  	+ \frac{1}{2}S'(q_\ep)q_\ep^2\big) \,\d t
		-  \frac12 S'(q_\ep) \pd_x (\sigma \pd_x (\sigma q_\ep))\,\d t  \\
&\quad	 + \pd_x(\sigma S(q_\ep))\,\d W  + \pd_x \sigma\big (q_\ep S'(q_\ep) - S(q_\ep)\big)\,\d W,
\end{align*}
and the last term on the first  line can be further expanded in order 
to maximise the number of terms in divergence form:
\begin{align*}
\mathfrak{L}&=  \d S(q_\ep)  +   \big(\pd_x (u_\ep S(q_\ep) )
	- q_\ep S(q_\ep) + \frac{1}{2}S'(q_\ep)q_\ep^2\big) \,\d t\\
	&\quad+ \pd_x(\sigma S(q_\ep))\,\d W 
		+ \pd_x \sigma\big (q_\ep S'(q_\ep) - S(q_\ep)\big)\,\d W \\
		& \quad  -  \frac12  \pd^2_{xx} (S(q_\ep) \sigma^2)\,\d t 
		+ \frac{1}{4} \pd_x (S(q_\ep) \pd_x \sigma^2) \,\d t
			+ \frac{1}{4} \pd_x \big( (S(q_\ep)  - S'(q_\ep) q_\ep)\pd_x \sigma^2\big)\,\d t\\
			&\quad+ \frac{1}{2} S''(q_\ep) ( \sigma \pd_x q_\ep)^2 \,\d t
			+   \frac{1}{4} S''(q_\ep) \pd_x\sigma^2 q_\ep \pd_x q_\ep\,\d t.
\end{align*}
Re-arranging the terms, one arrives at:
\begin{align*}
\mathfrak{L}&= \d S(q_\ep) +  \pd_x \Big(u_\ep S(q_\ep) 
	+ \frac{1}{4}\pd_x \sigma^2 S(q_\ep)\Big) \,\d t 
		+ \pd_x(\sigma S(q_\ep))\,\d W \\
			& \quad+ \pd_x \sigma \big(q_\ep S'(q_\ep) - S(q_\ep)\big)\,\d W 
			- \frac{1}{2} \pd_{xx}^2\big(\sigma^2 S(q_\ep)\big) \,\d t
				- \Big( q_\ep S(q_\ep) - \frac{1}{2}S'(q_\ep)q_\ep^2\Big) \,\d t \\
					&\quad- \frac{1}{4} \pd_{xx}^2 \sigma^2 \big(q_\ep S'(q_\ep) - S(q_\ep)\big) \,\d t
					- \frac{1}{2} S''(q_\ep) \big((\pd_x(\sigma q_\ep))^2\\
						&\quad - (\sigma \pd_x q_\ep)^2\big) \,\d t 
						 	+  \frac{1}{2} S''(q_\ep) (\pd_x(\sigma q_\ep))^2 \;\d t.
\end{align*}
Introducing $G_S(v) = vS'(v) - S(v)$, 
we can simplify the above as:
\begin{align*}
\mathfrak{L}&=\d S(q_\ep) + \pd_x\bigg(u_\ep S(q_\ep) +
	 \frac{1}{4}\pd_x \sigma^2 S(q_\ep) 
	- \frac{1}{2} \pd_x \sigma^2 G_S(q_\ep)\bigg) \,\d t 
		- \big(q_\ep S(q_\ep) - \frac{1}{2}q^2_\ep S'(q_\ep)\big)\,\d t\\
		&\quad - \bigg[\frac{1}{2}\pd_{xx}^2 \big(\sigma^2 S(q_\ep)\big) 
			+ \frac{1}{2}q_\ep G_S'(q_\ep) (\pd_x \sigma)^2 
				- \frac{1}{4} \pd_{xx}^2 \sigma^2 G_S(q_\ep)\bigg]\,\d t\\
				&\quad  + \pd_x (\sigma S(q_\ep))\,\d W
					+ \pd_x \sigma G_S(q_\ep) \,\d W 
						+  \frac{1}{2} S''(q_\ep) (\pd_x(\sigma q_\ep))^2 \;\d t.
\end{align*}
There is no pathwise energy estimate in the 
stochastic setting because of the term 
$\pd_x \sigma G_S(q_\ep) \,\d W$, which is not 
an exact spatial derivative.

Putting back in $r_\ep$, $\rho_\ep$, and $\tilde{r}_\ep$,
 we arrive at
\begin{equation}\label{eq:entropy}
\begin{aligned}
0&= \d S(q_\ep) + \pd_x\bigg(u_\ep S(q_\ep) 
	+ \frac{1}{4}\pd_x \sigma^2 S(q_\ep) 
		- \frac{1}{2} \pd_x \sigma^2 G_S(q_\ep)\bigg) \,\d t 
			- \big(q_\ep S(q_\ep) - \frac{1}{2}q^2_\ep S'(q_\ep)\big)\,\d t\\
			& \quad- \bigg[\frac{1}{2}\pd_{xx}^2 \big(\sigma^2 S(q_\ep)\big) 
				+ \frac{1}{2}q_\ep G_S'(q_\ep) (\pd_x \sigma)^2 
					- \frac{1}{4} \pd_{xx}^2 \sigma^2 G_S(q_\ep)\bigg]\,\d t\\
					& \quad + \pd_x (\sigma S(q_\ep))\,\d W
						+ \pd_x \sigma G_S(q_\ep) \,\d W \\
					& \quad  + \frac{1}{2} S''(q_\ep) \Big( (\pd_x(\sigma q_\ep))^2 
						-\big( \pd_x(\sigma q)\star J_\ep\big)^2 \Big)\; \d t
							- S'(q_\ep) ( r_\ep \;\d t+ \rho_\ep  \;\d t + \tilde{r}_\ep \;\d W ),
\end{aligned}
\end{equation}
where we have used
\begin{equation*}
 \pd_x(\sigma q)\star J_\ep 
 	= \pd_x (\sigma q_\ep) - \tilde{r}_\ep.
\end{equation*}

This puts most terms of the equation in divergence 
form and also sets up the mollification term ready 
for  an application of Lemma \ref{thm:mollifying_errors}.

\smallskip
{\it 2. The $L^{2}$-bound.} 

The $L^2$-bound follows directly from the requirement 
\eqref{eq:sol_cons} of Definition \ref{def:cons_sols} 
for conservative weak solutions.

We show that the weak energy balance 
\eqref{eq:sol_dissp} holds for weak dissipative solutions, 
from which shall follow the $L^2$-bound \eqref{eq:q_L2bound}.

 We can estimate $\|q(t)\|_{L^{2}_x}^{2}$ using the entropies:
\begin{equation*}
S(v) = S_\ell(v) := \begin{cases} v^2, & |v| \le \ell,  \\
2\ell |v| -  \ell^2,  & |v| > \ell. 
\end{cases}
\end{equation*}
This ensures that $S_\ell$ has bounded first and 
second derivatives for $\ell < \infty$, and allows 
us to exploit the convergences in $\ep \to 0$ of 
$r_\ep$, $\rho_\ep$, and $\tilde{r}_\ep$ proven 
in Lemma \ref{thm:mollifying_errors}. In particular,
\begin{equation*}
S_\ell'(v) = \begin{cases}2v, & |v|\le \ell, \\ 
2 \ell \,\sgn(v), & |v| > \ell, 
\end{cases}
\qquad \quad S''(v) = 2\mathds{1}_{\{|v| \le \ell\}}.
\end{equation*}
Furthermore, we have
\begin{equation*}
G_S(v) = vS_\ell'(v) - S_\ell(v) = v^2 \wedge \ell^2, 
\qquad G_S'(v) = vS_\ell''(v) = 2v \mathds{1}_{\{|v| \le \ell\}}.
\end{equation*}
and
\begin{equation*}
q_\ep S_\ell(q_\ep)- \frac{1}{2}q^2_\ep S'(q_\ep) 
	= \ell \,q_\ep(|q_\ep| - \ell) \mathds{1}_{\{|q_\ep| > \ell\}}.
\end{equation*}

Inserting these into \eqref{eq:entropy} and 
integrating in $x$ and $s$, we are left with
\begin{equation}\label{eq:pre_L2bound}
\begin{aligned}
0 &= \int_\R S(q_\ep)\;\d x\bigg|_0^t   
	- \int_0^t \int_\R  \ell \,q_\ep(|q_\ep| - \ell) \mathds{1}_{\{|q_\ep| > \ell\}} \;\d x\,\d s 
		+ \int_0^t \int_\R \pd_x \sigma G_S(q_\ep) \;\d x\,\d W\\
		& \quad- \int_0^t \int_\R q_\ep^2 \big((\pd_x \sigma)^2 
			- \frac{1}{4} \pd_{xx}^2 \sigma^2\big)\mathds{1}_{\{|q_\ep| \le \ell\}}\;\d x\,\d s 
				+  \frac{1}{4}\int_0^t \int_\R\ell^2\pd_{xx}^2 \sigma^2 \mathds{1}_{\{|q_\ep| > \ell\}} \;\d x\,\d s \\
				& \quad + \int_0^t \int _\R \mathds{1}_{\{|q_\ep| \le \ell\}}\Big( (\pd_x(\sigma q_\ep))^2 
						-\big( \pd_x(\sigma q)\star J_\ep\big)^2 \Big)\;\d x \,\d s 
							-\int_0^ t \int_\R S_\ell'(q_\ep)  \rho_\ep  \;\d x \,\d s \\
							& \quad - \int_\R \int_0^ tS_\ell'(q_\ep)
								 ( r_\ep  \;\d t+ \tilde{r}_\ep \;\d W )\;\d x.
\end{aligned}
\end{equation}

We provide further bounds for the terms
\begin{align}\label{eq:special1}
\int_0^t \int_\R  \ell \,q_\ep(|q_\ep| - \ell) \mathds{1}_{\{|q_\ep| > \ell\}} \;\d x\,\d s,
 \end{align}
and
\begin{align}\label{eq:special2}
  \frac{1}{4}\int_0^t \int_\R\ell^2\pd_{xx}^2 \sigma^2 \mathds{1}_{\{|q_\ep| > \ell\}} \;\d x\,\d s,
\end{align}
which cannot  immediately be dealt with by Gronwall's inequality.

By splitting $q_\ep$ into positive and negative 
parts of essentially disjoint support, i.e., 
$q_\ep = q^+_\ep + q^-_\ep$ so that 
$q^-_\ep \le 0 \le q^+_\ep$, we see that
\begin{align*}
   \eqref{eq:special1} 
   	&=  \int_0^t \int_\R \ell (q^+_\ep + q^-_\ep) (|q_\ep| - \ell)
   			\mathds{1}_{\{|q_\ep| > \ell\}}\;\d x\,\d s 
   		\le \int_0^t \int_\R \ell q^+_\ep (|q_\ep| - \ell)
   			 \mathds{1}_{\{|q_\ep| > \ell\}}\;\d x\,\d s.
\end{align*}

We shall be taking the limits in the order 
$\ep \to 0$ first and then $\ell \to \infty$ later. 
Using the upper-boundedness of weak dissipative 
equations mandated in Definition \ref{def:weak_diss_sols}, 
we can can take $\ep \to 0$ and conclude that there 
is always a sufficiently large $\ell$ beyond which the 
term $(|q_\ep| - \ell) \mathds{1}_{\{|q_\ep| > \ell\}}$ 
simply vanishes.

Secondly, by Markov's inequality,
\begin{equation*}
 |\eqref{eq:special2}| \le 
 	\frac{\|\pd_{xx}^2 \sigma ^2\|_{L^\infty}}{4}
 		 \int_0^t  \int_\R |q_\ep|^2\mathds{1}_{\{|q_\ep| > \ell\}}\;\d x\,\d s.
\end{equation*}
Finally, by Lemma \ref{thm:mollifying_errors}, 
equations \eqref{eq:r_ep_estimate}, 
\eqref{eq:tilder_ep_estimate}, and \eqref{eq:rho_ep_estimate}, 
the last two lines of \eqref{eq:pre_L2bound} are 
bounded by $C_{\ep, T}$, where $C_{\ep,T} \to 0$ as $\ep \to 0$.
This means all terms can either be handled by 
Gronwall's inequality or are bounded. 
First integrating against $\d \mathbb{P}$, 
we then take the limits $\ep \to 0$ and 
$\ell \to \infty$ and use Fatou's lemma in order 
to get the limit energy inequality
\begin{align}
\Ex \int_\R |q|^2 \;\d x\bigg|_0^t  
	&\le  \Ex \int_0^t  \int_\R q^2 \Big( (\pd_x \sigma)^2
		- \frac{1}{4} \pd_{xx}^2\sigma^2 \Big) \;\d x\,\d s .
\end{align}
for almost every $t \in [0,T]$.

\smallskip
{\it 3. The $L^{2+\alpha}$-bound.}

For the $L^{2 + \alpha}$-bound, with $\alpha\in[0,1)$, we use the 
entropies $S_\ell$ defined by
\begin{equation*}
S(v) = S_\ell(v) :=  \begin{cases}  
\frac12\alpha\ell^{2 - \alpha} v^3 +\frac12(2 - \alpha)\ell^{-\alpha} v, 
& |v| \le \ell^{-1}, \\  
v|v|^\alpha, &    \ell^{-1} < |v| \le \ell,\\ 
(1 + \alpha) v\ell^\alpha - \alpha \ell^{1 + \alpha}\sgn(v), & |v| > \ell. 
\end{cases}
\end{equation*}

In this way,
\begin{align*}
S_\ell'(v) &= \begin{cases} \frac32\alpha \ell^{2 - \alpha}v^2 + \frac12(2 - \alpha) \ell^{-\alpha}, 
& |v| \le \ell^{-1},\\  
(1 + \alpha) |v|^{\alpha}, &  \ell^{-1} < |v| \le \ell,\\ 
(1 + \alpha) \ell^\alpha, & |v| > \ell,  \end{cases} \\
\intertext{and} 
S''_\ell(v) &=  \begin{cases} 3\alpha \ell^{2 - \alpha} v, & |v| \le \ell^{-1},\\ 
\alpha (1 + \alpha) |v|^{\alpha - 1} \sgn(v), &\ell^{-1} < |v| \le \ell,\\ 
0,  &|v| > \ell.\end{cases}
\end{align*}
The values for $S_\ell(v)$ in the interval 
$[-\ell^{-1},\ell^{-1}]$ are the Hermite interpolation 
polynomial, matching the values and first derivatives 
of $v|v|^\alpha$ at the end-points 
$v = \pm \ell^{-1}$, so that $S_\ell'$ and $S_\ell''$ 
stay bounded for fixed $\ell$, as we require them to do.

Using these to compute $G_S(v):= vS_\ell'(v) - S_\ell(v)$
 and its derivatives, we find
\begin{equation*}
G_S(v) =  \begin{cases} \alpha \ell^{2 - \alpha} v^3, & |v| \le \ell^{-1},\\
 \alpha v |v|^{\alpha }, & \ell^{-1} < |v| \le \ell, \\ 
 \alpha \ell^{1 + \alpha}\sgn(v), & |v| > \ell, \end{cases}
\end{equation*}
and
\begin{align}\label{eq:dGS_estim}
 G'_S(v) := vS''_\ell(v) &= \begin{cases} 3\alpha \ell^{2 - \alpha} v^2, & |v| \le \ell^{-1},\\
  \alpha (1 + \alpha) |v|^{\alpha}, & \ell^{-1} < |v| \le \ell,\\ 
  0, & |v| > \ell, \end{cases} \\
  & \le 3\alpha |v|^\alpha. \notag
\end{align}

Moreover,
\begin{equation*}
v S_\ell(v) - \frac{1}{2}v^2 S_\ell'(v)
	=  \begin{cases}
	 -\frac14\alpha \ell^{2 - \alpha} {v}^4 + \frac14(2 - \alpha) \ell^{-\alpha} v^2, & |v| \le \ell^{-1},\\  
	\frac12(1 - \alpha) |v|^{\alpha+2}, & \ell^{-1}< |v| \le  \ell,\\ 
	\frac12(1 + \alpha) v^2 \ell^\alpha - \alpha \ell^{1 + \alpha}|v|, & |v| > \ell. \end{cases}
\end{equation*}
Clearly, $S_\ell(v)\to v|v|^\alpha$ and $G_S(v)\to  \alpha v |v|^\alpha$ as $\ell\to \infty$.

We can re-arrange \eqref{eq:entropy}, 
and integrate in $x$ and $s$ to get
\begin{align*}
 &\Ex \int_0^t \int_\R \big(q_\ep S(q_\ep) 
 	- \frac{1}{2}q^2_\ep S'(q_\ep)\big)\;\d x\,\d s\\
&\qquad \le \Ex \int_\R S(q_\ep)\;\d x\bigg|_0^t  
	- \Ex \int_0^t \int_\R\bigg[ \frac{1}{2}q_\ep G_S'(q_\ep) (\pd_x \sigma)^2 
		- \frac{1}{4} \pd_{xx}^2 \sigma^2 G_S(q_\ep)\bigg]\,\d x\,\d s\\
		 & \qquad \quad + \frac{1}{2} \int_0^t \int_\R S''(q_\ep)\Big( (\pd_x(\sigma q_\ep))^2 
 			-\big( \pd_x(\sigma q)\star J_\ep\big)^2 \Big)\;\d x \,\d s 
 				- \int_0^t\int_\R  S'(q_\ep)   \rho_\ep  \;\d x\,\d s \\
 				& \qquad \quad -  \int_\R \int_0^t S'(q_\ep) ( r_\ep \;\d s + \tilde{r}_\ep \;\d W )\;\d x,
\end{align*}
and insert the definitions of $S_\ell$ 
and $G_S$, and their derivatives.

By inspection, $S_\ell'$ and $S_\ell''$ are 
uniformly bounded on $\R$ for fixed $\ell$, 
so again by Lemma~\ref{thm:mollifying_errors}, 
and Eqs. \eqref{eq:r_ep_estimate}, \eqref{eq:tilder_ep_estimate}, 
and \eqref{eq:rho_ep_estimate}, the last two lines 
of \eqref{eq:pre_L2bound} are bounded by $C_{\ep, T}$, 
where $C_{\ep,T} \to 0$ as $\ep \to 0$.
We then take the limits $\ep \to 0$ and 
$\ell \to \infty$ and use Fatou's lemma in order 
to get the limit energy inequality
\begin{align}
\frac{1}{2}(1 - \alpha) \Ex \int_0^t \int_\R |q|^{2 + \alpha}\;\d x\,\d s
	&\le \Ex\int q|q|^\alpha \;\d x\bigg|_0^t- 
		\frac{1}{2}\alpha (\alpha + 1)\, \Ex \int_0^t \int_\R q|q|^\alpha (\pd_x \sigma )^2\;\d x \,\d s\notag\\
		&\quad + \frac{\alpha}{4}\, \Ex \int_0^t \int_\R \pd_{xx}^2 \sigma^2 q|q|^\alpha\;\d x\,\d s.
\end{align}
for almost every $t \in [0,T]$.
From the second part of this proof, we have the inclusions
$q\in L^2(\Omega\times [0,T] \times \R)$ and
$q\in L^\infty([0,T]; L^2(\Omega\times \R))$. 
This allows us to interpolate between 
$L^2_{t,x}$ and $L^1_{t,x}$ or 
between $L^2_x$ and $L^1_x$ to bound
the integrals on the right, thereby 
allowing us to control 
$\Ex\|q\|_{L^{2 + \alpha}_{x,t}}^{2 + \alpha}$ 
as well.

This concludes the proof.
\end{proof}

\section{The Defect Measure in the Deterministic Setting}\label{sec:appendixB}

Here we construct explicit and easily verifiable 
solutions in the manner of \cite{MR2182833} to a 
problem with step functions as the initial distribution, 
and show explicitly how blow-up and a defect 
measure recording that blow-up, arise.
This is to complement the discussion on
the defect measure in Section \ref{sec:deterministic_background}.

Let $[a,b]$ be evenly partitioned into $n$ intervals, with endpoints $x_i =a +  i(b - a)/n$ for $i = 0,\ldots, n$.

First we approximate $q_0$ by defining
\begin{align}
V^n_{0,i} = \fint_{x_{i - 1}}^{x_{i}} q_0(x) \;\d x,\label{eq:defin_V0}
\end{align}
and setting
\begin{equation*}
q_0^n(x) = \sum_{i = 1}^{n} V^n_{0,i} \mathds{1}_{(x_{i - 1},x_i)}, \qquad q_0^n(b) = V^n_{0,n}.
\end{equation*}

Next we postulate the following characteristics:
\begin{align}\label{eq:Xn_defin}
X^n_i(t) = a + \frac{(b - a)}{4n}\sum_{j = 1}^i 
(2 + V^n_{0,j}t)^2\mathds{1}_{\{t \ge 0  :  2 + V^n_{0,j} t  >  0\}}.
\end{align}

Notice that two characteristics $X_{i - 1}^n$ and 
$X^n_i$ coincide and remain coincident after 
$t = -2/V^n_{0,i}$ if $V^n_{0,i} < 0$.

Setting
\begin{align}
q^n(t,x) &= \sum_{i = 1}^n\frac{2 V^n_{0,i}}{2 + V^n_{0,i}t}\mathds{1}_{\{X^n_{i - 1}(t) < x < X^n_i(t),\, 2 + V^n_{0,i} t > 0\}},\label{eq:qn_defin}\\
u^n(t,x) &= \int_{-\infty}^x q^n(t,y)\;\d y  =\int_{X^n_0(t)}^x q^n(t,y)\;\d y \label{eq:un_defin},
\end{align}
we have by direct substitution of \eqref{eq:un_defin} in \eqref{eq:Xn_defin},
\begin{align*}
\d X^n_i(t) &= u^n(t,X^n_i(t))\;\d t + \sigma \d W.
\end{align*}

For simplicity we set $q^n(t,X^n_i(t)) = 0$ on 
the (finitely many) characteristics $X^n_i$, 
thereby defining $q^n(t)$ pointwise, and so that 
from the definition, if and when two characteristics 
eventually meet, there is no mass concentrated 
along their coincident path. This is the defining 
feature of a dissipative solution --- that $L^2_x$-mass 
is completely and eternally annihilated at wave-breaking
 --- on which we shall expound further below.

From the definition of $q^n$ in \eqref{eq:qn_defin}, we have 
\begin{align}\label{eq:approx_eq}
\pd_tq^n + u^n \pd_x(q^n) = -\frac{1}{2}(q^n)^2
\end{align}
in the sense of distributions --- to wit, from  \eqref{eq:qn_defin}:
\begin{align}
\pd_t(q^n)(t,x) 
	&= \sum_{i = 1}^{n}\bigg[\frac{-2(V^n_{0,i})^2}{(2 + V^n_{0,i}t)^2}
		\mathds{1}_{\{X^n_{i - 1}< x < X^n_i,\, 2 + V^n_{0,i}t > 0\}} \notag \\
		&\quad
			- \frac{2 V^n_{0,i}}{2 + V^n_{0,i}t}\mathds{1}_{\{X^n_{i - 1}(t) < x < X^n_i(t)\}} 
				\delta(2 + V^n_{0,i}t)\label{eq:q_det_box_sq_0} \\
&\quad - \frac{2 V^n_{0,i}}{2 + V^n_{0,i}t}\mathds{1}_{\{2 + V^n_{0,i}t > 0\}} 
	\big(\delta(x - X^n_{i - 1})u^n(t,X^n_{i - 1})  
		- \delta(x - X^n_{i}) u^n(t,X^n_{i})\big)\bigg],\notag\\[2mm]
\pd_x(q^n) (t,x)&= \sum_{i = 1}^{n}\frac{2V^n_{0,i}}{2 + V^n_{0,i}t} 
	\big(\delta(x - X^n_{i - 1}) - \delta(x - X^n_{i })\big)
		\mathds{1}_{\{2 + V^n_{0,i}t > 0\}},\label{eq:q_det_box_sq_1}\\[2mm]
(q^n)^2(t,x)&=  \sum_{i = 1}^{n}\bigg|\frac{2V^n_{0,i}}{2 + V^n_{0,i}t }\bigg|^2
	\mathds{1}_{\{X^n_{i - 1}< x < X^n_i,\,2 + V^n_{0,i}t > 0\}}.
		\label{eq:q_det_box_sq}
\end{align}
The quantity
\begin{equation*}
\frac{2 V^n_{0,i}}{2 + V^n_{0,i}t}
	\mathds{1}_{\{X^n_{i - 1}(t) < x < X^n_i(t)\}} \delta(2 + V^n_{0,i}t) 
\end{equation*}
in the equation for $\pd_t (q^n)(t,x)$ means
\begin{align}\label{eq:notrig_marole}
\lim_{\ep \to 0}\frac{2 V^n_{0,i}}{2 + V^n_{0,i}t}
	\mathds{1}_{\{X^n_{i - 1}(t) < x < X^n_i(t)\}}
		\frac{1}{\ep}\eta\big(\frac1\ep(t + 2/V^n_{0,i} +  \ep)\big),
\end{align}
where $\eta$ is a symmetric smooth bump of unit $L^1$-mass, supported 
on $[0,1]$. The limit is taken in the topology of 
distributions on $[0,T]\times \mathbb{R}$. We can 
interpret the expression thus, as differentiation is 
continuous in the topology of distributions. 
The limit evaluates to nought in the sense of 
distributions because $X^n_i(t)-X^n_{i - 1}(t)$ 
is proportional to $(2 + V^n_{0,i}t)^2$.  
Nevertheless a similar term is enormously important 
in the equation for $\pd_t (q^n)^2$ because  
dissipation arises from this term, 
which characterises dissipative solutions. 

From the expression for the difference 
$X^n_{i}(t) - X^n_{i - 1}(t)$ in \eqref{eq:Xn_defin}, 
and as mentioned there, we see that the difference 
is zero for $2 + V^n_{0,i} t \le 0$. 
Therefore by the expression for $(q^n)^2$, 
\eqref{eq:q_det_box_sq}, we can compute 
that, $\mathbb{P}$-almost surely,
\begin{align}\label{eq:qn_energy}
\int_{\mathbb{R}}(q^n)^2(t,x)\;\d x 
	= \frac{b - a}{n}\sum_{i = 1}^n \mathds{1}_{\{t \ge 0  :  2 + V^n_{0,i}t > 0\}} (V^n_{0,i})^2
		  \le \frac{b - a}{n}\sum_{i = 1}^n(V^n_{0,i})^2 
		  	= \int_{\mathbb{R}} (q^n_0(x))^2\;\d x,
\end{align}
a constant.

We can record the dissipation of 
$\|q^n(t)\|_{L^2(\R)}^2$ as a defect measure:
\begin{align}\label{eq:defect_meas}
\mathrm{m}^n(\d t, \d x) 
&= \sum_{i = 0}^n \frac{b - a}{n}(V^n_{0,i})^2 
	\delta(x - X^n_{i}(t)) \delta(t + V^n_{0,i}/2)\;\d x\,\d t.
\end{align}

From this measure we see that dissipation gives 
rise to the admissibility condition in \cite[Definition 2.2]{MR1361013},
\begin{equation*}
\pd_t (q^n)^2 + \pd_x (u^n (q^n)^2)  =- \frac{\mathrm{m}^n(\d t, \d x)}{\d t\,\d x} \le 0.
\end{equation*}

We carry out this computation explicitly below:
\begin{align*}
\pd_t ( q^n)^2
	&= \sum_{i = 1}^n\Big[ \,\big|\frac{2V^n_{0,i}}{2 + V^n_{0,i}t }\big|^2
		 \big(\delta(x - X^n_{i})  u(X^n_i)  -
		 	 \delta(x - X^n_{i - 1})  u(X^n_{i - 1}) \big) 
		 	 	\mathds{1}_{\{2 + V^n_{0,i}t > 0\}}\\
				&\qquad\qquad  - \bigg(\frac{2V^n_{0,i}}{2 + V^n_{0,i}t }\bigg)^3
	 			\mathds{1}_{\{X^n_{i - 1} < x < X^n_i,\, 2 + V^n_{0,i}t > 0\}}\\
	 			& \qquad\qquad - \frac{(2 V^n_{0,i})^2}{(2 + V^n_{0,i}t)^2}
	 				\mathds{1}_{\{X^n_{i - 1}(t) < x < X^n_i(t)\}} \delta(2 + V^n_{0,i}t)\Big], \\[2mm]
\pd_x(u^n (q^n)^2)
&= u^n \pd_x (q^n)^2 + (q^n)^3\\
&= u^n(x) \sum_{i = 1}^n \Big[\,\big|\frac{2V^n_{0,i}}{2 + V^n_{0,i}t }\big|^2 
	\big(\delta(x - X^n_{i - 1}) - \delta(x - X^n_{i})\big) \\
	&\qquad\qquad + \bigg(\frac{2V^n_{0,i}}{2 + V^n_{0,i}t }\bigg)^3 
		\mathds{1}_{\{X^n_{i - 1} < x < X^n_i,\,2 + V^n_{0,i}t > 0\}}\Big].
\end{align*}
Therefore again with due consideration for the 
difference $X_{i}^n(t) - X^n_{i - 1}(t) = (2 + V^n_{i,0}t)^2  (b - a)/4n$,
\begin{align*}
\pd_t (q^n)^2 + \pd_x (u^n (q^n)^2) 
= - \sum_{i =1}^n \frac{b - a}{n}(V^n_{0,i})^2
	\delta(x - X^n_i(t))\delta(2 + V^n_{0,i}t),
\end{align*}
where we understand the expression
\begin{equation*}
\frac{(2 V^n_{0,i})^2}{(2 + V^n_{0,i}t)^2}
	\mathds{1}_{\{X^n_{i - 1}(t) < x < X^n_i(t)\}} 
		\delta(2 + V^n_{0,i}t) 
\end{equation*}
as in \eqref{eq:notrig_marole} above.

The times at which ($L^2_x$-)mass is released 
from this defect measure and returned to the solution, 
with a necessary corresponding determination of how 
characteristics $X^n_i(t)$ are to be continued past 
$\{t :  2 + V^n_{0,i}t > 0\}$, determines the 
types of solution one seeks. When it is never returned 
(when the indicator function in \eqref{eq:defect_meas} 
attains unity for all $t$ sufficiently great), the 
solutions are ``dissipative''; when the measure 
only retains mass instantaneously, as in 
\cite{MR2653980} in for the similar Camassa--Holm equation, 
solutions are ``conservative''. There are 
uncountably many possibilities between these extremes. 

\bigskip
\bigskip

\bibliographystyle{plain}

\end{document}